\documentclass[12pt,reqno]{amsart}
\usepackage[margin=1in]{geometry}
\makeatletter
\def\section{\@startsection{section}{1}%
\z@{.7\linespacing\@plus\linespacing}{.5\linespacing}%
{\bfseries\normalfont\scshape
	\centering
}}
\def\@secnumfont{\bfseries}
\makeatother
\usepackage{enumitem}
\usepackage{hyperref}
\usepackage{mathtools}
\usepackage{amsmath,amsfonts,amsthm,txfonts}
\usepackage{dsfont}
\usepackage{mathrsfs}
\usepackage{enumitem}
\usepackage{empheq}
\usepackage{amssymb,amsmath}
\usepackage{graphicx}
\usepackage{breqn}
\usepackage{chngcntr}
\counterwithout{equation}{section}
\usepackage{accents}
\usepackage{trimclip}
\newtheorem{Lem}{Lemma}[section]
\newtheorem{Pro}{Proposition}[section]
\newtheorem{Rem}{Remark}[section]
\newtheorem{Ass}{Assumptions}[section]
\newtheorem{Def}{Definition}[section]
\newtheorem{Thm}{Theorem}[section]

\usepackage{graphicx,amssymb}
\usepackage{amsmath,amssymb,amsfonts}
\usepackage{amscd}
\usepackage{amsthm}
\usepackage{xcolor}

\begin{document}
\title[Stochastic Cahn-Hilliard Convective Brinkman-Forchheimer Model]{Strong Solutions For The Stochastic \vspace{.05in}\\ Cahn-Hilliard Convective Brinkman-Forchheimer Model For \vspace{.05in}\\Tumor Growth}
\author[  Kalpana Rawat and  Kumarasamy Sakthivel]{Kalpana Rawat and  Kumarasamy Sakthivel}
\address{Department of Mathematics \\
Indian Institute of Space Science and Technology (IIST) \\
Trivandrum- 695547, INDIA}
\email{kalpanarawat5hd@gmail.com; sakthivel.k@iist.ac.in}
\curraddr{}
\thanks{}
\date{}
\dedicatory{}

\begin{abstract}
             In this work, we analyze a diffuse-interface model 	for tumor growth, subject to multiplicative white noises, posed on a bounded domain $\mathcal{O} \subset \mathbb{R}^d$, $d=2,3$. The model couples a stochastic incompressible convective Brinkman-Forchheimer (CBF) equation or Navier-Stokes equation with damping $\eta|v|^{r-1}v $ for the averaged velocity field \(v\), to a Cahn-Hilliard (CH) equation for the phase field variable \(\phi\) and to a stochastic reaction-diffusion equation governing the nutrient concentration \(\sigma\). The system is endowed with homogeneous Neumann boundary conditions for the phase field, chemical potential \(\mu\), and nutrient, while a homogeneous Dirichlet boundary condition is prescribed for the velocity. We establish the existence of local strong solutions corresponding to initial data \((v_0, \phi_0, \sigma_0) \in\) \(\mathbb{H}^1\times H^2 \times H^1\), for $ r \geq 1 $ in $d=2$ and $ r \in [1,3] $ in $d=3$. We prove the weak-strong uniqueness holds in both $d = 2$ and $d = 3$. In addition, for $d = 2$, the uniqueness of weak solutions	is obtained for all $\eta,\nu > 0$, and $r \geq 1$, while it holds in $d = 3$ for $r \geq 3$ (with $\eta, \nu >0$ when $r > 3$, and $\eta \nu \geq 1$ when $r = 3$) under the assumption $\sigma \in L^4(\Omega;L^4([0,T];H^1))$. Moreover, for $d=2$ and $r \in [1,3]$, we obtain that the strong solution exists globally in time. The analysis is carried out for a broad class of smooth potentials (apart from the classical double-well potential), relying on the Galerkin scheme, projection onto low and high modes, stochastic energy inequalities derived from It\^o’s formula, and pairwise comparison argument.
\end{abstract}
\maketitle	
\section{Introduction}
			Many mathematical models of tumor growth are derived from continuum mixture theory, which treats the tumor as a system of interacting phases-typically representing tumor cells and healthy cells. Using conservation laws for mass and momentum, various studies formulate models governed by partial differential equations that describe how these components evolve (\cite{Wise2008, HawkinsDaarud2012,Garcke2016,EbenbeckGarckeNurnberg2021}). These models often couple equations describing tumor expansion with advection-reaction-diffusion systems representing the evolution of nutrients such as oxygen and glucose (\cite{Garcke2017, Colli2015}). Initially, most of the tumor growth models considered tissue as a porous medium and used Darcy's law (\cite{Greenspan1976, Wise2008,Garcke2016, Garcke2017, GarckeLam2016}) to model flow velocity by treating the cells as a slow-moving viscous fluid where pressure gradients drive cell motion, rather than inertial forces. Subsequent enhancements involved adopting Stokes or Brinkman flow (\cite{Ebenbeck2020, EbenbeckGarckeNurnberg2021, Ebenbeck2019b}) without modeling tissues as a porous medium (\cite{Franks2003,Friedman2006}).  However, for many cases when the inertial effects are significant (\cite{Agnaou2017}), the Forchheimer term is used to account for a nonlinear behavior in pressure difference versus flow rate, which shows Darcy's law is no longer valid and the nonlinear inertial corrections (\cite{Irmay1958,Lenci2022}) need to be added to Darcy's law. Recently, in \cite{Fritz2019}, Forchheimer corrections have been suggested for non-Darcy flow regimes in tumor growth models by incorporating nonlinear terms $|v|v$ and $|v|^{2}v$ into a time-dependent Darcy-Brinkman's law. Besides, the Navier–Stokes (NS) system (\cite{Lam2018}) was also applied to incorporate inertia effects by introducing acceleration terms ($\frac{\partial v}{\partial t}$ and $(v \cdot \nabla)v$) into the model.

			In this work, we study a stochastic version of a tumor growth model \eqref{e411}, where we have considered a nonlinear characterization of mixture velocity obeying the CBF law, a CH-type equation with additional source terms and regular potential for phase field variable, and an equation of advection-reaction-diffusion type for nutrient. This deterministic model itself is understood as an extension of the analyses of models explored in \cite{Ebenbeck2019b, GarckeLam2016,Fritz2019} for tumor growth by considering a Brinkman-Forchheimer equation modeling mixture velocity with a convective term $(v\cdot \nabla) v$ and a generalized Forchheimer term $\eta|v|^{r-1}v $ for $ r \geq1$.
		
			  For a bounded domain $ \mathcal{O} \subset  \mathbb{R}^d$, $d=2,3$ with smooth boundary $\partial\mathcal{O}$ and a fixed time $T > 0$, we consider the following Cahn-Hilliard convective Brinkman-Forchheimer (CH-CBF) reaction diffusion system for tumor growth:
\begin{equation}
	\begin{aligned}	\label{e411}
		\begin{cases}
			\frac{\partial v}{\partial t} +  (v \cdot \nabla)v + \eta \,|v|^{r-1}v   - \nu \,\Delta v  + \nabla p - \mu \nabla \phi -z    = 0 \quad &\text{in } \mathcal{O}  \times (0, T),\\
			\nabla \cdot v = 0 \quad &\text{in }\: \mathcal{O}  \times (0, T),\\
			 \frac{\partial \phi}{\partial t}+ (v \cdot \nabla)\phi - \Delta \mu  - (P\sigma - A - \alpha u)h(\phi)  = 0 \quad &\text{in }\: \mathcal{O}  \times (0, T),\\
			\frac{\partial \sigma}{\partial t} + (v \cdot \nabla)\sigma- \Delta \sigma  + c \sigma h(\phi)  + b(\sigma - w) \, dt = 0 \quad &\text{in }\: \mathcal{O}  \times (0, T),\\
			\mu = - \epsilon\,\Delta \phi + \epsilon^{-1}\psi'(\phi)\quad &\text{in }\: \mathcal{O}  \times (0, T),
			\end{cases}
	\end{aligned}
\end{equation}
		with initial and boundary data given as $(v, \phi, \sigma)(0) = (v_0, \phi_0, \sigma_0) \: \text{in }\: \mathcal{O}$, and  $v = \partial_{\mathbf{n}} \mu = \partial_{\mathbf{n}} \phi = 
		\partial_{\mathbf{n}} \sigma =  0\:\text{on }\: \partial \mathcal{O}  \times (0, T)$, respectively, where \textbf{n} denotes the outer unit normal to the boundary $\partial \mathcal{O} ,$ and $\partial_{\mathbf{n}} f := \nabla f \cdot \mathbf{n}$. 
		 
		The unknown functions are the volume-averaged velocity $v$ of the cell mixture, pressure $p$, order (phase) parameter $\phi$, the nutrient concentration for the tumor growth denoted by $\sigma$, and the quantity $\mu$, which is the chemical potential (for \( \phi\)) of the binary mixture. Motivated by \cite{Lam2018}, a general energy density is defined as follows:
\begin{equation}\label{e2}
		\mathcal{E}(v, \phi, \nabla \phi, \sigma) = \int_{\mathcal{O}}  \left( \frac{1}{2} |v|^2 + F(\phi, \nabla \phi) + \frac{1}{2} |\sigma|^2 \right) \,dx,
\end{equation}
			where  \( F(\phi, \nabla \phi) \) is a free energy functional of Ginzburg-Landau type defined by
			\(F(\phi, \nabla \phi) :=  \frac {\epsilon}{2}|\nabla \phi|^2 +  \epsilon^{-1} \psi(\phi)\) with $\psi(s)$ being a regular potential  having equal minima at \( s = \pm 1 \). It is worth noting that for any regular potential $ \psi \in C^2(\mathbb{R})$, the boundary condition \( \partial_{\mathbf{n}} \phi =\partial_{\mathbf{n}} \mu = 0\)  is equivalent to  \(\partial_{\mathbf{n}} \phi = \partial_{\mathbf{n}} \Delta \phi = 0\).
			  
            We assume that the parameters $ \eta, \nu, \epsilon, P, A, \alpha, c$, and $b$ are strictly positive constants. More precisely, $\eta$ denotes the Forchheimer influence of the fluid, $\nu$ the kinematic viscosity of the fluid, $\epsilon$ defines the interface thickness separating healthy and tumor cell species, $P$ the tumor proliferation rate, $A$ the apoptosis rate, $\alpha$ the effectiveness rate of the cytotoxic drugs, $c$ the nutrient consumption rate, and $b$ the nutrient supply rate. The parameter $r \in [1, \infty)$ represents the absorption or damping exponent, and $r=3$ is known as the critical exponent. The function \( h \) is typically assumed to be nonnegative, interpolating between $h(-1) = 0$ and $h(1) = 1,$ while $z$ is a given external force. The inclusion of additional source terms in the CH equation introduces biologically relevant mechanisms; the proliferation of tumor cells (tumor growth) is modeled by $P \sigma h(\phi)$, whereas the process of apoptosis is represented by $Ah(\phi)$, analogous to \cite{Orrieri2020}, we have included the effect of cytotoxic drugs via \( \alpha u h(\phi) \), where \( u \) acts as the source. For practical use, \( u: [0,T] \to [0,1] \) is spatially constant with $ u = 1 $ for full dosage and \( u = 0 \) for no dosage. Meanwhile, in the nutrient advection-reaction-diffusion type equation, consumption of the nutrient only in the presence of the tumor cells is represented by $c \sigma h(\phi)$. In view of \cite{BY}, we also consider that the tumor has undergone angiogenesis and $w$ denotes the nutrient concentration in a pre-existing vasculature and \( b(w-\sigma) \) models the supply of nutrient from the blood vessels if \( w > \sigma \), and the transport of nutrient away from the domain \( \mathcal{O}  \) if \( w < \sigma \).

			The deterministic tumor growth models have attracted significant analytical attention, including the well-posedness of the models. In \cite{ Garcke2017}, the authors have analyzed the CH system with chemotaxis and active transport mechanisms in bounded domains with a Lipschitz boundary in  
			$\mathbb{R}^d$, $d \in \mathbb{N}$. The existence of global weak solutions was shown for all $d$, and continuous dependence on initial and boundary data was established for $d \leq4$. For the Navier–Stokes-Cahn–Hilliard (NS-CH) system with chemotaxis and singular potentials (e.g., logarithmic type), \cite{He2021} proves the existence of a global weak solution in both dimensions $d=2,3$ and continuous dependence and uniqueness for $d=2$. The Darcy–Forchheimer–Brinkman equation extended by incorporating Forchheimer nonlinearities $F_1|v|v$, and $F_2|v|^{2}v$ together with local and nonlocal effects in tumor growth has been studied in \cite{Fritz2019} within a bounded Lipschitz domain in $\mathbb{R}^d$, $d \leq 3$. The work establishes the existence of a global weak solution and includes parameter-sensitivity analyses. Recently, \cite{brunk2025analysis} proves the existence of a weak solution within bounded 2D and 3D domains for a coupled system with non-constant mobility and convective transport consisting of a CH equation for phase separation and a generalized quasi-incompressible Forchheimer equation for the velocity field with Forchheimer nonlinearity $\eta(\phi)|v|^{r-1}v $ for $ r >1,$ and obtains the numerical solutions. 

			While deterministic models have provided valuable insights into the general dynamics of tumor growth, they fail to capture the inherent randomness of cellular processes. Tumor proliferation and differentiation are inherently stochastic, and experimental uncertainties in proliferation and apoptosis rates suggest that carcinogenesis should be treated as a stochastic process (\cite{Tan1998,Lo2007}). Recent work has incorporated random fluctuations at the cellular and tissue levels (\cite{Lima2014a}) and developed stochastic angiogenesis models to better represent capillary network formation (\cite{Niemisto2005, Capasso2009,Orrieri2020}), thereby addressing key limitations of purely deterministic approaches.
			
			In what follows, we consider a stochastic counterpart of the system \eqref{e411} by incorporating two independent random perturbations: one acting on the nutrient reaction diffusion equation with the aim of modeling the effects of angiogenesis, and another on the fluid velocity equation as an unknown internal microscopic thermal agitation, or a random source. Specifically, by adding the two independent cylindrical Wiener processes $W_1$ and $W_2$ to the deterministic system \eqref{e411}, we obtain the following stochastic CH-CBF reaction diffusion model for tumor growth:
\begin{subequations}
	\begin{align}
		&dv +  (v \cdot \nabla)v\, dt + \eta \,|v|^{r-1}v   - \nu \,\Delta v\, dt  + \nabla p \, dt- \mu \nabla \phi \, dt- z \, dt = G_1(v) \, dW_1 \quad &\text{in }\: \mathcal{O}  \times (0, T)\label{e01},\\
		&\nabla \cdot v = 0 \quad &\text{in }\: \mathcal{O}  \times (0, T)\label{e11},\\
		&d\phi + (v \cdot \nabla)\phi\, dt - \Delta \mu \, dt - (P\sigma - A - \alpha u)h(\phi) \, dt = 0\quad &\text{in }\: \mathcal{O}  \times (0, T)\label{e21},\\
		&d\sigma + (v \cdot \nabla)\sigma\, dt- \Delta \sigma \, dt + c \sigma h(\phi) \, dt + b(\sigma - w) \, dt = G_2(\sigma) \, dW_2\quad &\text{in }\: \mathcal{O}  \times (0, T)\label{e31},\\
		&\mu = - \epsilon\,\Delta \phi + \epsilon^{-1}\psi'(\phi) \quad &\text{in }\: \mathcal{O}  \times (0, T)\label{e41},
	\end{align}
\end{subequations}
	with the following boundary and initial conditions:
\begin{equation}\label{e5}
	\begin{cases}
		\begin{aligned}
			\partial_{\mathbf{n}} \mu = \partial_{\mathbf{n}} \phi = \partial_{\mathbf{n}} \sigma =  0 \quad &\text{on } \partial \mathcal{O}  \times (0, T), \\
			v = 0 \quad &\text{on } \partial \mathcal{O}  \times (0, T), \\
			(v, \phi, \sigma)(0) = (v_0, \phi_0, \sigma_0) \quad &\text{in } \mathcal{O} .
		\end{aligned}
	\end{cases}
\end{equation}  
			 	Here, $G_1$  and $G_2$ are suitable stochastically integrable processes with respect to $W_1$ and $W_2$, respectively. The terms $G_1(v) \, dW_1$ and $G_2(\sigma) \, dW_2$ represent random external forces depending on $v$ and $\sigma$, respectively.
			 
				Let us briefly review some of the literature on damped NS (CBF) equations in both deterministic and stochastic scenarios.  The global solvability of deterministic CBF equations in various settings, including the whole space, periodic domains, and bounded domains, has been studied in the literature. The paper \cite{Cai2008} studied the  NS equations with damping $\eta|v|^{r-1}v$ in $\mathbb{R}^3$, proving global weak solutions for $r \geq 1$, global strong solutions for $r \geq 7/2$ (see \cite{Zhou2012} for $r \geq 3$), and uniqueness of strong solutions for $7/2 \leq r \leq 5$. The authors in \cite{Kalantarov2012} established the global well-posedness of the 3D CBF  equations in bounded domains, proving the existence and uniqueness of weak solutions with energy estimates, as well as higher regularity of strong solutions for all positive times. The work \cite{Gautam2025} established the global well-posedness of weak solutions with energy inequality for the damped Navier-Stokes equations (CBF) in bounded/periodic domains in $\mathbb{R}^d\, (d \in \{2,3,4\}$) when $r \in [3, \infty)$ (with $r = 3$ provided $2\eta\nu \geq 1$), and the existence of global strong solutions in periodic domains. The optimal control of the damped Navier-Stokes-Voigt equations has been explored in \cite{Sakthivel2023}. For the case of the stochastic counterpart,	\cite{Bessaih2018} investigates the 3D NS equations with Brinkman-Forchheimer damping $\eta|v|^{2r}v$, anisotropic viscosity, and multiplicative stochastic forcing. They proved the global existence and uniqueness of weak solutions by exploiting the enhanced $L^{2r+2}$ regularity provided by the damping term with $r > 1 $. The stochastic CBF equations with multiplicative Gaussian noise on domains in $\mathbb{R}^d$, $d \in \{2,3,4\}$ are analyzed in the paper \cite{Kinra2025}. The authors established the energy estimates for $r \geq 1$, obtained the pathwise unique strong solutions (probabilistically) on general unbounded domains ($r > 3$ for any $\nu$ and $\eta$, $r = 3$ with $2\eta\nu \geq 1$), and further demonstrated the existence of a global strong solution (analytically) on the torus.

				 Although several studies have incorporated stochastic terms into the CH equation (\cite{Elezovic1991, Debussche2011}), in the present work, this component is modeled deterministically. This choice simplifies the analysis caused by the potential term and focuses on the stochastic effects on the other coupled variables. Beyond these individual analyses, a broader literature has considered pairwise or coupled interactions among the system variables, emphasizing how random perturbations contribute to the dynamics of binary fluid mixtures within the diffuse-interface (or phase-field) approach.
			
				  In \cite{Feireisl2019}, a stochastic Navier-Stokes-Allen-Cahn two-phase flow model is studied in a bounded domain in $\mathbb{R}^3$ in which both equations incorporate white noise. The authors established a stochastic relative energy inequality for dissipative martingale solutions and proved weak-strong uniqueness both pathwise and in law. The paper \cite{Deugoue2021} explored unique local strong solutions ($H^1 \times H^2$ data) for the stochastic NS-CH system with multiplicative noise (white noise) on a bounded domain in $\mathbb{R}^d$, $d=2,3$, and obtained the global existence in $d=2$. A stochastic phase-field tumor model in a smooth bounded domain in $\mathbb{R}^3$ with additive noise in the CH equation and multiplicative noise in the reaction-diffusion equation has been considered in \cite{Orrieri2020}. They discussed the well-posedness (existence of probabilistically strong solutions and continuous dependence) and the optimal control problem with cytotoxic drug concentrations (u, w) as control variables. More recently, \cite{FritzScarpa2023} considered a stochastic variant in a bounded 3D domain with multiplicative Wiener noise in both the CH and reaction-diffusion equations. The existence of martingale solutions is proved for variable mobility and non-increasing growth (logistic, Gompertzian) functions.  Moreover, numerical approximations are proposed with simulations illustrating the stochastic effects on tumor growth.
				  
				  To the best of our knowledge, this is the first study to address a stochastic tumor growth model that incorporates random perturbations in the CH-CBF reaction-diffusion model. Moreover, this study advances beyond existing works on the well-posedness theory for the system modeling tumor growth by establishing the local existence of a unique strong solution for $d=3$ and the global existence for $d=2$ in the subcritical range $r \in [1, 3]$.

				 The study of the proposed model poses several challenges and limitations due to its nonlinear and coupled structure. A major difficulty stems from the fact that the variable  $\sigma$ is strongly coupled with the variables $v$ and $\phi$.  More precisely, the coupling between $\sigma$ and $v$ given by the advection term $(v \cdot \nabla \sigma)$ in the stochastic reaction-diffusion equation  \eqref{e31} behaves similarly to the NS nonlinearity, making it more difficult to derive uniform estimates and establish existence and uniqueness results compared to the convective (advection) term $(v \cdot \nabla \phi)$  in \eqref{e21}, where this term can be handled due to the higher regularity of $\phi$. It is worth noting that due to the presence of $v \cdot \nabla \sigma$ term, the global existence of strong solutions (in $d = 3$) of the CH-CBF reaction-diffusion model seems difficult to establish for any $r > 1$, even in a periodic domain (or in the whole space). In fact, by virtue of this nonlinear term, the \emph{uniqueness of weak solutions} in $d =3$ is proven (see Proposition \ref{weakuniq}) only with an additional regularity assumption for $\sigma \in L^{4}(\Omega; L^{4}([0,T];H^1)).$ Besides, owing to the interaction of $\sigma$ with $\phi$ (as a proliferation term) in the CH equation \eqref{e21}, the model does not have the property of mass conservation for the phase field \(\phi\), which drives us to adopt the full norm for $\phi$ instead of the seminorm, and the use of the full norm makes the analysis significantly challenging. A further technicality arises when handling the potential $\psi$ with prescribed growth conditions (see [\ref{[A1]}]), particularly in proving the  uniqueness (see Propositions \ref{weak-strong} and \ref{weakuniq}), which is handled with an additional $L^4(\Omega;L^4((0,T) ; H^2))$  regularity for $\phi.$ The inclusion of an additional source term (proliferation) in \eqref{e21} also necessitates a more delicate analysis for several estimates.
                 
                 Finally, it is worth noting that in the deterministic/stochastic scenario, the Forchheimer term ($|v|^{r-1}v,\, r \geq 3$) often acts as a regularizing mechanism to get the strong solution of the CBF equation in a periodic domain/torus in $d =2,3$  or in the whole space (see \cite{Cai2008}, \cite{Kinra2025}), but it does not seem to work when it comes to a stochastic case in a bounded domain. The commutativity of the Leray-Helmholtz projector with the Laplace operator, which is used to get the regularity of the solutions, does not hold in a bounded domain (see \cite[Chapter 2]{Robinson2016}); consequently, we use the embedding theorems and restrict the absorption exponent $r$ based on the spatial dimensions to get various results.
				                   
                The main contributions of this work are the following:
				\begin{enumerate}[label=(\arabic*)]
					\item The $p$th-order energy estimates (Lemma \ref{Ener_est}) have been derived for the stochastic CH-CBF reaction diffusion system for all $r \geq1$ in both dimensions $d=2$ and $d=3$.
                 \item Weak-strong uniqueness (Proposition \ref{weak-strong}) for all  $ r \geq 1 $ in  $d=2$ and for $ r \in [1,3] $ in  $d=3$ cases is established. In addition to that, the uniqueness of weak solutions (Proposition \ref{weakuniq}) is obtained for all $\eta,\nu > 0$ and $r \geq 1$ in $d = 2,$ and for $r \geq 3$  in $d = 3$ (with $\eta, \nu > 0$ when $r > 3$, and $\eta\nu \geq 1$ when $r = 3$) under the assumption that $\sigma \in L^{4}(\Omega; L^{4}([0,T];H^1)).$ It is evident that in the absence of nutrient concentration $\sigma,$ that is, for the stochastic damped NS-CH system in $d=3$, the global uniqueness of weak solutions holds for any $r \geq 3.$
                  \item The existence of local maximal strong solutions (Theorem \ref{maxsol}) of the stochastic CH-CBF reaction diffusion system defined on a bounded domain $ \mathcal{O} \subset \mathbb{R}^d$, $d=2,3$ with initial data $(v_0, \phi_0, \sigma_0)$ in $\mathbb{H}^1 \times H^2 \times H^1$, for $ r \geq 1 $ in  $d=2$ and for $ r \in [1,3] $ in  $d=3$, is obtained.
				\item In the case $\mathcal{O} \subset \mathbb{R}^2$, with initial data $(v_0, \phi_0, \sigma_0)$ in $\mathbb{H}^1 \times H^2 \times H^1$, the global existence of unique maximal strong solutions (Theorem \ref{global2d}) has been established for $ r \in [1,3]$.
				\end{enumerate}
				
				The structure of the paper is organized as follows: Section \ref{Section 2} outlines the properties of nonlinear operators and the stochastic setup used in the paper. Section \ref{Section 3} configures the key assumptions and  develops the Galerkin approximation using high–low mode decomposition and a comparison technique. Section \ref{Section 4}  proves the existence and uniqueness of local strong  solutions and the uniqueness of weak solutions. Section \ref{Section 5} discusses global existence in the 2D case. Finally, Section \ref{Section 6} contains the appendix, which includes additional analytical details and auxiliary results that support the main theorems.
\section{Preliminaries}\label{Section 2}
			      We first introduce the function spaces used throughout the paper, and our analysis will focus on the physically relevant spatial dimensions $d = 2, 3$. Unless otherwise stated, the boundary of $\mathcal{O}  \subset \mathbb{R}^d$ is assumed to be sufficiently smooth.
          
            Let \( X \) be a real Banach space. We denote its norm by \( \|\cdot\|_{X} \), its dual space by \( X' \), and the duality pairing between \( X \) and \( X' \) by \( \langle \cdot, \cdot \rangle_{X', X} \).
            If \( X \) is further assumed to be a real Hilbert space, we denote the inner product by \( (\cdot, \cdot)_X \) with \(\|\cdot\|_X^2  =  (\cdot, \cdot)_X \). Finally, we define $\mathbb{X} := X^d = X \times \cdots \times X$ ($d$ times). 
            For the standard Sobolev spaces, we use the notation 
            \( W^{s,p} := W^{s,p}(\mathcal{O}) \) for \( s \geq 0, \: p \geq 1 \), equipped with the norm \( \| \cdot \|_{W^{s,p}} \) defined in the usual way by $
             \|u\|_{W^{s,p}} := \left( \sum_{|\alpha| \le s} \| D^{\alpha} u \|_{L^{p}}^p \right)^{1/p}$. When \( s = 0 \), we write \( W^{0,p} := L^{p} \). For \( p = 2 \), we denote \( W^{s,2} \) by \( H^{s} \), equipped with the inner product \( (u, v)_{H^{s}} := \sum_{|\alpha| \le s} (D^{\alpha}u, D^{\alpha}v)_{L^{2}} \), and we use the notation \( | \cdot |_{L^2} \) for the \( L^2 \) norm. Next, define the divergence-free spaces as follows:
     \begin{equation*}
             	\begin{aligned}
             		\begin{cases}
             	  	L^p_{\mathrm{div}}
             	  &:= {\{ v \in {\mathbb{L}}^p : \operatorname{div} v = 0,\, v \cdot \mathbf{n} \big|_{\partial \mathcal{O}} = 0 \}}, \, \text{for} \, \, p\geq2, \\[5pt]
             		V &:= {\{ v \in \mathbb{H}^1 : \operatorname{div} v = 0,\, v \big|_{\partial \mathcal{O}} = 0 \}}.
             		\end{cases}
             		\end{aligned}
       \end{equation*}
                    Let the Hilbert space $L^2_{\mathrm{div}}$ be equipped with $L^2$ scalar product and norm denoted by $(y, v) := \int_{\mathcal{O}} y \cdot v \, dx$ and $ \|v\|^2_{L^2_{\mathrm{div}}} := (v, v),$ respectively. Similarly, we characterize the space $V$ endowed with $H^1_0$ inner product and norm defined as  $((y,v)) := \sum_{i=1}^{d}(\partial_{x_i}y,\partial_{x_i}v), $ and $  \|v\|^2 := ((v,v))$, respectively. For $p \in (2, \infty)$, we define the  \(L^p_{\mathrm{div}}\) norm by \(\|v\|_{L^p_{\mathrm{div}}}^p :=  \int_{\mathcal{O}} |v(x)|^p \, dx\), while for $p = \infty$, we set $\|v\|_{L^\infty_{\mathrm{div}}} := \operatorname*{ess\,sup}_{x \in \mathcal{O}} |v(x)|$.
                     Next, we denote the dual of the Banach space \( V \cap L^p_{\mathrm{div}}\) as \(V' + 	L^{p'}_{\mathrm{div}}\), where \(\frac{1}{p}+ \frac{1}{p'} = 1\). We have a continuous embedding of  \( V \cap L^{p}_{\mathrm{div}} \hookrightarrow V \hookrightarrow L^2_{\mathrm{div}} \cong (L^{2}_{\mathrm{div}})^{'} \hookrightarrow V' \hookrightarrow V' + L^{p'}_{\mathrm{div}}.\) 
        
                    For  $y \in L^1,$ we define the (generalized) mean value by $ \bar{y} := \frac{1}{|\mathcal{O}|} \int_{\mathcal{O}} y \, dx,$ where $|\mathcal{O}|$ denotes the Lebesgue measure of the domain $\mathcal{O}$. We shall use the Poincar\'e–Wirtinger inequality,
  		            $|y - \bar{y}|_{L^2} \leq C_{\mathcal{O}} |\nabla y|_{L^2}, \, \forall y \in H^1,$ where \(C_{\mathcal{O}}\) is a constant depending only on \(\mathcal{O}\).

    \subsection{Linear Operators}\
 
   	        The Stokes operator \( A_0 \) is defined by $A_0 v := - \mathcal{P} \Delta v,$ with domain $D(A_0) := \mathbb{H}^2 \cap V$, where \(\mathcal{P}\) denotes the Leray-Helmholtz projector from $\mathbb{L}^2$ onto $L^2_{\mathrm{div}}$. 
            By the classical spectral theory for the self-adjoint, compact and bounded operator $A_0^{-1}$, there exists a sequence \( \{\lambda_j \}_{j \in \mathbb{N}}\) consisting of eigenvalues of $A_0$, satisfying $
            0 < \lambda_1 < \lambda_2 \leq \dots \leq \lambda_n \leq \lambda_{n+1} \leq \dots$ and corresponding eigenfunctions \( \{ w_j \}_{j \in \mathbb{N}} \) forming an orthonormal basis of $L^2_{\mathrm{div}}$ and an orthogonal basis of \( V \) such that
  \begin{equation*} 
	\begin{cases}
		\begin{aligned}
			A_0 w_j &= \lambda_j w_j && \text{in } \mathcal{O}, \\
			\nabla \cdot w_j &= 0 && \text{in } \mathcal{O}, \ \ \ \
			w_j = 0  \ \ \text{on } \ \ \partial \mathcal{O}.
		\end{aligned}
	\end{cases}
\end{equation*}
                 For each $\alpha \geq 0$ and $v = \sum_j v_j w_j \in L^2_{\mathrm{div}}$, where $v_j = (v,w_j)$, take $ D(A_0^\alpha) = \{v \in L^2_{\mathrm{div}}: \sum_j \lambda_j^{2\alpha} v_j^2  < \infty\},$ 
                and for any \(v \in D(A_0^\alpha)\) define the fractional powers of $A_0$, by $A_0^{\alpha}v = \sum_j \lambda_j^{\alpha}v_jw_j$ . The space \( D(A_0^\alpha) \) is equipped with the norm given by  $|v|_\alpha^2 := \|A_0^\alpha v\|^2_{ L^2_{\mathrm{div}}} = \sum_j \lambda_j^{2\alpha} v_j^2.$ From the standard results, the norms  $\|A_0^{1/2} \cdot\|_{ L^2_{\mathrm{div}}}$ and $\|A_0 \cdot\|_{ L^2_{\mathrm{div}}}$ are equivalent to the norms in $\mathbb{H}^1$ and $\mathbb{H}^2 $, respectively (see \cite[Chapter 3]{Temam2001}).

             We define a linear nonnegative unbounded operator \( A_1 \) on \( L^2 \) by $A_1 \varphi := -\Delta \varphi$ with domain $D(A_1) := \{\varphi  \in H^2: \partial_{\textbf{n}} \varphi = 0 \text{ on } \partial \mathcal{O}\}$. 
             From classical theory, there exists a sequence of eigenvalues \(\{ \beta_j \}_{j \in \mathbb{N}}\) that satisfy $0 = \beta_1 < \beta_2 \leq \dots \leq \beta_n \leq \beta_{n+1} \leq \dots$, an orthonormal basis \( \{ \psi_j \}_{j \in \mathbb{N}} \) for \( L^2 \) consisting of eigenfunctions of $A_1$, which is also an orthogonal basis for \( H^1\), such that
        \begin{equation*}
        	\begin{cases}
        		\begin{aligned}
        			A_1 \psi_j &= \beta_j \psi_j && \text{in } \:\mathcal{O}, \\
        			\partial_{\textbf{n}}\psi_j &= 0 && \text{on}\: \partial \mathcal{O}.
        		\end{aligned}
        	\end{cases}
        \end{equation*}
        		 For $\alpha \geq 0$, and $\varphi = \sum_j \varphi_j \psi_j \in L^2$, where $\varphi_j = (\varphi,\psi_j)$, we define the family of Hilbert spaces, as $D(A_1^\alpha) = \{\varphi \in L^2: \sum_j (1 + \beta_j^{2\alpha})\varphi_j^2 < \infty\}$, and introduce the fractional powers of $A_1$, as the operators from $D(A_1^\alpha)$ into \(L^2\), that is, \(A_1^\alpha : D(A_1^\alpha) \rightarrow L^2 \, \text{such that} \,
        	   A_1^{\alpha}\varphi = \sum_j \beta_j^{\alpha}\varphi_j\psi_j \). The space \( D(A_0^\alpha) \) is equipped with the norm $
        	  |\varphi|_{D(A_1^{\alpha})}^2:= |\varphi|^2_{ L^2} +|A_1^\alpha \varphi|^2_{ L^2} = \sum_j (1+ \beta_j^{2\alpha}) \varphi_j^2.$\\
          	 Moreover, using Lemma \ref{turiq0} and  Lemma \ref{turiq}, we deduce that for $y \in 	H^{2}$ with
    		 $\partial_{\mathbf{n}} y = 0$: 
    \begin{equation}\label{norm1}
      		\|y\|^2_{H^2} \leq C (\:|y|^2_{ L^2} + |A_1 y|^2_{ L^2}).
    \end{equation}
    		 Besides, for every $y \in H^{4}$ obeying the boundary conditions 
   		     $\partial_{\mathbf{n}} y = \partial_{\mathbf{n}}\Delta y = 0$ on $\partial \mathcal{O}$, we have 
    \begin{align}
      		\|y\|^2_{H^3} &\leq C(\:|y|^2_{ L^2} + |A_1^{3/2} y|^2_{ L^2}),\\ 
     		 \|y\|^2_{H^4} &\leq C(\:|y|^2_{ L^2} + |A_1^2 y|^2_{ L^2}). \label{norm3}
     \end{align}
			The positive constant \(C\) in inequalities \eqref{norm1}-\eqref{norm3} depends solely on \(\mathcal{O}\).
			\begin{Rem}\label{emer}
			 For any \(y \in H^3\)such that \(\partial_{\textbf{n}} y = 0 \) in \(\partial \mathcal{O}\), we have \(|A_1 y|_{ L^2} \leq |A_1^{1/2} y|^{1/2}_{ L^2}|A^{3/2}_1 y|^{1/2}_{ L^2}\). Moreover, if \(y \in H^4\) with \(\partial_{\textbf{n}} \Delta y = 0 \) in \(\partial \mathcal{O}\), then we have \(|A^{3/2}_1 y|_{ L^2} \leq |A_1 y|^{1/2}_{ L^2}|A^{2}_1 y|^{1/2}_{ L^2}\).
        \end{Rem}
        
      	  \begin{Rem}\label{egv} It is well known that for some suitable constant C, the Stokes eigenvalues $\lambda_n,$ satisfy $ \lambda_n\sim C\, n^{2/d }$ as $ n \to \infty$ (for large n). where d is the space dimension (see \cite{Metivier1978}). Moreover, this behavior is the same as that of the eigenvalues $\beta_n$ of the Neumann Laplace operator $A_1.$  
       \end{Rem}
          Suppose \(\mathcal{X}_n := \) span\(\{w_1, w_2,...,w_n\}\) and \(\mathcal{Y}_n:= \) span\(\{\psi_1, \psi_2,...,\psi_n\}\). We define \(\mathcal{P}_1^n\) as the projection of $L^2_{\mathrm{div}}$ onto \( \mathcal{X}_n\) together with its complement \(\mathcal{Q}_1^n = \mathcal{I}- \mathcal{P}_1^n\) and 
          \(\mathcal{P}_2^n\) as the projection operator from \(L^2\) onto \(\mathcal{Y}_n \) with its complement \(\mathcal{Q}_2^n = \mathcal{I}- \mathcal{P}_2^n\).
    \begin{Lem} \label{norm_esti}
         Suppose that \( \alpha_1 < \alpha_2 \). For any $v \in D(A^{\alpha_2}_0) $ and $ \varphi \in D(A_1^{\alpha_2})$
         we have the following generalized Poincar\'e and inverse Poincar\'e inequalities:
  \begin{equation}\label{est1}
          \text{(i)} \quad 	\left\|A^{\alpha_2}_0 \mathcal{P}_1^n(v) \right\|_{ L^2_{\mathrm{div}}} \leq \lambda_n^{(\alpha_2 - \alpha_1)} \left\|A^{\alpha_1}_0 \mathcal{P}_1^n(v) \right\|_{ L^2_{\mathrm{div}}}, \quad 
           	\left\|A^{\alpha_1}_0 \mathcal{Q}_1^n(v) \right\|_{ L^2_{\mathrm{div}}} \leq \lambda_n^{(\alpha_1 - \alpha_2)} \left\|A^{\alpha_2}_0 \mathcal{Q}_1^n(v) \right\|_{ L^2_{\mathrm{div}}},
   \end{equation}
     \begin{equation}\label{est2}
            		\text{(ii)} \quad \left|A^{\alpha_2}_1 \mathcal{P}_2^n(\varphi) \right|_{ L^2} \leq \beta_n^{(\alpha_2 - \alpha_1)} \left|A^{\alpha_1}_1 \mathcal{P}_2^n(\varphi) \right|_{ L^2}, \quad 
            		\left|A^{\alpha_1}_1 \mathcal{Q}_2^n(\varphi) \right|_{ L^2} \leq \beta_n^{(\alpha_1 - \alpha_2)} \left|A^{\alpha_2}_1 \mathcal{Q}_2^n(\varphi) \right|_{ L^2}.
        \end{equation}
                    Moreover, for \(\mathcal{P}_2^n\) and  	\(\mathcal{Q}_2^n\) defined as above, we have the following rate of approximation of the norms: $|\mathcal{Q}_2^n(\varphi)|_{ L^2} \leq \frac{C}{\sqrt{\beta_n}}\|\mathcal{Q}_2^n(\varphi)\|_{H^1},$ $\|\mathcal{Q}_2^n(\varphi)\|_{H^1} \leq \frac{C}{\sqrt{\beta_n}}\|\mathcal{Q}_2^n(\varphi)\|_{H^2},$ $\|\mathcal{Q}_2^n(\varphi)\|_{H^2} \leq \frac{C}{\sqrt{\beta_n}}\|\mathcal{Q}_2^n(\varphi)\|_{H^3},$
         		and $\|\mathcal{Q}_2^n(\varphi)\|_{H^3} \leq \frac{C}{\sqrt{\beta_n}}\|\mathcal{Q}_2^n(\varphi)\|_{H^4},$
         	    where \( \lambda_n\) and \( \beta_n\) denote the \(n^{th}\) eigenvalues of operators \(A_0\) and \(A_1\), respectively.
        \end{Lem}
       \begin{proof} The proofs of $(i)$ and $(ii)$ follow from Lemma 2.1 of \cite{Glatt-Holtz2009}. Next, we use Lemmas \ref{turiq0} and  \ref{turiq} together with the inequality \eqref{est2} to get
 \begin{equation*}
 	\begin{aligned}
 			&\quad |\mathcal{Q}_2^n(\varphi)|^2_{ L^2} = |A^{(0)}_1\mathcal{Q}_2^n(\varphi)|^2_{ L^2}\, \leq \, \beta_n^{2(0-1/2)}|A^{1/2}_1\mathcal{Q}_2^n(\varphi)|^2_{ L^2},\\
       		& \quad \|\mathcal{Q}_2^n(\varphi)\|^2_{H^1} = 	|\mathcal{Q}_2^n(\varphi)|^2_{ L^2} + |A^{1/2}_1\mathcal{Q}_2^n(\varphi)|^2_{L^2}\leq \frac{C}{\beta_n}\|\mathcal{Q}_2^n(\varphi)\|^2_{H^1} + \beta_n^{2(1/2-1)}|A_1\mathcal{Q}_2^n(\varphi)|^2_{ L^2}.
          \end{aligned} 
   \end{equation*}
            Similarly, we can derive the last two approximations of the norms.
   \end{proof}
		Hereafter, we use the following product Hilbert spaces with their respective norms: 
\begin{align*}
        \mathcal{H} &:= L^2_{\mathrm{div}} \times H^1\times L^2;\;\|(v, \phi, \sigma)\|_{\mathcal{H}}^2 :=  \|v\|^2_{L^2_{\mathrm{div}}} + \|\phi\|_{H^1}^2 + |\sigma|^2_{ L^2},\\
        \mathcal{V} &:= V \times H^2\times H^1;\; 	\|(v, \phi, \sigma)\|_{\mathcal{V}}^2  := \|A_0^{1/2} v\|^2_{ L^2_{\mathrm{div}}} + \| \phi\|_{H^2}^2 +\|\sigma\|_{H^1}^2,\\
         \mathcal{Z} &:= D(A_0) \times H^4\times H^2;\; \|(v, \phi, \sigma)\|_{\mathcal{Z}}^2  := \|A_0 v\|^2_{ L^2_{\mathrm{div}}} + \| \phi\|_{H^4}^2 +\|\sigma\|_{H^2}^2.   
      \end{align*}
\subsection{Nonlinear Operators}\ 

\textbf{Forchheimer Term.}
		Let us define an operator associated to the nonlinear term \(|v|^{r-1}v\) as \( \mathcal{A}_r(v):=\)  \(\mathcal{P}(|v|^{r-1}v)\), where \(v \in V\cap L^{r+1}_{\mathrm{div}}\). Moreover, for any \(v \in  L^{r+1}_{\mathrm{div}}, r\geq 1\) we have
		\begin{eqnarray}
		\langle \mathcal{A}_r(v), v \rangle = \int_{\mathcal{O}}|v|^{r-1}v \cdot v\,dx = \|v\|_{L^{r+1}_{\mathrm{div}}}^{r+1}. \label{dam}
		\end{eqnarray}
	  The following are some of the technical calculations we encounter with the term \(|v|^{r-1}v\):
			\begin{itemize}
				\item[(i)] For any \(v, y \in \mathbb{R}^d\) we have
	 \begin{equation}\label{for1}
		 		\big|~|v|^{r-1}v - |y|^{r-1}y~\big| \leq r\left(~|v|+ |y|~\right)^{r-1}|v-y|.
		\end{equation}
		
				\item [(ii)] For \( v(x,t) = (v_1(x,t), \ldots, v_d(x,t) ), x = (x_1, \ldots, x_d) \in \mathbb{R}^d \), we have $\nabla |v(x)|= \frac{1}{|v(x)|}(\nabla v(x))^{'}v(x)$,
		 where \('\) denotes the transpose operator, and hence we deduce
     	\begin{equation}\label{for2}
			\begin{aligned}
				\nabla\big(|v(x)|^{r-1}\,v(x)\big)&=(r-1)|v(x)|^{r-2}\,(\nabla|v(x)|) \otimes v(x) + |v(x)|^{r-1}\,\big(\nabla v(x)\big)\\
				&= (r-1)|v(x)|^{r-2}\,\left(\frac{1}{|v(x)|}\big(\nabla v(x)\big)^{'}v(x)\right) \otimes v(x) + |v(x)|^{r-1}\,\big(\nabla v(x)\big)\\
				&= (r-1)|v(x)|^{r-3}\,\Big(\big(\nabla v(x)\big)^{'}v(x) \otimes v(x)\Big) + |v(x)|^{r-1}\,\big(\nabla v(x)\big),\\
			\end{aligned}
		\end{equation}
			where \(a\otimes e := ae' \in \mathbb{R}^{d \times d} \quad \forall a, e \in \mathbb{R}^d.\) Further, the above expression can be rewritten as:
		\begin{equation}
			\nabla(|v|^{r-1}\,v ) :=\begin{cases}
			\begin{aligned}
				\nabla v \quad :& \quad r =1,\\
				|v|^{r-1}\nabla v + (r-1)\frac{1}{|v|^{3-r}}\big(((\nabla v)'v) \otimes v\big)\quad :&\quad 1<r<3,\\
				|v|^{r-1}\nabla v + (r-1)|v|^{r-3}\big(((\nabla v)'v) \otimes v\big)\quad :&\quad r \geq 3.
				\end{aligned}
			\end{cases}
		\end{equation} 
		
\end{itemize}
	    \textbf{Trilinear Operator I.}	For $y, v, \xi \in V$, the  continuous trilinear form \( b_0 : V \times V \times V \to \mathbb{R} \) is defined by:
		$
		b_0(y, v, \xi) = \int_{\mathcal{O}} ((y \cdot \nabla) v) \cdot \xi \, dx = \sum_{i,j=1}^d 	\int_{\mathcal{O}} y_i \frac{\partial v_j}{\partial x_i} \xi_j \, dx.
		$
		For  $v,y \in V$, we denote by \( B_0(y, v) 	\in V' \) the linear functional such that \( \langle B_0(v, y), \xi \rangle = b_0(v, y, \xi)\),  \( \forall\: \xi\in V\). We also denote \( B_0(v) = B_0(v, v) = \mathcal{P}((v \cdot \nabla)v) \). Using integration by parts, \( b_0(\cdot, \cdot, \cdot) \) satisfies (\cite{Temam2001}):
  \begin{equation}\label{b0_1}
    \begin{aligned}
    	\begin{cases}
		  b_0(y, v, v) = 0 \quad &\forall y, v \in V ,\\
		  b_0(y, v, \xi) = -b_0(y, \xi, v) \quad &\forall y, v, \xi \in V .
		 \end{cases}
     \end{aligned}
  \end{equation}
       Furthermore, we collect the following inequalities satisfied by \(b_0(\cdot, \cdot, \cdot)\)\,:
\begin{enumerate}[label=(\roman*)]
    \item  Applying the Gagliardo-Nirenberg inequality (see Lemma \ref{l1}) and embedding $H^1 \hookrightarrow L^6$, we get
        \begin{eqnarray}\label{b_03}
		     |b_0(y,v,\xi)| &\leq& 
	\begin{cases}
			c \|y\|_{L_{\mathrm{div}}^4} \|\nabla v\|_{L_{\mathrm{div}}^4} \|\xi\|_{L^2_{\mathrm{div}}} & \text{in } d = 2, \nonumber\\
			c \|y\|_{L_{\mathrm{div}}^6} \|\nabla v\|_{L_{\mathrm{div}}^3}\|\xi\|_{L^2_{\mathrm{div}}} & \text{in } d = 3,\nonumber\\
    \end{cases}\\
            &\leq& \begin{cases}
		      c \|y\|^{1/2}_{ L^2_{\mathrm{div}}} \|y\|^{1/2} \|v\|^{1/2} \|A_0 v\|^{1/2}_{ L^2_{\mathrm{div}}} \|\xi\|_{ L^2_{\mathrm{div}}} & \text{in } d = 2, \\
		      c \|y\|\, \|v\|^{1/2} \|A_0 v\|^{1/2}_{L^2_{\mathrm{div}}} \|\xi\|_{L^2_{\mathrm{div}}} & \text{in } d = 3,\\
	\end{cases}
	\end{eqnarray}
            for all $y \in V$, $v \in D(A_0)$, $\xi \in L^2_{\mathrm{div}}$.
        \item Using embedding $H^1 \hookrightarrow L^p$ for $ p= 4,6$, we obtain
    \begin{eqnarray}\label{b0_4}
		      |b_0(y,v,\xi)| &\leq& 
		\begin{cases}
			c \|y\|_{L_{\mathrm{div}}^4} \|\nabla v\|_{L^2_{\mathrm{div}}} \|\xi\|_{L_{\mathrm{div}}^4} & \text{in } d = 2, \nonumber\\
			c \|y\|_{L_{\mathrm{div}}^6} \|\nabla v\|_{L^2_{\mathrm{div}}} \|\xi\|_{L_{\mathrm{div}}^3} & \text{in } d = 3,\nonumber\\
        \end{cases}\\
            &\leq& c \|y\|\, \|v\|\, \|\xi\|_{L^2_{\mathrm{div}}}^{1/2}\|\xi\|^{1/2}  \qquad \text{in } d = 2,3, \:  \forall \,y,v,\xi \in V.
	\end{eqnarray}
	   \end{enumerate}	
      \begin{Rem}  
      For the convective (advection) nonlinearity appearing in \eqref{e21} and \eqref{e31}, we use the same trilinear operator form. Moreover, the properties and inequalities for this operator stated below are valid whenever the involved variables possess the required regularity.
      \end{Rem}
	\textbf{Trilinear Operator II.}	For $ v \in V$ and $\varphi, \theta \in H^1$, the continuous trilinear form \( b_1 : V \times H^1 \times H^1    \to \mathbb{R} \) is defined by: $
		b_1(v, \varphi, \theta) = \int_{\mathcal{O}} ((v \cdot \nabla) \varphi)  \theta \, dx 	= 	\sum_{i,j=1}^d \int_{\mathcal{O}} v_i \frac{\partial \varphi_j}{\partial x_i} \theta_j \, dx.$ For   \( v \in V, \varphi \in H^1\), the bilinear form \( B_1(v, \varphi) \in (H^1)'  \) is the linear functional given by \( \langle B_1(v, \varphi), \theta \rangle = b_1(v, \varphi, \theta)\), \( \forall \,\theta \in H^1\). We also write \( B_1(v, \varphi) = ((v \cdot \nabla)\varphi) \). Similarly, integration by parts yields the following properties for \( b_1(\cdot, \cdot, \cdot) \):
 \begin{equation}\label{b1_1}
	\begin{aligned}
		\begin{cases}
		 b_1(v, \varphi, \varphi) = 0 \quad &\forall  v \in V, \varphi \in H^1,\\
		 b_1(v, \varphi, \theta) = -b_1(v, \theta, \varphi) \quad &\forall v \in V, \varphi, \theta \in H^1.
     \end{cases}
   \end{aligned}
   \end{equation}
			 Moreover, we obtain the following inequalities for \( b_1(\cdot, \cdot, \cdot) \)\,:
  \begin{enumerate}[label=(\roman*)]
  			\item $|b_1(v,\varphi, \theta)| \leq 
  			c \|v\|_{L^{\infty}_{\mathrm{div}}} |\nabla \varphi|_{L^2} |\theta|_{L^2}$ 
  			for all $v \in L^{\infty}_{\mathrm{div}}, \nabla\varphi \in L^2, \,\theta \in L^2.$ From Agmon's inequality (see Lemma \ref{l2}), we further derive
	\begin{equation}\label{b1_2}
 			|b_1(v, \phi, \theta)| \leq 
 		\begin{cases}
 			c \|v\|^{1/2}_{ L^2_{\mathrm{div}}} \|A_0 v\|^{1/2}_{ L^2_{\mathrm{div}}} \|\varphi\|_{H^1} |\theta|_{ L^2} & \text{in } d = 2, \\
 			c \|v\|^{1/2} \|A_0 v\|^{1/2}_{ L^2_{\mathrm{div}}} \|\varphi\|_{H^1} |\theta|_{ L^2} & \text{in } d = 3,
 	\end{cases}
 \end{equation}
 			for all $v \in D(A_0)$, $\varphi \in H^1$, $\theta \in L^2$. 
			\item  $|b_1(v,\varphi, \theta)| \leq 
			c \|v\|_{L^4_{\mathrm{div}}} \|\nabla \varphi\|_{L^4} |\theta|_{L^2}$ 
			for all $v \in L^4_{\mathrm{div}}, \nabla\varphi \in L^4, \,\theta \in L^2$. Hence, by the Gagliardo-Nirenberg inequality  (Lemma \ref{l1}), we get
\begin{equation}\label{b1_4}
	|b_1(v,\varphi,\theta)| \leq 
	\begin{cases}
		c \|v\|^{1/2}_{L^2_{\mathrm{div}}} \|v\|^{1/2} |A_1^{1/2}\varphi|_{L^2}^{1/2} \|\nabla\varphi\|_{H^1}^{1/2} |\theta|_{ L^2} & \text{in } d = 2, \\
		c \|v\|^{1/4}_{L^2_{\mathrm{div}}} \|v\|^{3/4} |A_1^{1/2}\varphi|_{L^2}^{1/4} \|\nabla\varphi\|_{H^1}^{3/4} |\theta|_{ L^2} & \text{in } d = 3,
	\end{cases}
\end{equation}
		for all $v \in V$, $\varphi \in H^2$, $\theta\in L^2$.
	\item  By the embedding of \(H^1 \hookrightarrow  L^p\), for            $p =4,6$, we obtain	
    \begin{eqnarray} \label{b1_5}
			|b_1(v,\varphi,\theta)| &\leq& 
			\begin{cases}
				c \|v\|_{L^4_{\mathrm{div}}} |A_1^{1/2} \varphi|_{L^2} \|\theta\|_{L^4} & \text{in } d = 2, \nonumber\\
				c \|v\|_{L^6_{\mathrm{div}}} |A_1^{1/2} \varphi|_{L^2} \|\theta\|_{L^3} & \text{in } d = 3, 
               \end{cases} \\
               &\leq& c \|v\|\, |A_1^{1/2}\varphi|_{L^2} |\theta|_{L^2}^{1/2}\|\theta\|_{H^1}^{1/2}  \  \text{in } d = 2,3, \  \forall v \in V, \ \varphi, \theta \in H^1.
  \end{eqnarray}
 	\end{enumerate}
        \textbf{Coupling Term.}
	  For $ \mu \in H^1$, $\phi \in H^2$ and $y \in L^2_{\mathrm{div}}$, we define a map \( r_0: H^1 \times H^2 \times L^2_{\mathrm{div}} \to \mathbb{R}\), corresponding to the coupling term $\mu  \nabla \phi$ by: $
		r_0(\mu, \phi, y) = \int_{\mathcal{O}} (\mu  \nabla \phi) \cdot y \, dx 	= 		\sum_{i,j=1}^d \int_{\mathcal{O}} \mu \frac{\partial \phi}{\partial x_i} y_j \, dx.$ We 
			denote \(R_0(\mu, \phi) = \mathcal{P}(\mu \nabla \phi) \).
\begin{Rem}\label{pre}
        We can replace the term \(\mu \nabla \phi\) by \(\epsilon A_1 (\phi) \nabla \phi\) after reformulation of the pressure in the \(v\) equation (for more details, see \cite{Garcke2016}).
 \end{Rem}
 		We derive the following inequalities for  $R_0(\cdot, \cdot)$:
 		\begin{enumerate}[label=(\roman*)]
 		\item $|r_0(A_1 \phi, \phi, v)|\: \leq 
 		c |A_1 \phi|_{L^2} \|\nabla \phi\|_{L^{\infty}} \|v\|_{L^2_{\mathrm{div}}}$, 
 		for all $v \in L^2_{\mathrm{div}}, \phi \in H^3$. Using Agmon's inequality (see Lemma \ref{l2}), we obtain
 		\begin{equation}\label{r_01}
 			|(R_0(A_1 \phi, \phi), v)| \leq 
 			\begin{cases}
 				c \|v\|_{ L^2_{\mathrm{div}}} |A_1 \phi|_{ L^2} |A_1 ^{1/2}\phi|^{1/2}_{L^2} \|\nabla \phi\|_{H^2}^{1/2} & \text{in } d = 2, \\
 				c \|v\|_{ L^2_{\mathrm{div}}} |A_1 \phi|_{ L^2} \|\nabla \phi\|_{H^1}^{1/2} \|\nabla \phi\|_{H^2}^{1/2} & \text{in } d = 3.
 			\end{cases}
 		\end{equation}
	\item 	Using Lemma \ref{l1} and the Sobolev embedding $H^1 \hookrightarrow L^p$, for $p =4,6$, we derive\begin{eqnarray}\label{r_04}
		|r_0(\mu, \phi, v)| &\leq &
		\begin{cases}
			c \|\mu\|_{L^4} |A_1^{1/2}\varphi|_{L^2} \|v\|_{L^4_{\mathrm{div}}} & \text{in } d = 2, \\
			c \|\mu\|_{L^6} |A_1^{1/2} \varphi|_{L^2} \|v\|_{L^3_{\mathrm{div}}} & \text{in } d = 3,\nonumber
		\end{cases}\\
& \leq & c \|\mu\|_{H^1}\, |A_1^{1/2}\phi|_{L^2} \|v\|_{ L^2_{\mathrm{div}}}^{1/2} \|v\|^{1/2}, \ \ \text{in} \ \ d=2,3,
		\end{eqnarray}
for all $v \in V$, $\phi \in H^1$ and $\mu \in H^1$.
\end{enumerate}
	\begin{Rem}\label{R0B1}
		 One can also note that $\bigl( R_{0}(A_{1}\phi, \phi), v \bigr) 
		= \bigl( B_{1}(v, \phi), \mu \bigr) 
		= \bigl( B_{1}(v, \phi), A_{1}\phi \bigr) 
		\; \forall (v, \phi) \in V \times H^2.$
\end{Rem}
\subsection{Stochastic Framework}
			
			Suppose  \( W_1 \) and \( W_2 \) are independent cylindrical Wiener processes defined on a complete filtered probability space \( (\Omega, \mathcal{F}, \mathbb{P}, \{ \mathcal{F}_t \}_{t \geq 0}) \), satisfying the usual hypotheses, i.e., \( \{ \mathcal{F}_t \}_{t \geq 0} \) is a right-continuous filtration such that \( \mathcal{F}_0 \) contains all the \( \mathbb{P} \)-null subsets of \( (\Omega, \mathcal{F}) \). These processes take values in separable Hilbert spaces \( U_1 \) and \( U_2 \), respectively. Assume each \( U_i \) is equipped with an orthonormal basis (ONB) \( \{e^i_n\}_{n \in \mathbb{N}} \), allowing the representation for a \( U_i \) valued Wiener process \( W_i \), by $W_i(t) = \sum_{n=1}^{\infty} e^i_n\beta^i_n(t), \: i \in \{1,2\}$, where for \(i = 1,2,\, n \in \mathbb{N}\), \( \beta^i_n \) denotes one-dimensional independent and identically distributed (i.i.d.) \( \mathcal{F}_t\)- adapted Brownian motions. Moreover, it is well-known that the above series does not converge in  \( U_i \) but rather in some Hilbert-Schmidt extension \( \tilde{U}_i \) of \( U_i \). Throughout, we use the following stochastic basis: 
			\begin{equation}\label{e7}
				\mathcal{S}_B = \big( \Omega, \mathcal{F}, \mathbb{P}, \{\mathcal{F}_t\}_{t \geq 0},  \{\beta_k^i(t), t \geq 0, k \in \mathbb{N}, i \in \{1,2\} \} \big).
			\end{equation}

\begin{Def}(Hilbert-Schmidt operator)
       		 Let $U$ be a separable Hilbert space and $X$ be a Hilbert space. As usual, $\mathcal{L}(U,X)$ denotes the collection of all bounded (or continuous) linear operators from $U$ to $X$. The space of Hilbert-Schmidt operators from \( U \) to \( X \) is denoted by \( \mathcal{L}_2(U, X) \), and is defined as: $	\mathcal{L}_2(U,X) := \left\{ Z \in \mathcal{L}(U,X) \mid \sum_{k=1}^{\infty} \|Z e_k\|_X^2 < \infty \right\},$  where \( \{e_k\}_{k \in \mathbb{N}} \) is a given ONB of \( U \). The space \( \mathcal{L}_2(U,X) \) is itself a Hilbert space when equipped with the inner product and hence the corresponding norm $\langle R, Z \rangle_{\mathcal{L}_2(U,X)} = \sum_{k=1}^{\infty} \langle R e_k, Z e_k \rangle_X,$ and, $	\| R \|_{\mathcal{L}_2(U,X)} = \left( \sum_{k=1}^{\infty} \| R e_k \|_X^2 \right)^{1/2}$, respectively, where \( R, Z \in \mathcal{L}_2(U,X) \).
 \end{Def}
 \begin{Rem}
          It is important to note that the definition of a Hilbert–Schmidt operator and its norm, defined as above, are independent of the choice of the ONB \(\{e_k\}\) (see \cite{PrevotRockner}).
        \end{Rem}
            Furthermore, for $\omega \in \Omega$, $t \ge 0$, $x \in  L^2_{\mathrm{div}},$ and $y \in L^2$, let $\{g_{1,k}(\omega, t, x)\}_{k \in \mathbb{N}} \subset L^2_{\mathrm{div}}$ and $\{g_{2,k}(\omega, t, y)\}_{k \in \mathbb{N}}$ $ \subset L^2$ be sequences such that 
			$\sum_{k=1}^{\infty} \| g_{1,k}(\omega, t, x) \|_{L^2_{\mathrm{div}}}^2 < \infty, \; \sum_{k=1}^{\infty} | g_{2,k}(\omega, t, y) |_{L^2}^2 < \infty.$
			Thus, the operators \( G_1 : L^2_{\mathrm{div}} \to \mathcal{L}_2(U_1,L^2_{\mathrm{div}}) \) and \( G_2 :L^2 \to \mathcal{L}_2(U_2,L^2) \) are well-defined, given by:
\begin{equation*}\label{eq:ass80}
			G_1(x) : e^1_k \mapsto g_{1,k}(x), \quad \forall x \in L^2_{\mathrm{div}}, \quad \forall k \in \mathbb{N},
\end{equation*}  			
	\begin{equation*}\label{eq:ass9}
			G_2(y) : e^2_k \mapsto g_{2,k}(y), \quad \forall y \in L^2, \quad \forall k \in \mathbb{N}.
	\end{equation*}
				Hence, the stochastic forcing terms are written as:
	\begin{equation}\label{equinoise}
		     \begin{aligned}
			        G_1(t,v) \, dW_1(t) = \sum_{k=1}^{\infty} g_{1,k}(t,v) \, d\beta^1_k(t),\ \ \ \
					G_2(t,\sigma) \, dW_2(t) = \sum_{k=1}^{\infty} g_{2,k}(t,\sigma) \, d\beta^2_k(t).
				\end{aligned}	
		\end{equation}
                 For a given Banach space $X$ (or a Hilbert space, as appropriate in context), we define, for $p \in [1,\infty)$:
        \begin{align*}
         	&L^p_{\text{loc}}([0,\infty); X) 
         	= \bigcap_{T > 0} L^p(0,T; X), \\
         	&C_w([0,\infty); X) 
         	= \left\{ f \in L^\infty_{\text{loc}}([0,\infty); X) \;|\; (f(t), x)_X \in C([0,\infty); \mathbb{R}) \;\forall x \in X \right\},\\
         	&L^p(\Omega;X):=	L^p(\Omega,\mathcal{F}, \mathbb{P};X) 
         	=  \left\{ f: \Omega \to X \;|\; f \;\text{is strongly measurable and} \int_{\Omega} \|f\|_X^p d \mathbb{P} < \infty \right\},\\
           &L^\infty(\Omega;X) := 	L^\infty(\Omega,\mathcal{F}, \mathbb{P};X) 
         	= \{\,f:\Omega\to X \text{ measurable} \;|\; 
         	\exists\, C<\infty,\; \|f(\omega)\|_X \le C \ \text{a.s.}\,\}.
         \end{align*}
      	\textbf{Stochastic Integral:} Let $U_i$ be separable Hilbert spaces with orthonormal bases $\{e^i_n\}_{n \in \mathbb{N}}$, and $X_i$ Hilbert spaces, for $i=1,2$. Suppose that $\rho_i$ is a predictable process with values in $\mathcal{L}_2(U_i, X_i)$, and satisfies $\rho_i \in L^2\big(\Omega; L^2_{\mathrm{loc}}([0,\infty); \mathcal{L}_2(U_i, X_i))\big)$ for $i = 1,2$. Then the stochastic integral of $\rho_i$ with respect to the cylindrical Wiener process $W_i$ (taking values in $U_i$) is denoted by $\int_0^t \rho_i(s) \, dW_i(s)$. This integral is a unique continuous square integrable  \( X_i \)-valued \( \mathcal{F}_t \)-martingale (see \cite{DaPrato2014}) satisfying: 
           \begin{equation*}
         ( \int_0^t \rho_i(s) \, dW_i(s), z_i )_{X_i}= \sum_{k=1}^{\infty} \int_0^t (\rho_i(s) e^i_k, z_i)_{X_i} \, d\beta_k^i(s),  \quad \text{for all} \: z_i \in X_i, \ \ i=1,2.
           \end{equation*}
         \begin{Def}(Uniformly Lipschitz Functions)\label{ULLF}
       		 Let \( (U, |\cdot|_U) \) and \( (Y, |\cdot|_Y) \) be two metric spaces. We say that a function $  l : \Omega \times [0,T] \times U \to Y$ is \emph{uniformly Lipschitz} with constant \( L_U \) if, for all \( x, y \in U \),
   \begin{equation} \label{e14}
         	|l(\omega,t,x) - l(\omega,t,y)|_Y \leq L_U |x - y|_U, \quad \forall (\omega,t) \in \Omega \times [0,T].
 \end{equation}
 			 In addition, any \(l\) (defined as above) is said to be have  \emph{linear growth condition} if, for all, \(x \in U\) there exists a positive constant \(B_U\),  such that
 \begin{equation} \label{e15}
   		|l(\omega,t,x)|_Y \leq B_U \big(1+ |x |_U \big), \quad \forall \;(\omega,t) \in \Omega \times [0,T].
   \end{equation}
   		 We denote the collection of all such functions satisfying \eqref{e14} and \eqref{e15} by \( \mathfrak{L}_{\mathrm{lin}}(U,Y) \).  
    \end{Def} 
    
  \subsection{Abstract Formulation}
 			 We take the Helmholtz-Leray orthogonal projection \(\mathcal{P}\) in  \eqref{e01} together with the notation introduced for v, \(\phi\) and \(\sigma\) to get the following abstract formulation: 
  \begin{subequations}
  			\begin{align}
 			 &dv + \nu A_0 v \, dt  + \eta\mathcal{A}_r(v) \, dt + B_0(v,v) \, dt -  R_0(\epsilon A_1 \phi,\phi) \, dt - z \, dt = G_1(v) \: dW_1 \quad \text{in}  ~\: V'+L^{\frac{r+1}{r}}_{\mathrm{div}},\label{a1}\\
  			 &d\phi +  A_1 \mu \, dt + B_1(v,\phi) \, dt - (P\sigma - A - \alpha u) h(\phi) \, dt = 0 \quad \text{in} ~\: (H^{1})^{\prime},\label{a2} \\  
  			 &d\sigma + A_1 \sigma \, dt + B_1(v,\sigma) \, dt  + c \sigma h(\phi) \, dt + b(\sigma - w) \, dt = G_2(\sigma) \, dW_2 \quad \text{in}~ \: (H^{1})^{\prime}, \label{a3}\\  
  			 &\mu = \epsilon A_1 \phi + \epsilon^{-1}\psi '(\phi), \label{a4} \\  
  			 &(v, \phi, \sigma)(0) = (v_0, \phi_0, \sigma_0)\label{a5},
  		\end{align}	 
  	\end{subequations}
 where with a slight abuse of notation, we have written \(\mathcal{P}(G_1(v) \, dW_1) = G_1(v) \, dW_1\) for simplicity.

		\section{Assumptions and Approximations}\label{Section 3}
\begin{Ass}	The following assumptions are used throughout the paper: 
	\begin{enumerate}[label=\textbf{[A\arabic*]}, ref=A\arabic*]
		       \item \label{[A1]}
		       The function $\psi \in C^4(\mathbb{R})$ is non-negative such that 
		       \begin{equation} \label{psi-decomp}
		       	\psi(s) = \psi_1(s) + \psi_2(s), \quad \forall s \in \mathbb{R},
		       \end{equation}  
		       where $\psi_1, \psi_2 \in C^4(\mathbb{R})$ and satisfies the growth conditions:  \begin{equation} \label{psi1-bounds}
		       	R_1 \big( 1 + |s|^{\rho - 2} \big) \leq \psi_1''(s) \leq R_2 \big( 1 + |s|^{\rho - 2} \big), \quad \forall s \in \mathbb{R},
		       \end{equation}  
		       \begin{equation} \label{psi2-bounds}
		       	|\psi_2''(s)| \leq R_3, \quad \forall s \in \mathbb{R},
		       \end{equation}  
		       where $R_i, \; i = 1,2,3$, are positive constants with $R_1 < R_2$ and $\rho \in [2,6]$ (see \cite{Ebenbeck2019b} or \cite{Ebenbeck2019}).
            Utilizing \eqref{psi-decomp}-\eqref{psi2-bounds}, one can further deduce that there exists a constant \(C_{\psi}>0\) such that:
	\begin{equation} \label{eq:ass1}
		\begin{cases}
				\psi(s) \geq C_{\psi} (|s|^{\rho}- 1); \quad \forall s \in \mathbb{R}, \\
				|\psi(s_1)-\psi(s_2)| \leq C_{\psi}(1+ |s_1|^{\rho-1} + |s_2|^{\rho-1})|s_1-s_2|;\quad \forall s_1, s_2 \in \mathbb{R},\\
				|\psi'(s)|  \leq C_{\psi}(1+ |s|^{\rho-1});\quad \forall s \in \mathbb{R}, \ \ \rho \in[2,6].
		\end{cases}
	\end{equation}
		    Furthermore, we make the following assumptions for various derivatives of \(\psi\):
 \begin{equation}
		\begin{aligned}\label{eq:ass2}
			\begin{cases}
			 & |\psi^{(k)}(s_1)-\psi^{(k)}(s_2)| \leq C_{\psi}(1+ |s_1|^{5-k} + |s_2|^{5-k})|s_1-s_2|;\quad \forall s_1, s_2 \in \mathbb{R}, \ k=1,2,3,\\ 
			  &|\psi^{(k)}(s)| \leq C_{\psi} (1 + |s|^{6-k}); \quad \forall s \in \mathbb{R}, \ \ k=2,3,4, 
			   \end{cases}
    	    \end{aligned}
	\end{equation}
		    where $\psi^{(k)}(\cdot)$ denotes the $k$-th derivative of $\psi.$
		   
			 \item \label{[A2]} The functions \( G_i \) for \( i \in \{1,2\} \) are defined as:
	\begin{equation*}\label{eq:ass3}
			\begin{cases}
				G_1 :& \Omega \times [0,\infty) \times L^2_{\mathrm{div}} \to \mathcal{L}_2(U_1,L^2_{\mathrm{div}}), \\
				G_2 :& \Omega \times [0,\infty) \times L^2 \to \mathcal{L}_2(U_2,L^2),
			\end{cases}
	\end{equation*}
	and moreover, we shall assume that the following hold (see Definition \ref{ULLF} for $\mathfrak{L}_{\mathrm{lin}}(\cdot,\cdot)$): 
	\begin{align*}
			 G_1 &\in \mathfrak{L}_{\mathrm{lin}}\big(L^2_{\mathrm{div}}, \mathcal{L}_2(U_1, L^2_{\mathrm{div}})\big)\, \cap\, \mathfrak{L}_{\mathrm{lin}}\big(V, \mathcal{L}_2(U_1, V)\big)\, \cap\, \mathfrak{L}_{\mathrm{lin}}\big(D(A_0), \mathcal{L}_2(U_1, D(A_0))\big),\\
			  G_2 &\in \mathfrak{L}_{\mathrm{lin}}\big(L^2, \mathcal{L}_2(U_2, L^2)\big)\, \cap\, \mathfrak{L}_{\mathrm{lin}}\big(H^1, \mathcal{L}_2(U_2, H^1)\big)\, \cap\, \mathfrak{L}_{\mathrm{lin}}\big(H^2, \mathcal{L}_2(U_2, H^2)\big).
			\end{align*}
		
        	\item \label{[A3]} The constants \(\nu, \eta, \epsilon, P, A, \alpha, b, c \) are positive, and the function \( h :  \mathbb{R} \to [0,1] \) is continuously differentiable with a bounded derivative (and hence Lipschitz-continuous with Lipschitz constant \( L_h > 0 \)).
		
		    \item \label{[A4]} \( z \) is a \( V' \)-valued predictable process, and \( u \) and \( w \) are predictable processes taking values in \( L^2 \) and \( (H^1)^{\prime} \), respectively, satisfying:
  \begin{equation*}\label{eq:ass12}
					(z,u,w)  \in L^2(\Omega; L^2_{\text{loc}}([0,\infty); \mathcal{H})). 
	       \end{equation*}
	       Further, for any \(T>0\), the function $u \in [0,1] $  a.e. in  $\Omega \times (0,T)\times \mathcal{O}$.
	        \item \label{[A5]} We assume that the initial conditions satisfy:
	\begin{equation*}\label{eq:ass13}
				(v_0, \phi_0, \sigma_0) \in L^2(\Omega; \mathcal{V}).      	
	\end{equation*}
		\end{enumerate}	
	\end{Ass}
 \begin{Rem}\label{furesti}    
		 It is straightforward to check that the typical classical double-well potential function  \( \psi(s) = \frac{1}{4}(s^2-1)^2\) fulfills the conditions \eqref{psi-decomp}-\eqref{eq:ass2}. Moreover, for the splitting (see \eqref{psi-decomp}-\eqref{psi2-bounds}), we can choose, $\psi_1(s) = \frac{1}{4}s^4$ and $\psi_2(s) = -\frac{1}{2} s^2 + \frac{1}{4}.$  
		It is important to emphasize that the constant \(C_{\psi} >0\) is generic; hence, it is allowed to change from line to line or estimate to estimate.
		\end{Rem} 
\begin{Rem}
    It is worth noting that in $d=2$, \eqref{eqequality} implies that  $\phi(\cdot \wedge \tau) \in L^2(\Omega, L^{\infty}([0,\infty); L^p))$, for all p $\in [1,\infty)$. Hence, the potential $\psi$ is allowed to have arbitrarily large polynomial growth as compared to $d=3$, where we are bound to choose a fixed polynomial growth as defined in \eqref{eq:ass1}.
\end{Rem}

	\subsection{Notion of a Solution}
        	We assume a stochastic system defined on a fixed filtered probability space \(\mathcal{S}_B\) (see \eqref{e7}) and \((v_0,\phi_0, \sigma_0)\) be \(\mathcal{F}_{0}\) measurable with values in \(\mathcal{V}\) satisfying \([\ref{[A5]}] \). Moreover, suppose we have [\ref{[A2]}]-[\ref{[A3]}] and, \([\ref{[A4]}] \) for stochastic forcing \(G_1, G_2\), function \(h\), and external sources, \(z,u,w\), respectively.
        	\begin{Def}(Local weak solutions)
			\label{def4.1} We call the pair \( \{(v,\phi, \sigma), \tau \} \) is a local weak solution of \eqref{a1}-\eqref{a5} if \( \tau \) is a strictly positive stopping time and \( (v(\cdot \wedge \tau), \phi(\cdot \wedge \tau), \sigma(\cdot \wedge \tau))\) is a predictable process in \( V'+ L^{\frac{r+1}{r}}_{\mathrm{div}} \times (H^{1})^{\prime}\times (H^{1})^{\prime} \),  $r \geq 1$ such that the following  conditions hold:
			\[
		\left\{	
	       \begin{aligned}\label{eq:sol1}
				&(v(\cdot \wedge \tau),\phi(\cdot \wedge \tau),\sigma(\cdot \wedge \tau)) \in L^2(\Omega; C_w([0,\infty); L^2_{\mathrm{div}}\times H^1 \times L^2)), \\
                & v \mathbf{1}_{t \leq \tau} \in L^2(\Omega; L^2_{\text{loc}}([0,\infty); V) \cap L^{r+1}(\Omega; L^{r+1}_{\text{loc}}([0,\infty);L^{r+1}_{\mathrm{div}})), ( \phi \mathbf{1}_{t \leq \tau}, \sigma \mathbf{1}_{t \leq \tau})\in L^2(\Omega; L^2_{\text{loc}}([0,\infty);  H^2 \times H^1 )).
		\end{aligned} 
 			\right.
 		\]
			Moreover, for any \( t > 0 \), and	for all  $(\theta_1, \theta_2, \theta_3) \in V \cap L^{r+1}_{\mathrm{div}} \times H^1\times H^1,  $ the solutions satisfy the weak formulation 
\begin{equation}\label{regu1}
	\begin{aligned}
				&\big(v(t \wedge \tau, \theta_1\big) + \int_0^{t \wedge \tau} \big\langle \nu A_0 v + \eta\mathcal{A}_r(v) + B_0(v,v)  - \epsilon R_0 (A_1 \phi, \phi), \theta_1 \big\rangle \,ds \\
				& \qquad = \big (v_0, \theta_1\big) + \int_0^{t \wedge \tau} \big\langle z, \theta_1 \big\rangle \, ds + \int_0^{t \wedge \tau} \sum_{k=1}^{\infty} \big\langle g_{1,k}(s,v),  \theta_1\big\rangle d\beta_k^1(s), \\  
				&\big(\phi(t \wedge \tau),\theta_2\big)+ \int_0^{t \wedge \tau} \big\langle A_1 \mu + B_1(v,\phi),  \theta_2\big\rangle ds = \big(\phi_0,  \theta_2\big) + \int_0^{t \wedge \tau} \big\langle \big(P\sigma - A - \alpha u \big) h(\phi), \theta_2 \big\rangle  \,ds, \\  
				&\big(\sigma(t \wedge \tau), \theta_3 \big) + \int_0^{t \wedge \tau} \big\langle A_1 \sigma + B_1(v,\sigma), \theta_3 \big\rangle  \, ds \\
				& \qquad = \big(\sigma_0, \theta_3\big) - \int_0^{t \wedge \tau} \big\langle c \sigma h(\phi) + b(\sigma - w), \theta_3 \big\rangle  ds + \int_0^{t \wedge \tau} \sum_{k=1}^{\infty} \big\langle g_{2,k}(s,\sigma), \theta_3 \big\rangle d\beta_k^2(s), \\ 
				&\mu =\epsilon A_1 \phi + \epsilon^{-1}\psi'(\phi).
			\end{aligned}	
		\end{equation}
		\end{Def}
		
		\begin{Def}(Local strong solutions)\label{def4.2} 
			The pair \( \{(v,\phi, \sigma), \tau \} \) is a local strong solution  of \eqref{a1}-\eqref{a5} if \( \tau \) is a strictly positive stopping time and  \( (v(\cdot \wedge \tau), \phi(\cdot \wedge \tau), \sigma(\cdot \wedge \tau)) \) is a predictable process in \( \mathcal{H} \) satisfying the following conditions:
			
			\[
			\left\{
			\begin{aligned}\label{eq:sol3}
				&(v(\cdot \wedge \tau),\phi(\cdot \wedge \tau),\sigma(\cdot \wedge \tau)) \in L^2(\Omega; C([0,\infty); V\times H^2 \times H^1)), \\
				&(v \mathbf{1}_{t \leq \tau},\phi \mathbf{1}_{t \leq \tau},\sigma \mathbf{1}_{t \leq \tau}) \in L^2(\Omega; L^2_{\text{loc}}([0,\infty); D(A_0) \times H^4 \times H^2)).
			\end{aligned}
			\right.
			\]
			Additionally, for any \(t>0\), the pair \( (v,\phi, \sigma) \) satisfies the following equivalent formulation of \eqref{regu1} as an equation in \( L^2_{\mathrm{div}} \times L^2\times L^2: \)
        \begin{align}
		 		 &v(t \wedge \tau) + \int_0^{t \wedge \tau} \big[ \nu A_0 v + \eta\mathcal{A}_r(v) + B_0(v,v) - \epsilon R_0 (A_1 \phi, \phi) \big] ds 
		 		 = v_0 + \int_0^{t \wedge \tau} z \, ds + \int_0^{t \wedge \tau} G_1(s,v) \, dW_1(s), \nonumber\\  
				&\phi(t \wedge \tau) + \int_0^{t \wedge \tau} \big[ A_1 \mu + B_1(v,\phi) \big] ds = \phi_0 + \int_0^{t \wedge \tau} \big( P\sigma - A - \alpha u \big) h(\phi) \,ds, \label{regu2}\\  
				&\sigma(t \wedge \tau) + \int_0^{t \wedge \tau} \big[ A_1 \sigma + B_1(v,\sigma) \big] ds
				 = \sigma_0 - \int_0^{t \wedge \tau} \big[ c \sigma h(\phi) + b(\sigma - w) \big] ds + \int_0^{t \wedge \tau} G_2(s,\sigma) \, dW_2(s),\nonumber \\ &\mu = \epsilon A_1 \phi + \epsilon^{-1}\psi'(\phi)\nonumber.
		\end{align}
\end{Def}
		\begin{Def}(Maximal strong solutions)
				\label{def4.3}
				   Suppose that \((v,\phi, \sigma)\) is a predictable process in \( V'+ L^{\frac{r+1}{r}}_{\mathrm{div}} \times (H^1)^{'}\times (H^1)^{'} \) and that \(\zeta\) is a strictly positive stopping time. The triple \(\{ (v,\phi, \sigma),\{\tau_n\}_{n \in \mathbb{N}}, \zeta\} \) is said to be a \(\mathcal{V}\)-valued maximal strong solution of the system \eqref{a1}-\eqref{a5}, provided  \(\tau_n\) is a sequence of stopping times that are \(\mathbb{P}\)-a.s. monotone increasing and convergent to \(\zeta\), whereby \(\{( v, \phi, \sigma), \tau_n\}\) is an \(\mathcal{V}\)-valued local strong solution of the system \eqref{a1}-\eqref{a5} for each $n$ so that
			 \begin{equation*}
				   \sup_{t \in [0,\zeta]} \|(v,\phi, \sigma)\|_{\mathcal{V}}^2 + \int_0^\zeta \left(\|A_0 v\|^2_{L^2_{\mathrm{div}}} + \|\phi\|_{H^4}^2 + \|\sigma\|_{H^2}^2\right)ds = \infty, 
			 \end{equation*}
				    on the set  \(\{ \zeta < \infty\}.\) A maximal strong solution \(\{ (v,\phi, \sigma),\{\tau_n\}, \zeta\} \) is said to be global if  \(\zeta = \infty\) a.s..
				    \end{Def}
			 \begin{Def}(Pathwise uniqueness)\label{pathuq}
                    Solutions of  \eqref{a1}-\eqref{a5} are said to be (pathwise) unique if, given any pair of local (weak or strong) solutions \( \{(v_1, \phi_1, \sigma_1), \tau_1\} \) and \( \{(v_2, \phi_2,\sigma_2), \tau_2\} \) defined on the same filtered probability space \(\mathcal{S}_B\) (see \eqref{e7}), which coincide at $t = 0$ on the event $\bar{\Omega} = \{\big(v_{1}, \phi_{1},\sigma_{1} \big)(0) = \big(v_{2}, \phi_{2},\sigma_{2} \big)(0)\} \subseteq \Omega$, then
            \begin{equation*}
		           \mathbb{P} \{\mathbf{1}_{\bar{\Omega}}(v_1, \phi_1, \sigma_1)(t \wedge \tau) = \mathbf{1}_{\bar{\Omega}}(v_2, \phi_2, \sigma_2)(t \wedge \tau);  \quad\forall t \ge 0 \} = 1, \ \ \text{where} \ \ \tau:= \tau_1 \wedge\tau_2.
            \end{equation*}
        \end{Def}
  
\subsection{The Galerkin Approximation}
                We proceed with the Galerkin approximation in the spectral subspaces defined earlier. Let us recall that for a fixed \(n \in \mathbb{N}\), \(\mathcal{X}_n =\) span\{\(w_1,\ldots, w_n\)\}, \(\mathcal{Y}_n =\) span\{\(\psi_1,\ldots, \psi_n\)\}, where \(\{ w_i, i \in \mathbb{N}\}\) and \( \{\psi_i, i\in \mathbb{N}\}\) are the eigenfunctions of the Stokes and Neumann-Laplace operators, respectively. For an arbitrary but fixed $T > 0,$ the Galerkin approximation consists of seeking an adapted process \( (v_n,\phi_n, \sigma_n) \in  C\left([0,T];\mathcal{X}_n\times \mathcal{Y}_n\times \mathcal{Y}_n \right)\). We say that $(v_n,\phi_n, \sigma_n)$ solves the Galerkin system of order \( n \) if it satisfies
   \begin{subequations}
	\begin{align}
				&dv_n + \nu A_0 v_n\,dt + \mathcal{P}_1^n \big[\eta \mathcal{A}_r(v_n) + B_0(v_n, v_n) - \epsilon R_0(A_1 \phi_n, \phi_n) - z(t) \big]\,dt \nonumber \\&\hspace{2in} = \sum_{k=1}^\infty \mathcal{P}_1^n g_{1,k}(t, v_n)\, d\beta^1_k(t),\label{gal21} \\
				&d\phi_n + \mathcal{P}_2^n \big[ A_1 \mu_n + B_1(v_n, \varphi_n) - (P \sigma_n - A -\alpha u(t))h(\phi_n) \big]\,dt = 0, \label{gal22}\\
				&d\sigma_n +  A_1 \sigma_n\,dt + \mathcal{P}_2^n \big[B_1(v_n, \sigma_n) +c \sigma_n h(\phi_n) + b( \sigma_n -w(t)) \big]\,dt = \sum_{k=1}^\infty \mathcal{P}_2^n g_{2,k}(t, \sigma_n)\, d\beta^2_k(t), \label{gal23}\\
				&\mu_n = \epsilon A_1 \phi_n + \epsilon^{-1} \psi'(\phi_n), \label{gal24}\\
				& \mathcal{P}_n(v_0, \phi_0, \sigma_0) := (v_{n}^{(0)}, \phi_{n}^{(0)}, \sigma_{n}^{(0)})\label{gal25}.
		\end{align}
	\end{subequations}
				Here, \(\mathcal{P}_n := (\mathcal{P}^1_n, \mathcal{P}^2_n) :  L^2_{\mathrm{div}} \times L^2  \to \mathcal{X}_n \times \mathcal{Y}_n\) is defined as the orthogonal projection. Applying the argument used in the proof of \cite[Theorem~1.2.1]{Breckner1998}, we deduce the existence and uniqueness of a solution process \((v_n, \phi_n, \sigma_n) \in L^2(\Omega;\mathcal{X}_n\times\mathcal{Y}_n\times\mathcal{Y}_n )\) to \eqref{gal21}-\eqref{gal25}, with trajectories in $C\left([0,T];\mathcal{X}_n\times \mathcal{Y}_n\times \mathcal{Y}_n \right)$ almost surely (a.s.).
               
               We now move forward to establish the main result of this section. It is worth noting that to prove the existence of local strong solutions, we are employing the pairwise comparison techniques (Lemma 5.1, \cite{Glatt-Holtz2009}) developed for the stochastic NS system. 
\begin{Pro}\label{nes_com}
		     	For $d=2,3$ and $ r \geq 1 $ in  $d=2$, $ r \in [1,3] $ in  $d=3,$   let $(v_n, \phi_n, \sigma_n), n \in \mathbb{N},$ be the sequence of solutions of  \eqref{gal21}-\eqref{gal25}. For any $\widetilde{K} > 0,$ assume that
	           $\|(v_0, \phi_0, \sigma_0)\|_{\mathcal{V}} \leq \widetilde{K}\; \text{ a.s.},$ and that Assumption [\ref{[A2]}]-[\ref{[A3]}], and [\ref{[A4]}] are satisfied  by  $G_i,i=1,2,$ $h,$ and $z,u,w$, respectively. 
			 
           For some $T > 0$ and $M > 1$, consider the collection of stopping times
	\begin{equation}\label{stop}
		\begin{cases}
			\begin{aligned}
			\mathcal{T}^{T,M}_n = \{ \tau &\leq T : \left[ \sup_{t \in [0,\tau]} \|(v_n,\phi_n, \sigma_n)\|_{\mathcal{V}}^2 + \int_0^\tau \left(\nu\|A_0 v_n\|^2_{L^2_{\mathrm{div}}} + \epsilon\|\phi_n\|_{H^4}^2 + \|\sigma_n\|_{H^2}^2\right)ds \right]^{1/2}\\
			\qquad  &\leq \|(v_{n}^{(0)},\phi_{n}^{(0)}, \sigma_{n}^{(0)})\|_{\mathcal{V}} + M \},
			\end{aligned}
			\end{cases}
	\end{equation}
	    	and define $\mathcal{T}^{T,M}_{m,n} := \mathcal{T}^{T,M}_m \cap \mathcal{T}^{T,M}_n.$ Then:
	\begin{itemize}
		    	\item[(i)] For any $T > 0$ and $M > 1$,
		\begin{equation}\label{first}
			\begin{aligned}
				&\lim_{n \to \infty} \sup_{m > n} \sup_{\tau \in \mathcal{T}^{T,M}_{m,n}} \mathbb{E}\Bigg[ \sup_{t \in [0,\tau]} \|(v_m,\phi_m, \sigma_m) - (v_n,\phi_n, \sigma_n)\|_{\mathcal{V}}^2	 \\
				&\hspace{1in} + 2 \int_0^\tau \Bigg( \nu\|A_0(v_m - v_n)\|^2_{L^2_{\mathrm{div}}} +\epsilon \|\phi_m - \phi_n\|_{H^4}^2 + \|\sigma_m - \sigma_n\|_{H^2}^2 \Bigg)ds \Bigg] = 0.
			\end{aligned}
	\end{equation}
			\item[(ii)] Moreover, if for $n \in \mathbb{N}$, $\zeta > 0$ and a stopping time $\tau$, we define
	\begin{equation}\label{secd}
		\begin{aligned}
				Z_n(\tau,\zeta) := \Bigg\{ \sup_{t \in [0, \tau \wedge \zeta]} \|(v_n,\phi_n, \sigma_n)\|_{\mathcal{V}}^2 + &\int_0^{\tau \wedge \zeta} ( \nu\|A_0 v_n\|^2_{L^2_{\mathrm{div}}} + \epsilon\|\phi_n\|_{H^4}^2 + \|\sigma_n\|_{H^2}^2 )ds  \\
				&> \|(v_{n}^{(0)},\phi_{n}^{(0)}, \sigma_{n}^{(0)})\|_{\mathcal{V}}^2 + (M-1)^2\Bigg\},
			\end{aligned}
		\end{equation}
			such that\: $\lim_{\zeta \to 0} \sup_n \sup_{\tau \in \mathcal{T}^{M,T}_n} \mathbb{P}\left(Z_n(\tau,\zeta)\right) = 0$.
        \end{itemize}
		\end{Pro}
	\begin{proof} Let us first prove the convergence (i). 
			 Given $m > n$ and letting $(v_{m,n}, \phi_{m,n}, \sigma_{m,n}) = (v_m, \phi_m, \sigma_m) - (v_n, \phi_n, \sigma_n)$, we derive that $(v_{m,n}, \phi_{m,n}, \sigma_{m,n})$ satisfies:
\begin{subequations}\label{esic}	
		\begin{align}
			&dv_{m,n} + \Big[\nu A_0 v_{m,n} + \eta\Big(\mathcal{P}_1^m \mathcal{A}_r(v_m)- \mathcal{P}_1^n \mathcal{A}_r(v_n) \Big)+ \mathcal{P}_1^m B_0(v_m,v_m) - \mathcal{P}_1^n B_0(v_n,v_n) \nonumber\\
			&\quad \quad- \epsilon\big(\mathcal{P}_1^m R_0(A_1 \phi_m, \phi_m) - \mathcal{P}_1^n R_0(A_1 \phi_n, \phi_n)\big) - (\mathcal{P}_1^m-\mathcal{P}_1^n)z \Big]dt \nonumber\\
			& \quad \quad \quad=  \sum_{k=1}^\infty [\mathcal{P}_1^m g_{1,k}(t,v_m) - \mathcal{P}_1^n g_{1,k}(t,v_n)] d\beta^1_k(t), \label{esic1}\\[9pt]
			&d\phi_{m,n} + \Big[\mathcal{P}_2^m A_1 \mu_m - \mathcal{P}_2^n A_1 \mu_n + \mathcal{P}_2^m B_1(v_m, \phi_m)- \mathcal{P}_2^n B_1(v_n, \phi_n)\nonumber\\
			&\quad \quad - \Big(\mathcal{P}_2^m(P\sigma_m- A- \alpha u)h(\phi_m) - \mathcal{P}_2^n(P\sigma_n- A- \alpha u)h(\phi_n)\Big)\Big] dt= 0,\label{esic2}
            \end{align}
            \begin{align}
			&d\sigma_{m,n} +  \Big[A_1\sigma_{m,n} + \mathcal{P}_2^m B_1(v_m, \sigma_m)- \mathcal{P}_2^n B_1(v_n, \sigma_n) + c\Big(\mathcal{P}_2^m \sigma_m h(\phi_m)- \mathcal{P}_2^n \sigma_n h(\phi_n) \Big)\nonumber\\
			&\quad + b\sigma_{m,n}  - b(\mathcal{P}_2^m-\mathcal{P}_2^n)w\Big]dt = \sum_{k=1}^\infty [\mathcal{P}_2^m g_{2,k}(t,\sigma_m) - \mathcal{P}_2^n g_{2,k}(t,\sigma_n)] d\beta^2_k(t),\label{esic3} \\
			&\mathcal{P}_2^m A_1 \mu_m - \mathcal{P}_2^n A_1 \mu_n = \epsilon A_1^2  \phi_{m,n} + \epsilon^{-1} \Big(\mathcal{P}_2^m A_1 \psi'(\phi_m) - \mathcal{P}_2^n A_1 \psi'(\phi_n)\Big), \label{esic4}\\
			&(v_{m,n}, \phi_{m,n}, \sigma_{m,n})(0) = (\mathcal{P}^m - \mathcal{P}^n)(v_0, \phi_0, \sigma_0)\label{esic5}.
		\end{align}
	\end{subequations}
			We shall do the estimations in four steps: \\
			\textbf{Step 1 (Estimate of the velocity equation).}
			Let \( \tau \in \mathcal{T}^{T,M}_{m,n} \) be arbitrary but fixed. Let \( \tau_a \) and \( \tau_b \) be given stopping times such that \( 0 \leq \tau_a \leq  \tau_b \leq \tau \). By applying the It\^o formula to the processes \( \|A^{1/2}_0 v_{m,n}(\cdot)\|^2_{L^2_{\mathrm{div}}}\) via \eqref{esic1}, integrating over \( [\tau_a, r] \) followed by taking the supremum over \( [\tau_a, \tau_b] \),  we get
	\begin{equation}\label{esiv1}
		\begin{aligned}
			\mathbb{E} \sup_{t \in [\tau_a, \tau_b]} \|v_{m,n}(t)\|^2+ 2\nu \mathbb{E} \int_{\tau_a}^{\tau_b}  \|A_0 v_{m,n}\|^2_{ L^2_{\mathrm{div}}} ds
			\leq \mathbb{E} \|v_{m,n}(\tau_a)\|^2+\sum_{i=1}^5\mathfrak I_i, 
     \end{aligned} \end{equation}      
        where 
\begin{eqnarray*}
         \mathfrak I_1&:=&   2\eta\mathbb{E}\int_{\tau_a}^{\tau_b}|(\mathcal{P}_1^m \mathcal{A}_r(v_m) - \mathcal{P}_1^n \mathcal{A}_r(v_n), A_0 v_{m,n})| ds, \\
	       \mathfrak I_2&:=& 2\mathbb{E} \int_{\tau_a}^{\tau_b} |(\mathcal{P}_1^m B_0(v_m,v_m) - \mathcal{P}_1^n B_0(v_n,v_n), A_0 v_{m,n})| ds, 
\end{eqnarray*} 
\begin{eqnarray*}
\mathfrak I_3&:=& 2\epsilon \, \mathbb{E} \int_{\tau_a}^{\tau_b} |(\mathcal{P}_1^m R_0(A_1 \phi_m, \phi_m) - \mathcal{P}_1^n R_0(A_1 \phi_n, \phi_n), A_0 v_{m,n})| ds, \\
\mathfrak I_4&:=&2\mathbb{E} \int_{\tau_a}^{\tau_b} \big|(	(\mathcal{P}_1^m-\mathcal{P}_1^n)z, A_0 v_{m,n} )\big| ds + \, \mathbb{E} \int_{\tau_a}^{\tau_b} \|\mathcal{P}_1^m  G_1(s,v_m) - \mathcal{P}_1^n G_1(s,v_n)\|_{\mathcal{L}_2(U_1,V)}^2 ds, \\
        \mathfrak I_5&:=&2\mathbb{E} \sup_{r \in [\tau_a, \tau_b]} \left|\sum_{k=1}^\infty \int_{\tau_a}^{r} (\mathcal{P}_1^m g_{1,k}(s,v_m) - \mathcal{P}_1^n g_{1,k}(s,v_n), A_0 v_{m,n}) d\beta^1_k(s)\right|.
\end{eqnarray*} 
           We next estimate each of the terms appearing on the right-hand side of equation \eqref{esiv1}.
          One can readily derive the estimate for $\mathfrak{I}_1$ when $r=1$, and so assume that $r > 1$.
	\begin{equation}\label{dec1}
		\begin{aligned}	
			\eta(\mathcal{P}_1^m \mathcal{A}_r(v_m) - \mathcal{P}_1^n \mathcal{A}_r(v_n), A_0 v_{m,n}) &= \eta\left((\mathcal{A}_r(v_m)-\mathcal{A}_r(v_n)) +(\mathcal{P}_1^m - \mathcal{P}_1^n) \mathcal{A}_r(v_n), A_0 v_{m,n}\right)
			:= I_1 + I_2
		\end{aligned}
	\end{equation}\label{dec2}
           The properties of $\mathcal P$, estimate \eqref{for1} and generalized  H\"older’s inequality lead to the following:
	\begin{equation*}
		\begin{aligned}
	|I_1| &
             = |\eta(|v_m|^{r-1}v_m - |v_n|^{r-1}v_n,  A_0 v_{m,n})|
             \leq  C_{c_r, \eta}\int_{\mathcal{O}} |v_{m,n}|\, (~|v_m|+|v_n|~)^{r-1}| A_0 v_{m,n}| \,dx\\
			 &\leq C_{c_r, \eta} \|v_{m,n}\|_{L^{s}_{\mathrm{div}}}\, \|\,|v_m|+|v_n|\,\|_{L^{k(r-1)}_{\mathrm{div}}}^{r-1}\|A_0 v_{m,n}\|_{L^2_{\mathrm{div}}}, \quad \, s, k >2, \quad \quad  \frac{1}{s} + \frac{1}{k} + \frac{1}{2} = 1.
			\end{aligned}
		\end{equation*}
        
			For $ d=2$, since $\mathbb{H}^1 \hookrightarrow \mathbb{L}^p$ for all $p \geq 2$, and given that $\frac{1}{s} + \frac{1}{k} = \frac{1}{2}$ with $s > 2$, we have $k = \frac{2s}{s-2} > 2$. This ensures both $\mathbb{H}^1 \hookrightarrow \mathbb{L}^s$ and $\mathbb{H}^1 \hookrightarrow \mathbb{L}^{k(r-1)}$ hold for any $r \in (1, \infty)$.
			For $ d=3,$   $\mathbb{H}^1 \hookrightarrow \mathbb{L}^p$ for all $p \in [2,6]$, and given that $\frac{1}{s} + \frac{1}{k} = \frac{1}{2}$ with $s \in (2,6]$, we have $k = \frac{2s}{s-2} \in [3,\infty)$ and $\mathbb{H}^1 \hookrightarrow \mathbb{L}^s$. The embedding $\mathbb{H}^1 \hookrightarrow \mathbb{L}^{k(r-1)}$ requires $k(r-1) \in [2,6]$, which yields the admissible range $r \in (1, 3]$.\\
			Therefore, for any $ r \in (1, \infty)$ when $d= 2$ and $ r \in (1,3]$ when $d= 3,$ we infer  that 
\begin{equation}\label{eststa1}
		\begin{aligned}
			|I_1| 	& \leq C_{c_r, \eta} \|v_{m,n}\| \,\|\,|v_m|+|v_n|\,\|^{r-1}\|A_0 v_{m,n}\|_{ L^2_{\mathrm{div}}}\\
			& \leq  C_{c_r, \eta, \nu_1}(~\|v_m\|^{2(r-1)} + \|v_n\|^{2(r-1)}) \|v_{m,n}\|^2+ \nu_1\|A_0 v_{m,n}\|^2_{ L^2_{\mathrm{div}}},
		\end{aligned}
	\end{equation}
		where $\nu_1>0.$  By employing H\"older's inequality, estimate \eqref{est1} and the boundedness of $\mathcal P,$ we get	
	\begin{equation*}
		\begin{aligned}	
			|I_2| & 
			\leq \eta \|\mathcal{Q}^n_1\mathcal{A}_r(v_n)\|_{ L^2_{\mathrm{div}}} \|A_0 v_{m,n}\|_{ L^2_{\mathrm{div}}}\\
			& \leq \nu_2\|A_0 v_{m,n}\|^2_{L^2_{\mathrm{div}}} + \frac{C_{\eta ,\nu_2} }{\lambda_n}\|\mathcal{Q}^n_1\mathcal{A}_r(v_n)\|^2
			 \leq \nu_2\|A_0 v_{m,n}\|^2_{L^2_{\mathrm{div}}} + \frac{ C_{\eta,\nu_2}}{\lambda_n}\|~|v_n|^{r-1}v_n~\|^2.
		\end{aligned}
	\end{equation*}
			By invoking \eqref{for2}, we get \(\|~|v_n|^{r-1}v_n~\|^2 \leq C_r\big\|\,|v_n|^{r-1} |\nabla v_n|\,\big\|^2_{L^2_{\mathrm{div}}}\).  Consequently,  $$|I_2| \leq \nu_2\|A_0 v_{m,n}\|^2_{L^2_{\mathrm{div}}} + \frac{C_{\eta,c_r,\nu_2}}{\lambda_n}\|v_n\|_{L^{2k(r-1)}_{\mathrm{div}}}^{2(r-1)} \|\nabla v_n\|_{L^{2s}_{\mathrm{div}}}^2, \ s,k>1, \ \frac{1}{s} + \frac{1}{k} = 1.$$ 
				Using $\mathbb{H}^1 \hookrightarrow \mathbb{L}^p$  with H\"older exponents satisfying $\frac{1}{s} + \frac{1}{k} = 1$, we obtain $r \in (1,\infty)\,( \text{in}\:~  d= 2)$ and  $r \in (1, 3] \,( \text{in}\:~  d= 3)$ hence, the embeddings 
				$\mathbb{H}^1 \hookrightarrow \mathbb{L}^{2s}$ and 
				$\mathbb{H}^1 \hookrightarrow \mathbb{L}^{2k(r-1)}$	follow. Thus, we have
		\begin{equation}\label{est2'}
		      |I_2| \leq \nu_2\|A_0 v_{m,n}\|^2_{L^2_{\mathrm{div}}} + \frac{C_{\eta,c_r,\nu_2}}{\lambda_n}\|v_n\|^{2(r-1)} \|A_0 v_n\|^2_{L^2_{\mathrm{div}}}.
	\end{equation}
	           The bound for the bilinear term in $\mathfrak J_3$ and the coupling term in $\mathfrak J_4$  follow from  Proposition 3.1 of \cite{Deugoue2021}:
\begin{equation}
		\begin{aligned}
			&(\mathcal{P}_1^m B_0(v_m, v_m) - \mathcal{P}_1^n B_0(v_n, v_n), A_0 v_{m,n})\\ 
			&\leq  \nu_3 \|A_0 v_{m,n}\|^2_{L^2_{\mathrm{div}}} + c_{\nu_3} \|v_{m,n}\|^2 \|A_0 v_m\|^2_{ L^2_{\mathrm{div}}} 
			+ c_{\nu_3} \|v_n\|^4 \, \|v_{m,n}\|^2 + \frac{c_{\nu_3}}{\lambda_n^{1/2}} \|v_n\|^2\|A_0 v_n\|^2_{L^2_{\mathrm{div}}}.
		\end{aligned}
	\end{equation} and
	\begin{eqnarray}\label{dec3}
		\lefteqn{	|\epsilon(\ \mathcal{P}_1^m R_0(A_1 \phi_m, \phi_m) - \ \mathcal{P}_1^n R_0(A_1 \phi_n, \phi_n), A_0 v_{m,n})|} \\
			&&\leq {\nu_4} \|A_0 v_{m,n}\|^2_{L^2_{\mathrm{div}}} 
		   + \epsilon_0 \|\phi_{m,n}\|^2_{H^4}
	       + c_{\nu_4,\,\epsilon_0,\epsilon}\Big( |A_1 \phi_m|^4_{ L^2} 
		    +  |A_1 \phi_n|^4_{ L^2} \Big) \| \phi_{m,n}\|^2_{H^2}
	    	+  \frac{c_{\nu_4, \, \epsilon}}{\lambda_n}\left(\,|A_1 \phi_n|^4_{ L^2} + \| \phi_n\|^2_{H^2} \|\phi_n\|^2_{H^4}\,\right).\nonumber
\end{eqnarray}
            Next, using the uniform Lipschitz continuity and linear growth of \(G_1\) (see [\ref{[A2]}] and Definition \ref{ULLF}), and estimate \eqref{est1}, we have
\begin{eqnarray} \label{ceqn}
	\lefteqn{\big|(	(\mathcal{P}_1^m-\mathcal{P}_1^n)z, A_0 v_{m,n} )\big|+\|\mathcal{P}_m^1 G_1(s, v_m) - \mathcal{P}_n^1 G_1(s, v_n)\|_{\mathcal{L}_2(U_1,V)}^2}\nonumber\\
    &&\leq C_{\nu_5}\|\mathcal{Q}_1^n\,z\|^2_{L^2_{\mathrm{div}}}+ {\nu_5}\|A_0 v_{m,n}\|^2_{L^2_{\mathrm{div}}}+c_{L_{V}} \|v_{m,n}\|^2 + \frac{c_{B_{D(A_0)}}}{\lambda_n}
	\left( 1+ \|A_0 v_n\|^2_{L^2_{\mathrm{div}}} \right).
\end{eqnarray}
		     The  estimate for the stochastic integral $\mathfrak J_5$ follows from that of a similar integral in \cite{Glatt-Holtz2009}: 
	\begin{equation}\label{exis8} 
		\begin{aligned}
			\mathfrak J_5 \leq \frac{1}{2} \mathbb{E} \sup_{s \in [\tau_a, \tau_b]} \| v_{m,n}(s) \|^2 
			+ c_{L_{V},B_{D(A_0)}} \mathbb{E} \int_{\tau_a}^{\tau_b} 
			\left( \| v_{m,n}(s) \|^2+ \frac{1}{\lambda_n}  + \frac{1}{\lambda_n} \| A_0 v_n(s) \|^2_{L^2_{\mathrm{div}}} 
			\right) ds.
		\end{aligned}
	\end{equation}
        	\textbf{Step 2 (Estimate of the phase parameter equation).} Taking the  \( L^2 \) inner product of \eqref{esic2}  with \( 2(A_1^2 \phi_{m,n}+ \phi_{m,n}) \), integrating by parts, using \eqref{esic4}, and
			repeating the steps of \eqref{esiv1} yield:
	\begin{equation}\label{esiphi1}
		\begin{aligned}
			&\mathbb{E} \sup_{t \in [\tau_a, \tau_b]} \| \phi_{m,n}(t)\|^2_{H^2} + 2 \epsilon\,\mathbb{E} \int_{\tau_a}^{\tau_b} \left( |A_1\phi_{m,n}|^2_{ L^2} + |A_1^2 \phi_{m,n}|^2_{ L^2} \right) ds \\
			&\leq \mathbb{E} \| \phi_{m,n}(\tau_a)\|_{H^2}^2 + 2\epsilon^{-1}\, \mathbb{E} \int_{\tau_a}^{\tau_b}  |(\mathcal{P}_2^m A_1 \psi'(\phi_m) - \mathcal{P}_2^n A_1 \psi'(\phi_n),\, \phi_{m,n} + A_1^2 \phi_{m,n})| ds \\
			&\qquad + 2\mathbb{E} \int_{\tau_a}^{\tau_b} |(\mathcal{P}_2^m B_1(v_m, \phi_m) - \mathcal{P}_2^n B_1(v_n, \phi_n),\, \phi_{m,n} + A_1^2 \phi_{m,n})| ds \\
	    	&\qquad + 2 \mathbb{E} \int_{\tau_a}^{\tau_b}|(\mathcal{P}_2^m(P\sigma_m- A- \alpha u)h(\phi_m) - \mathcal{P}_2^n(P\sigma_n- A- \alpha u)h(\phi_n), \,\phi_{m,n}+A_1^2 \phi_{m,n})| ds. 
	    	\end{aligned}
	\end{equation}
			We shall now estimate each of the terms of the right-hand side of \eqref{esiphi1}. First, we infer that
		\begin{equation}\label{dec5}
			\begin{aligned}
				& \epsilon^{-1}(\mathcal{P}_2^m A_1 \psi'(\phi_m) - \mathcal{P}_2^n A_1 \psi'(\phi_n), \, \phi_{m,n}+ A_1^2 \phi_{m,n}) \\
				&=   \epsilon^{-1}\left((A_1 \psi'(\phi_m) - A_1 \psi'(\phi_n)) + (\mathcal{P}_2^m - \mathcal{P}_2^n) A_1 \psi'(\phi_n),\, \phi_{m,n}+A_1^2 \phi_{m,n}\right) := J_{1} + J_{2}.
			\end{aligned}
		\end{equation}
				Let us estimate $J_1$ and $J_2.$ Note that in view of \eqref{norm3}, for any $\epsilon_1>0,$ we have
		\begin{equation}\label{est1'}
				|J_{1}|\, = \epsilon^{-1}\Big|\left(A_1 \psi'(\phi_m) - A_1 \psi'(\phi_n),\phi_{m,n}+A_1^2 \phi_{m,n}\right)\Big| \leq C_{\epsilon_1, \epsilon}|A_1 \psi'(\phi_m) - A_1 \psi'(\phi_n)|^2_{ L^2} +{\epsilon_1}\|\phi_{m,n}\|^2_{H^4}.
		\end{equation}
			Next we decompose the first term in \eqref{est1'} as follows:	  			
		\begin{align}
				&A_1 \psi'(\phi_m) - A_1 \psi'(\phi_n)  = -\psi'''(\phi_m) |\nabla \phi_m|^{2} +  \psi''(\phi_m)\,A_1\phi_m +\psi'''(\phi_n) |\nabla \phi_n|^{2} - \psi''(\phi_n)A_1\phi_n \label{estst2}\\
				&= -\psi'''(\phi_m)\left(|\nabla \phi_m|^2-|\nabla \phi_n|^2\right)- |\nabla\phi_n|^2\left(\psi'''(\phi_m)-\psi'''(\phi_n) \right) + \psi''(\phi_m)A_1\phi_{m,n}+ A_1\phi_n\left(\psi''(\phi_m)-\psi''(\phi_n) \right) \nonumber.	
			\end{align}
				With the aid of H\"older’s inequality, assumption \eqref{eq:ass2} and the Sobolev embeddings of \(H^2\) and \(H^1\) in \( L^{\infty}\) and \(L^4\), respectively, we have the following:
		\begin{align}
				\Big|\psi'''(\phi_m)\left(|\nabla \phi_m|^2- |\nabla \phi_n|^2\right)\Big|^2_{ L^2}
				&\leq \|\psi'''(\phi_m)\|^2_{L^{\infty}}\|\nabla \phi_m+\nabla \phi_n\|^2_{L^4}\|\nabla \phi_{m,n}\|^2_{L^4}\nonumber \\
				&\leq C_{\psi}\big( 1+ \|\phi_m\|^6_{L^{\infty}}\big)\|\nabla \phi_m+\nabla \phi_n\|^2_{H^1}\|\nabla \phi_{m,n}\|^2_{H^1}\nonumber\\
				&\leq C_{\psi}\big(1+\|\phi_{n}\|^4_{H^2}+\| \phi_{m}\|^4_{H^2} +\|\phi_{m}\|^{12}_{H^2} \big)\|\phi_{m,n}\|^2_{H^2},\label{est4}\\
		    	\Big|\;|\nabla\phi_n|^2\left(\psi'''(\phi_m)-\psi'''(\phi_n) \right)\Big|^2_{ L^2} &\leq C_{\psi}(1+ \|   \phi_{m}\|^4_{L^{\infty}}+\|\phi_{n}\|^4_{L^{\infty}})\|\nabla \phi_m\|^4_{L^4}\|\phi_{m,n}\|^2_{L^{\infty}}\nonumber\\
				&\leq C_{\psi}(1+\| \phi_{m}\|^8_{H^2}+\|\phi_{n}\|^8_{H^2})\|\phi_{m,n}\|^2_{H^2},\label{est5} \\
		   |\psi''(\phi_m)A_1\phi_{m,n}|^2_{ L^2} &\leq
				\|\psi''(\phi_m)\|^2_{L^{\infty}}|A_1\phi_{m,n}|^2_{ L^2}
				\leq C_{\psi}(1+ \|\phi_m\|^8_{H^2})\|\phi_{m,n}\|_{H^2}^2,\label{est6}\\
		    	|A_1\phi_n\left(\psi''(\phi_m)-\psi''(\phi_n) \right)|^2_{ L^2}
				&\leq C_{\psi}(1+ \|\phi_n\|^4_{H^2} + \|\phi_m\|^{12}_{H^2}+\|\phi_n\|^{12}_{H^2})\|\phi_{m,n}\|^2_{H^2}\label{est7}.
		\end{align}
		  		 Combining \eqref{est1'} and \eqref{estst2} with the estimates \eqref{est4}--\eqref{est7}, we derive the following bound for \(J_1\) that holds for both \(d=2\) and \(d=3\): 
		\begin{equation}\label{est7'}
			   |J_{1}|\leq C_{\epsilon_1, \epsilon}\mathbb{Q}(\|\phi_n\|_{H^2},\|\phi_m\|_{H^2})\|\phi_{m,n}\|^2_{H^2}+\epsilon_1\|\phi_{m,n}\|^2_{ H^4},
		\end{equation}
		 		where \(\mathbb{Q}\) is a monotone non-decreasing function of the parameters \(\|\phi_n\|_{H^2}\) and \(\|\phi_m\|_{H^2}\). Invoking Lemma \ref{norm_esti} and Young's inequality, we obtain
	\begin{equation}\label{est8}
			\begin{aligned}	
				|J_{2}| & 
				\leq \frac{c_{\epsilon_2,\epsilon}}{\beta_n}\|\mathcal{Q}_2^n A_1 \psi'(\phi_n)\|^2_{H^1} + {\epsilon_2}\|\phi_{m,n}\|^2_{ H^4} \;\leq \frac{c_{\epsilon_2,\epsilon}}{\beta_n}\|A_1 \psi'(\phi_n)\|^2_{H^1} + {\epsilon_2}\|\phi_{m,n}\|^2_{ H^4}.
			\end{aligned}
	\end{equation}
			Since $A_1 \psi'(\phi_n) = -\psi'''(\phi_n) |\nabla \phi_n|^{2} +  \psi''(\phi_n)A_1\phi_n,$ we derive
	\begin{equation}\label{est9}
			\begin{aligned}	
			     \|A_1 \psi'(\phi_n)\|^2_{H^1} &\leq \Big|\,\psi'''(\phi_n) |\nabla \phi_n|^{2}\,\Big|^2_{ L^2} +  \Big|\psi''(\phi_n)A_1\phi_n\Big|^2_{ L^2} + \Big|\,\psi''''(\phi_n) |\nabla \phi_n|^{3}\,\Big|^2_{ L^2}\\
				& \quad +2\Big|\,\psi'''(\phi_n) \,|\nabla \phi_n|\, A_1\phi_n\,\Big|^2_{ L^2}+\Big|\psi''(\phi_n)\, \nabla \Delta\phi_n\Big|^2_{ L^2}:=\sum_{i=3}^7J_i, 
			\end{aligned}
	\end{equation}
			where we have exploited the fact that \(\psi \in C^4(\mathbb{R})\). Again, using assumption \eqref{eq:ass2}, we have the following estimates of the above decomposition, which holds for \(d=2, 3\):  
		\begin{align}	
				J_3 + J_4
				& \leq \|\psi'''(\phi_n)\|_{L^{\infty}}^2\|\nabla \phi_n\|^4_{L^4} +\|\psi''(\phi_n)\|^2_{L^{\infty}}|A_1\phi_n|^2_{L^2}\nonumber \\
				&\leq C_{\psi}(1+\|\phi_n\|^6_{H^2})\|\phi_n\|^4_{H^2}+ C_{\psi}(1+\|\phi_n\|^8_{H^2})|A_1\phi_n|^2_{ L^2},\label{est10}\\
	     		J_5+ J_6 &\leq 
				C_{\psi}(1+\|\phi_n\|^4_{H^2})\|\phi\|^6_{H^2} +C (1+\|\phi_n\|^6_{H^2})\|\phi_n\|^2_{H^2}\|\phi_n\|^2_{H^3}.\label{est11}
	\end{align}
	       Finally, from \eqref{est8}-\eqref{est11}, we have
	\begin{equation}\label{est13} 
			|J_{2}|  \leq	\frac{c_{\epsilon_2,\epsilon,C_{\psi}}}{\beta_n}\left(\|\phi_n\|^2_{H^2} +\|\phi_n\|^4_{H^2}+\|\phi_n\|^6_{H^2}+  \|\phi_n\|^{10}_{H^2} +\|\phi_n\|^2_{H^3} + \|\phi_n\|^2_{H^2}\|\phi_n\|_{H^3}^2+\|\phi_n\|^8_{H^2}\|\phi_n\|_{H^3}^2\right) + {\epsilon_2}\|\phi_{m,n}\|^2_{H^4}.
	\end{equation}
	      	The estimate for the trilinear term  can be obtained from a similar integral in \cite{Deugoue2021}: 
	  \begin{equation}\label{exis5}
	     	\begin{aligned}
	     	 &|(\mathcal{P}_2^m B_1(v_m, \phi_m) - \mathcal{P}_2^n B_1(v_n, \phi_n), \phi_{m,n}+ A_1^2 \phi_{m,n})|\\
	     	 &\leq{\epsilon_3} \| \phi_{m,n}\|^2_{H^4} + c_{\epsilon_3}  |A_1 \phi_m|^{2}_{ L^2}\|v_{m,n}\|^2  + c_{\epsilon_3}  \|v_n\|^2 \| \phi_{m,n}\|^2_{H^2} + \frac{c_{\epsilon_3}}{\beta_n}\|A_0 v_n\|^2_{L^2_{\mathrm{div}}} \| \phi_n\|^2_{H^2}.
	      \end{aligned}
	\end{equation}
	 	   	 	 We rewrite the last term of \eqref{esiphi1} as follows:
		\begin{align}
	 				&(\mathcal{P}_2^m(P\sigma_m-A- \alpha u)h(\phi_m) - \mathcal{P}_2^n(P\sigma_n-A- \alpha u)h(\phi_n), \phi_{m,n}+A_1^2 \phi_{m,n})\nonumber\\
	 				&\quad= \Big((P\sigma_m- A- \alpha u)(h(\phi_m)-h(\phi_n)) +P\sigma_{m,n}h(\phi_n) \nonumber\\
                    &\hspace{1in}+ (\mathcal{P}_2^m - \mathcal{P}_2^n)(P\sigma_n- A- \alpha u)h(\phi_n),\phi_{m,n}+A_1^2 \phi_{m,n}\Big)
	 				:= J_{8} + J_{9} + J_{10}.\label{sigmaphi}
	 		\end{align}
				Using H\"older's and Young's inequalities, assumptions on \(u\) and \(h\) (see [\ref{[A3]}], [\ref{[A4]}]), we obtain the following for \( d= 2,3\):
\begin{align}	 				
|J_{8}|  
	 				& \leq  c_{\epsilon_4, L_h}|(P\sigma_m- A- \alpha u) \phi_{m,n}|^2_{ L^2} + {\epsilon_4}\|\phi_{m,n}\|^2_{ H^4} \leq  c_{\epsilon_4, L_h}(|P\sigma_m \phi_{m,n}|^2_{ L^2} + |(A+ \alpha u) \phi_{m,n}|^2_{ L^2} )+ {\epsilon_4}\|\phi_{m,n}\|^2_{ H^4}\nonumber\\
	 				& \quad\quad \leq  c_{\epsilon_4, L_h, P}\|\sigma_m\|^2_{H^1}\|\phi_{m,n}\|^2_{H^1} + c_{\epsilon_4, L_h, A, \alpha, }|\phi_{m,n}|^2_{ L^2}+ {\epsilon_4}\|\phi_{m,n}\|^2_{ H^4},\label{est20} \\
        |J_{10}| 
	 				& \leq c_{\epsilon_6}|\mathcal{Q}_2^n(P\sigma_n- A- \alpha u)h(\phi_n)|^2_{ L^2}+ {\epsilon_6}\|\phi_{m,n}\|^2_{ H^4} \leq \frac{c_{\epsilon_6}}{\beta_n}\|\mathcal{Q}_2^n(P\sigma_n- A- \alpha u)h(\phi_n)\|^2_{H^1} +  {\epsilon_6}\|\phi_{m,n}\|^2_{ H^4} \nonumber\\
	 				& \quad \quad \leq \frac{c_{\epsilon_6}}{\beta_n}\Big(|P\sigma_n- A- \alpha u|^2_{ L^2} + |A_1^{1/2}(P\sigma_n- A- \alpha u)h(\phi_n)|^2_{ L^2}\Big)+ {\epsilon_6}\|\phi_{m,n}\|^2_{ H^4} \nonumber\\
	 				&\quad \quad \leq \frac{c_{\epsilon_6, P, L_h, A, \alpha, |\mathcal{O}|}}{\beta_n}\Big(1 + \|\sigma_n\|^2_{H^1} + |A_1^{1/2}\phi_n|^2_{ L^2} + \|u\|^2_{H^1} +\|\sigma_n\|^2_{H^1}\|\phi_n\|^2_{H^2} \Big)+ {\epsilon_6}\|\phi_{m,n}\|^2_{ H^4}. \label{est22}	
	 \end{align}
                                
\noindent\textbf{Step 3 (Estimate of the nutrient concentration equation).} Finally, when the It\^o formula is applied to the process \(\|\sigma_{m,n}(\cdot)\|_{H^1}^2\) as in the previous steps, and using $\eqref{esic3} $, we obtain
\begin{equation}\label{esisigma1}
    \begin{aligned}
		 		&\mathbb{E}\, \sup_{t \in [\tau_a, \tau_b]} \| \sigma_{m,n}(t)\|_{H^1}^2 + 2 \mathbb{E} \int_{\tau_a}^{\tau_b}\left(|A^{1/2}_1\sigma_{m,n}|^2_{ L^2} + |A_1\sigma_{m,n}|^2_{ L^2}\right) ds\\
				&\leq \mathbb{E} \,\| \sigma_{m,n}(\tau_a)\|_{H^1}^2 + 2 \mathbb{E} \int_{\tau_a}^{\tau_b}|\left(\mathcal{P}_2^m B_1(v_m, \sigma_m)- \mathcal{P}_2^n B_1(v_n, \sigma_n),\, \sigma_{m,n}+ A_1 \sigma_{m,n}\right)|ds\\
				& \quad+ 2c \mathbb{E} \int_{\tau_a}^{\tau_b}|\left(\mathcal{P}_2^m \sigma_m h(\phi_m)- \mathcal{P}_2^n \sigma_n h(\phi_n),\,\sigma_{m,n}+A_1 \sigma_{m,n}\right)|ds\\
                & \quad+ 2b \mathbb{E} \int_{\tau_a}^{\tau_b}|\left(\sigma_{m,n}, \sigma_{m,n}+ A_1 \sigma_{m,n}\right )|ds + 2b\mathbb{E} \int_{\tau_a}^{\tau_b}|((\mathcal{P}_2^m-\mathcal{P}_2^n)w,\,  \sigma_{m,n} + A_1 \sigma_{m,n})| ds\\
				&\quad+ \, \mathbb{E} \int_{\tau_a}^{\tau_b} \|\mathcal{P}_1^m  G_2(s,\sigma_m) - \mathcal{P}_2^n G_2(s,\sigma_n)\|_{\mathcal{L}_2(U_2,H^1)}^2 ds \\
				& \quad+ 2\mathbb{E} \sup_{r \in [\tau_a, \tau_b]} 	\left|\sum_{k=1}^\infty \int_{\tau_a}^{r} (\mathcal{P}_2^m g_{2,k}(s,\sigma_m) - \mathcal{P}_2^n g_{2,k}(s,\sigma_n), \sigma_{m,n} + A_1 \sigma_{m,n}) d\beta^2_k(s)\right| :=\sum_{i=1}^6 \mathfrak J_k.
		\end{aligned}
	\end{equation}
			Let us first rewrite the terms in the integrals $\mathfrak J_2$ and $\mathfrak J_3$ as follows and estimate them one by one:
\begin{equation}\label{dec6}
	\begin{aligned}
     		&(\mathcal{P}_2^m B_1(v_m, \sigma_m)- \mathcal{P}_2^n B_1(v_n, \sigma_n), \sigma_{m,n}+ A_1 \sigma_{m,n})\\
    		&= (B_1(v_{m,n}, \sigma_m) + B_1(v_n, \sigma_{m,n}) + (\mathcal{P}_2^m - \mathcal{P}_2^n) B_1(v_n, \sigma_n), \sigma_{m,n}+A_1 \sigma_{m,n}):= K_{1} + K_{2} + K_{3},
	\end{aligned}
\end{equation}	
\begin{equation}\label{dec7}
	\begin{aligned}
			&(\mathcal{P}_2^m \sigma_m h(\phi_m)- \mathcal{P}_2^n \sigma_n h(\phi_n),\sigma_{m,n}+A_1 \sigma_{m,n})\\
			&=(\sigma_m(h(\phi_m)-h(\phi_n)) + h(\phi_n)\sigma_{m,n}+ (\mathcal{P}_2^m - \mathcal{P}_2^n)\sigma_n h(\phi_n), \sigma_{m,n}+A_1 \sigma_{m,n}):= K_{4} + K_{5} + K_{6}.
	\end{aligned}
\end{equation}
    		Next, we estimate the nonlinear contributions of \(K_{1}, K_{2},\) and \(K_{3}\) for both \(d= 2,3\). Utilizing H\"older's and Young's inequalities, together with the embedding \(H^{1} \hookrightarrow L^{4}\), we derive
 \begin{equation}\label{est14}
	  		\begin{aligned}	
	  				| K_{1}|  
	  				& \leq |(B_1(v_{m,n}, \sigma_m), \sigma_{m,n})| + |(B_1(v_{m,n}, \sigma_m),A_1 \sigma_{m,n} )|\\
	  				& \leq \|v_{m,n}\|_{L^4_{\mathrm{div}}} |A_1^{1/2} \sigma_m|_{ L^2}\|\sigma_{m,n}\|_{L^4} + \|v_{m,n}\|_{L^4_{\mathrm{div}}} \|A_1^{1/2} \sigma_m\|_{L^4}|A_1\sigma_{m,n}|_{ L^2}\\
	  				& \leq c_{\gamma_1}|A^{1/2}_1\sigma_m|^2_{ L^2}\|\sigma_{m,n}\|^2_{H^1} + {\gamma_1} \|v_{m,n}\|^2 + c_{\gamma_1}\|\sigma_m\|^2_{H^2} \|v_{m,n}\|^2 + {\gamma_1}|A_1\sigma_{m,n}|^2_{L^2}.
	  	\end{aligned}
	  \end{equation}
			Note that $(B_1(v_n, \sigma_{m,n}) , \sigma_{m,n})=0$ by \eqref{b0_1}. By invoking H\"older's, Gagliardo-Nirenberg's (\ref{GNI1}), and Young's inequalities, we estimate $K_2$, as follows:
    	\begin{equation}\label{est15}
	  		\begin{aligned}	
	  	            | K_{2}| & 
	  				& \leq \|v_{n}\|\,|A^{1/2}_1 \sigma_{m,n}|_{ L^2}^{1/2}\|\sigma_{m,n}\|_{H^2}^{3/2}
	  				 \leq c_{\gamma_2}\| v_n\|^4\,\|\sigma_{m,n}\|^2_{H^1} + {\gamma_2} \|\sigma_{m,n}\|^2_{H^2}.
	  		\end{aligned}
	 	\end{equation}
				Recalling \eqref{est1} and applying the Sobolev embedding results along with H\"older’s, Gagliardo-Nirenberg and Agmon’s (see \ref{l1},\ref{l2}) inequalities, we obtain the following estimate for \(K_{3}\): 
                
	  	\begin{equation}\label{est16}
	  	\begin{aligned}	
	  	    	   | K_{3}| 
	  				&\leq |\mathcal{Q}_2^n B_1(v_n, \sigma_n)|_{ L^2} |\sigma_{m,n}+A_1 \sigma_{m,n}|_{ L^2} \leq \frac{c_{\gamma_3}}{\beta_n}\|\mathcal{Q}_2^n B_1(v_n, \sigma_n)\|_{H^1}^2 + {\gamma_3}\|\sigma_{m,n}\|^2_{ H^2} \\
	  				& \quad \ \ \leq \frac{c_{\gamma_3}}{\beta_n}\left(|A^{1/2}_1 \sigma_n|^2_{ L^2}\|A_0 v_n\|^2_{ L^2_{\mathrm{div}}} + \|A^{1/2}_0 v_n\|_{ L^2_{\mathrm{div}}}\|A_0 v_n\|_{ L^2_{\mathrm{div}}}\|\sigma_n\|^2_{H^2}\right )  + {\gamma_3}\|\sigma_{m,n}\|^2_{ H^2} \\
	  				& \quad \ \ \leq \frac{c_{\gamma_3}}{\beta_n}\left(|A^{1/2}_1 \sigma_n|^2_{ L^2}\|A_0 v_n\|^2_{L^2_{\mathrm{div}}} + \lambda_n^{1/2}\|v_n\|^2\|\sigma_n\|^2_{H^2} \right) + {\gamma_3}\|\sigma_{m,n}\|^2_{ H^2}.
	  		\end{aligned}
	\end{equation}
 To estimate  \( K_{4},  K_{5}\) and \( K_{6}\), we invoke Assumption [\ref{[A3]}] for $h,$ Lemma \ref{norm_esti}, together with H\"older’s and Young’s inequalities, leading to the following bounds (valid for both \(d=2, 3\)): 
\begin{equation}\label{est17}  		
	\begin{aligned}	
		          |K_{4}| 
				 & \leq \int_{\mathcal{O}}|\sigma_m(h(\phi_m)-h(\phi_n)) (\sigma_{m,n}+A_1 \sigma_{m,n})|\,dx\\
				 & \leq c_{L_h}\int_{\mathcal{O}}|\sigma_m|\,|\phi_{m,n}|\,|\sigma_{m,n}+A_1 \sigma_{m,n}|\,dx
				 \leq c_{L_h, \gamma_4}\|\sigma_m\|^2_{H^1}\|\phi_{m,n}\|^2_{H^1} + \gamma_4\|\sigma_{m,n}\|^2_{H^2}, 
		\end{aligned}
\end{equation}
		and 
	$|K_{6}| 
				\leq \frac{c_{\gamma_6}}{\beta_n}\left(\|\sigma_n\|_{H^1}^2 + c_{L_h}| A^{1/2}_1\phi_n|^2_{ L^2}\|\sigma_n\|^2_{H^2} \right) + {\gamma_6}\|\sigma_{m,n}\|^2_{H^2}.
	$
 A similar estimate also holds for $K_5$. 	For the term in $\mathfrak J_5,$	using Assumption [\ref{[A2]}]  for \(G_2\), together with Lemma \ref{norm_esti}, we deduce
	  \begin{eqnarray}\label{exis6}
		&\|\mathcal{P}_m^2 G_2(s, \sigma_m) - \mathcal{P}_n^2 G_2(s, \sigma_n)\|^2_{\mathcal{L}_2(U_2, H^1)} 
			\leq  \| G_2(s, \sigma_m) - G_2(s, \sigma_n)\|^2_{\mathcal{L}_2(U_2, H^1)} +  \| \mathcal{Q}^n_2G_2(s, \sigma_n)\|^2_{\mathcal{L}_2(U_2, H^1)} \nonumber\\
			&\leq c_{L_{H^1}} \|\sigma_{m,n}\|^2_{H^1} + \frac{1}{\beta_n} \|\mathcal{Q}^n_2 G_2(s, \sigma_n)\|^2_{\mathcal{L}_2(U_2, H^2)} \leq c_{L_{H^1}} \|\sigma_{m,n}\|^2_{H^1} + \frac{c_{B_{H^2}}}{\beta_n}
			\big(1 + \| \sigma_n\|^2_{H^2} \big).
	\end{eqnarray}
	We invoke the Burkholder–Davis–Gundy inequality together with \eqref{exis6} to obtain          
\begin{equation}	
	\begin{aligned}
    \mathfrak J_6	&= 2 \mathbb{E} \sup_{r \in [\tau_a, \tau_b]} \left| \sum_{k=1}^\infty 
		\int_{\tau_a}^{r} \left( \big(\mathcal{P}_2^m g_{2,k}(s, \sigma_m) - \mathcal{P}_2^n g_{2,k}(s, \sigma_n),\sigma_{m,n}(s)\big) d\beta^2_k(s)\right.\right.\\
		&\left.\left.\quad + \big( \mathcal{P}_2^m g_{2,k}(s, \sigma_m) - \mathcal{P}_2^n g_{2,k}(s, \sigma_n), A_1\sigma_{m,n}(s) \big) d\beta^2_k(s)\right)\right| \\
	& \leq C \, \mathbb{E} \left( \int_{\tau_a}^{\tau_b} 
		| \sigma_{m,n}(s) |^2_{ L^2} \, \| \mathcal{P}_2^m G_2(s,\sigma_m) - \mathcal{P}_2^n G_2(s,\sigma_n) \|^2_{\mathcal{L}_2(U_2,L^2)} \, ds \right)^{1/2} \\
		&  \quad +C \, \mathbb{E} \left( \int_{\tau_a}^{\tau_b} 
		|A_1^{1/2} \sigma_{m,n}(s) |^2_{ L^2} \, \|A_1^{1/2}\big( \mathcal{P}_2^m G_2(s,\sigma_m) - \mathcal{P}_2^n G_2(s,\sigma_n)\big) \|^2_{\mathcal{L}_2(U_2,L^2)} \, ds \right)^{1/2} \\
		& \leq\mathbb{E}\left(\frac{1}{2}\sup_{s \in [\tau_a, \tau_b]} \| \sigma_{m,n}(s) \|^2_{H^1} + C\int_{\tau_a}^{\tau_b}
		\| \mathcal{P}_2^m G_2(s,\sigma_m) - \mathcal{P}_2^n G_2(s,\sigma_n)\|^2_{\mathcal{L}_2(U_2,H^1)} \, ds \right)\\
		&\leq \frac{1}{2} \mathbb{E} \sup_{s \in [\tau_a, \tau_b]} \| \sigma_{m,n}(s) \|^2_{H^1} 
		+ c_{L_{H^1}, B_{H^2}} \mathbb{E} \int_{\tau_a}^{\tau_b} 
		\left( \| \sigma_{m,n}(s) \|^2_{H^1} + \frac{1}{\beta_n}  + \frac{1}{\beta_n} \|\sigma_n(s) \|^2_{H^2} 
		\right) ds. \label{sinoise}	
   \end{aligned}
\end{equation} 
			\textbf{Step 4.} 
            Choose the absorption parameters $\nu_i$ ($i=1,\ldots,5$), $\epsilon_i$ ($i=0,\ldots,6$), and $\gamma_i$ ($i=1,\ldots,7$) small enough to ensure $\sum_{i=1}^{5} \nu_i < \nu$, $\sum_{i=0}^{6} \epsilon_i < \epsilon$, and $\sum_{i=1}^{7} \gamma_i < 1$ such that the original coefficients on the left side retain at least half their strength.\\
            Thus, adding  \eqref{esiv1},\eqref{esiphi1}, and \eqref{esisigma1} together, incorporating the corresponding estimates for each right-hand side provided in \eqref{dec1}-\eqref{exis8}, \eqref{dec5}-\eqref{est22}, and \eqref{dec6}-\eqref{sinoise},  we arrive at the following:
	{\footnotesize	\begin{eqnarray}
		\lefteqn{\mathbb{E} \sup_{t \in [\tau_a, \tau_b]} \|(v_{m,n}, \phi_{m,n}, \sigma_{m,n})(t)\|_{\mathcal{V}}^2 +    \mathbb{E} \int_{\tau_a}^{\tau_b}\left( \nu \|A_0 v_{m,n}\|^2_{ L^2_{\mathrm{div}}}+ \epsilon \| \phi_{m,n}\|^2_{H^4} + \|\sigma_{m,n}\|^2_{H^2} \right)ds}\nonumber\\
			&&\leq \mathbb{E} \|(v_{m,n}, \phi_{m,n}, \sigma_{m,n})(\tau_a)\|_{\mathcal{V}}^2 + C_1\Bigg(\mathbb{E}\int_{\tau_a}^{\tau_b}\left[1+ \|\sigma_m\|^2_{H^1} + \| v_n\|^4 \right]\|\sigma_{m,n}\|^2_{H^1} \,ds\nonumber\\
			&&+ \mathbb{E}\int_{\tau_a}^{\tau_b}\left[ 1+\|v_m\|^{2(r-1)} + \|v_n\|^{2(r-1)}+ \|v_n\|^4 + \|A_0 v_m\|^2_{L^2_{\mathrm{div}}} +\|\sigma_m\|^2_{H^2}  +|A_1 \phi_m|^2_{ L^2} \right]\|v_{m,n}\|^2 \,ds\nonumber\\
			&&+ \mathbb{E}\int_{\tau_a}^{\tau_b}\left[1+|A_1 \phi_m|^4_{ L^2} 
			+\|v_n\|^2  +|A_1 \phi_n|^4_{ L^2} +\|\sigma_m\|^2_{H^1}+ \mathbb{Q}(\|\phi_n\|_{H^2},\|\phi_m\|_{H^2}) \right]\|\phi_{m,n}\|^2_{H^2} \,ds\nonumber\\
			&&+ \mathbb{E}\int_{\tau_a}^{\tau_b}\left[\lambda_n^{-1}+ \lambda_n^{-1/2}\|v_n\|^2 +  \lambda_n^{-1}\|v_n\|^2 +\lambda_n^{-1}\|v_n\|^{2(r-1)} + \beta_n^{-1}\|\phi_n\|^2_{H^2} + \beta_n^{-1}\|\sigma_n\|^2_{H^1}\right]\|A_0v_{n}\|^2_{L^2_{\mathrm{div}}} \,ds\nonumber\\
			&&+ \mathbb{E}\int_{\tau_a}^{\tau_b}\left[\beta_n^{-1} + \lambda_n^{1/2}\beta_n^{-1}\|v_n\|^2 +\beta_n^{-1}\|\phi_n\|^2_{H^1} +\beta_n^{-1}\|\phi_n\|^2_{H^2}\right]\|\sigma_{n}\|^2_{H^2} \,ds\nonumber\\
            &&+ \mathbb{E}\int_{\tau_a}^{\tau_b}\beta_n^{-1}\left[1+\|\phi_n\|^2_{H^2} +\|\phi_n\|^4_{H^2}+\|\phi_n\|^6_{H^2}+  \|\phi_n\|^{10}_{H^2} +\|\phi_n\|^2_{H^3} \right]\,ds \label{comb1}\\
			&& + \mathbb{E}\int_{\tau_a}^{\tau_b}\left[ \lambda_n^{-1}\|\phi_n\|^2_{H^2} + \beta_n^{-1}\|\phi_n\|^2_{H^2} + \beta_n^{-1}\|\phi_n\|^8_{H^2} \right]\|\phi_{n}\|^2_{H^4} \,ds + \mathbb{E}\int_{\tau_a}^{\tau_b}\left[ \lambda_n^{-1} + \beta_n^{-1} +\|\mathcal{Q}^n_1 z\|^2_{L^2_{\mathrm{div}}} + \beta_n^{-1} \| u\|^2_{H^1}+|\mathcal{Q}^n_2 w|^2_{ L^2} \right]\,ds\nonumber\Bigg),
	\end{eqnarray}  }
			 where the constant $C_{1}$ depends on the standard parameter set $\Sigma_0 := \{\{\nu_i\}_{i=1}^5,\{\epsilon_i\}_{i=0}^6, \{\gamma_i\}_{i=1}^7, \, \alpha, L_{D(A_0)},$ $L_{V},L_{H^1},L_{H^2},L_{h},B_{D(A_0)},B_{H^2},C_r,\eta, C_{\psi}, |\mathcal{O}|, P, A, b, c\}$ but independent of  $m, n$ and stopping times $\tau_a, \tau_b$.

	Since \( \tau_a, \tau_b \in \mathcal{T}^{T,M}_{m,n} := \mathcal{T}^{T,M}_m \cap \mathcal{T}^{T,M}_n \) (see \eqref{stop}), we conclude that \( \| (v_m, \phi_m, \sigma_m)\|_{\mathcal{V}}, \|(v_n, \phi_n,\sigma_n)\|_{\mathcal{V}} \leq \widetilde{K} + M \) over the relevant time intervals, and therefore
		\begin{align*}
				&\mathbb{E} \sup_{t \in [\tau_a, \tau_b]} \underbrace{\|(v_{m,n}, \phi_{m,n}, \sigma_{m,n})(t)\|_{\mathcal{V}}^2}_{:=X(t)}+    \mathbb{E} \int_{\tau_a}^{\tau_b}\underbrace{\left( \nu \|A_0 v_{m,n}(s)\|^2_{L^2_{\mathrm{div}}}+ \epsilon \| \phi_{m,n}(s)\|^2_{H^4} + \|\sigma_{m,n}(s)\|^2_{H^2} \right)}_{:=Y(s)}ds\\
				&\leq \mathbb{E} \|(v_{m,n}, \phi_{m,n}, \sigma_{m,n})(\tau_a)\|_{\mathcal{V}}^2+ C_2\mathbb{E}\int_{\tau_a}^{\tau_b}\|(v_{m,n}, \phi_{m,n}, \sigma_{m,n})(s)\|_{\mathcal{V}}^2\underbrace{\left(1+\|A_0 v_m\|^2_{L^2_{\mathrm{div}}} +\|\sigma_m\|^2_{H^2} \right)}_{:=R(s)}ds \\
				&+ C_2\mathbb{E}\int_{\tau_a}^{\tau_b}\underbrace{\left( \left(\lambda_n^{-1/2}+\lambda_n^{-1}  +\beta_n^{-1} + \lambda_n^{1/2}\beta_n^{-1} \right)\left( 1+\|A_0 v_n\|^2_{L^2_{\mathrm{div}}} +\|\phi_{n}\|^2_{H^4} + \|\sigma_{n}\|^2_{H^2} \right)+ S_n(s)\right)}_{:=Z(s)}\,ds,
		\end{align*}
					where $S_n(t)=\|\mathcal{Q}^n_1 z(t)\|^2_{L^2_{\mathrm{div}}} + \beta_n^{-1} \| u(t)\|^2_{H^1}+|\mathcal{Q}^n_2 w(t)|^2_{ L^2},$ the constant $C_{2}$ depends on $\Sigma_0 \cup \{\widetilde{K}, M\}$ with modified Sobolev constants but is independent of \( \tau_a \) and \( \tau_b \). Furthermore,
					$\int_0^{\tau} R(s) ds \leq C \ \text{a.s.},$
			where C depends on \(\widetilde{K}, M ,\nu\) and \(T.\)
	Now utilizing the stochastic Gronwall lemma (see Lemma \ref{l3}) with $X,Y,Z$ and $R$ above,	we get
	 \begin{eqnarray}\label{fiex}
	 \mathbb{E} \sup_{t \in [0, \tau]} X(t) &+&    \mathbb{E} \int_{0}^{\tau}Y(s)ds \leq  C_3 \mathbb{E}\|(v_{m,n}, \phi_{m,n}, \sigma_{m,n})(0)\|_{\mathcal{V}}^2  \\
	 		&&+ C_3\left(\lambda_n^{-1/2}+\lambda_n^{-1} + \beta_n^{-1} + \lambda_n^{1/2}\beta_n^{-1}+ \mathbb{E}\int_{0}^{T}\left[ \|\mathcal{Q}^n_1 z\|^2_{L^2_{\mathrm{div}}} + \beta_n^{-1} \| u\|^2_{H^1}
	 		+|\mathcal{Q}^n_2 w|^2_{ L^2}  \right]ds\right), \nonumber
 \end{eqnarray}
 
			where $C_{3}$ depends on $\Sigma_0 \cup \{\widetilde{K}, M, T\}$ but is independent of mesh parameters $m,n$ and the choice of the stopping time  $\tau \in \mathcal{T}^{T,M}_{m,n}$.
	 		Hence, \eqref{fiex} (and see Remark \ref{egv})  implies \eqref{first} upon taking the supremum over \( \mathcal{T}^{T,M}_{m,n} \), followed by the limit as \( m > n \) and \( n \to \infty \). \\
            Next, we prove the convergence (ii). For any fixed \( \tau \in \mathcal{T}^{T,M}_n \) and \( \zeta > 0 \),  as in part (i), we obtain the following estimate (see \eqref{comb1}) for $d=2,3$, where $ r \in [1, \infty)$ when $d= 2$ and $ r \in [1,3]$ when $d= 3$ :
	\begin{equation}\label{ex32}
	 	\begin{aligned}
	 			& \sup_{t \in [0, \tau \wedge \zeta]} \|(v_{n}, \phi_{n}, \sigma_{n})(t)\|_{\mathcal{V}}^2 +  \int_{0}^{\tau \wedge \zeta}\Big( \nu \|A_0 v_{n}\|^2_{L^2_{\mathrm{div}}}+ \epsilon \| \phi_{n}\|^2_{H^4} + \|\sigma_{n}\|^2_{H^2} \Big)ds\\&\qquad\leq \|(v_{n}^{(0)}, \phi_{n}^{(0)}, \sigma_{n}^{(0)})\|_{\mathcal{V}}^2 + C_1 \int_{0}^{\tau \wedge \zeta}\Big(\,\|v_n\|^2+ \|v_n\|^{2r}+ \|v_n\|^6 + \|v_n\|^2|A_1^{1/2} \sigma_n|^4_{ L^2}+ \|\phi_n \|^2_{H^2}\\&\qquad\quad+ \|\phi_n \|^4_{H^2}+  \|\phi_n \|^6_{H^2}+ \|\phi_n \|^{10}_{H^2}+  \|\sigma_n \|^2_{H^1} \Big) \,ds + C_2 \int_{0}^{\tau \wedge \zeta}\left(\,\|z\|^2_{ L^2_{\mathrm{div}}} +|u|^2_{ L^2} +|w|^2_{ L^2} \right) ds\\
	 			&\qquad\quad+ 2\sup_{r \in [0, \tau \wedge \zeta]} \left|\int_{0}^{r}\sum_{k=1}^\infty ( g_{1,k}(t,v_n), A_0 v_{n}) d\beta^1_k(t)+ \int_{0}^{r}\sum_{k=1}^\infty (g_{2,k}(t,\sigma_n) ,\, \sigma_{n}+ A_1\sigma_{n} ) d\beta^2_k(t) \,\right|,
	\end{aligned}
	 \end{equation}  
	where  \( C_1 := C_1 \Big(\nu,\epsilon,\, C_{\psi}, \eta,\alpha,b,c,P, B_{V},B_{H^1} \Big)\) and  \(C_2 := C_2\left(\nu,\epsilon,|\mathcal{O}|, A\right)\). 
	From \eqref{ex32} we can see if,
	\begin{align*}
		I_1(\Omega) &\equiv \Bigg\{ \omega \in \Omega : 
		\sup_{s \in [0, \tau \wedge \zeta]} \|(v_n, \phi_n, \sigma_n)\|_{\mathcal{V}}^2  
		+ \int_0^{\tau \wedge \zeta} \left( \nu \|A_0 v_n\|^2_{ L^2_{\mathrm{div}}} 
		+ \epsilon\|\phi_n\|^2_{H^4}+ \|\sigma_{n}\|^2_{H^2}  \right)\,ds \\
		&\quad > \|(v_{n}^{(0)}, \phi_{n}^{(0)},\sigma_{n}^{(0)})\|_{\mathcal{V}}^2 + (M - 1)^2 \Bigg\}, \\
		I_2(\Omega) &\equiv \Bigg\{ \omega \in \Omega : 
		C_1 \int_0^{\tau \wedge \zeta} \Big( 
		\|v_n\|^2 +  \|v_n\|^{2r}+ \|v_n\|^6 + \|v_n\|^2|A_1^{1/2} \sigma_n|^4_{ L^2}+ \|\phi_n \|^2_{H^2}+ \|\phi_n \|^4_{H^2} \\
		&\quad+  \|\phi_n \|^6_{H^2}+ \|\phi_n \|^{10}_{H^2}+  \|\sigma_n \|^2_{H^1} \Big) ds  + C_2 \int_0^{\tau \wedge \zeta} \left( 1+
		\|z\|^2_{L^2_{\mathrm{div}}} +|u|^2_{ L^2} +|w|^2_{ L^2}
		\right) ds > \frac{(M - 1)^2}{2} \Bigg\},\\ 
        I_3(\Omega) &\equiv \Bigg\{ \omega \in \Omega : 
		2 \sup_{r \in [0, \tau \wedge \zeta]} \left| 
    	\int_0^r \sum_{k=1}^{\infty}  ( g_{1,k}(t,v_n), A_0 v_{n}) d\beta^1_k(t)\right|\\
		&\qquad+	2 \sup_{r \in [0, \tau \wedge \zeta]} \left| 
		\int_0^r \sum_{k=1}^{\infty}(g_{2,k}(t,\sigma_n), \sigma_{n}+ A_1\sigma_{n} ) d\beta^2_k(t)\right| > \frac{(M - 1)^2}{2} \Bigg\},
    	\end{align*}
then $\mathbb{P}(I_1(\Omega)) \leq \mathbb{P}(I_2(\Omega)) + \mathbb{P}(I_3(\Omega)).$ By Chebyshev's inequality and the fact \(\tau \in \mathcal{T}^{T,M}_n \), we get
	\begin{equation}\label{ex34}
		\begin{aligned}
			\mathbb{P}(I_2(\Omega)) 
			&\leq \frac{2C_1}{(M - 1)^2} 
			\mathbb{E} \int_0^{\tau \wedge \zeta} 
			\Big(\, \|v_n\|^2 +  \|v_n\|^{2r}+ \|v_n\|^6+ \|v_n\|^2|A_1^{1/2} \sigma_n|^4_{ L^2}+ \|\phi_n \|^2_{H^2}+ \|\phi_n \|^4_{H^2}+  \|\phi_n \|^6_{H^2}\\ &+ \|\phi_n \|^{10}_{H^2}+\|\sigma_n \|^2_{H^1} 
			\Big) ds  + \frac{2C_2}{(M - 1)^2} 
			\mathbb{E} \int_0^{\tau \wedge \zeta} 
			\left(1+\|z\|^2_{L^2_{\mathrm{div}}} +|u|^2_{ L^2} +|w|^2_{ L^2} \right) ds \\
			&\leq C_3 
			\mathbb{E} \int_0^{\zeta} 
			\left( 1 +\|z\|^2_{L^2_{\mathrm{div}}} +|u|^2_{ L^2} +|w|^2_{ L^2}  \right) ds,
		\end{aligned}
		\end{equation}
			where \(C_3 := C_3(\nu, \epsilon, \, C_{\psi},\eta, |\mathcal{O}|, \alpha, c, b, P, A, B_{V},B_{H^1}, M,\widetilde{K})\). Next, utilizing Doob’s maximal inequality, we arrive at
		\begin{equation}\label{ex35}
		\begin{aligned}
			\mathbb{P}(I_3(\Omega)) 
			&\leq \frac{16}{(M - 1)^4} \, 
			\mathbb{E} \int_0^{\tau \wedge \zeta}\Big( \,
			\|v_n\|^2 \,\| \mathcal{P}_1^nG_1(t,v_n)\|_{\mathcal{L}_2(U_1,V)}^2 + \|\sigma_n\|^2_{H^1}\, \|\mathcal{P}_2^n G_2(t,\sigma_n)\|_{\mathcal{L}_2(U_2,H^1)}^2 \, \Big)ds \\
			&\leq \frac{16C_{B_{V},B_{H^1}}}{(M - 1)^4} \, 
			\mathbb{E} \int_0^{\tau \wedge \zeta}\Big(\,
			\|v_n\|^2 ( 1+ \|v_n\|^2 
			 )+ \|\sigma_n\|^2_{H^1} (\, 1+ \|\sigma_n\|^2_{H^1} 
			 ) \Big)ds \leq C_4\zeta,
		\end{aligned}
	\end{equation}
			where \(C_4 := C_4(B_{V}, B_{H^1}, M,\widetilde{K})\). Finally using the bounds from \eqref{ex34} and \eqref{ex35}, we conclude that
	  \begin{equation}\label{ex36}
	   		\mathbb{P}(I_1(\Omega)) 
	   	\leq C_3 \, 
	   	\mathbb{E} \int_0^{\zeta} 
	   	\left( 1 + \|z\|^2_{L^2_{\mathrm{div}}} +|u|^2_{ L^2} +|w|^2_{ L^2}  \right) ds.
	   \end{equation}
	  	Given assumptions on $z, u, w$, and the \(\tau\)-independence of the right-hand side of \eqref{ex36}, the convergence of \eqref{secd} follows. This completes the proof.
	\end{proof}
	  \section{Local Existence and Uniqueness of Solutions}\label{Section 4}
	
			The lemma below provides estimates concerning the weak solutions $((v, \phi, \sigma); \tau)$. Besides being crucial in the derivation of higher-order estimates for $\phi$, the associated energy inequality will also be employed to obtain the unique global strong solution for the two-dimensional case.
			For any $(v, \phi, \sigma) \in \mathcal{H}$, define  $\mathcal{E}_{\mathrm{tot}}(v, \phi, \sigma) := \epsilon|\phi|_{L^2}^2 + 2\mathcal{E}(v, \phi, \nabla \phi, \sigma)$, where  $\mathcal{E}(v, \phi, \nabla \phi, \sigma)$ is defined in \eqref{e2}.
\subsection	{A Priory Estimate }
\begin{Lem}  (Energy inequality)\label{Ener_est}
	 	Let \(d=2,3\), \(G_1 \in \mathfrak{L}_{\mathrm{lin}}\big(L^2_{\mathrm{div}}, \mathcal{L}_2(U_1, L^2_{\mathrm{div}})\big) \),  \(G_2 \in \mathfrak{L}_{\mathrm{lin}}\big(L^2, \mathcal{L}_2(U_2, L^2)\big)\) and h is a non-negative function with \(|h| \leq 1\). Moreover, we assume that for all \(p \geq 2\),
	 \begin{equation}
	 	\begin{cases}
	 		\begin{aligned}
	 		& (z,u,w) \in L^p\left(\Omega;  L^2_{\text{loc}}\big([0,\infty);V' \times L^2 \times (H^{1})^{\prime}~\big)\right),  \text{with} \, u \in [0,1] \, \text{a.e. in }\, \Omega \times (0,T) \times \mathcal{O},\\
	 		& (v_0, \phi_0, \sigma_0) \in L^p\big(\Omega; \mathcal{H}\big)\quad \text{satisfies} \quad [\mathcal{E}_{\text{tot}}(v_0, \phi_0, \sigma_0)]^{p/2} < \infty. \label{inidata}
	 \end{aligned}
 \end{cases}
\end{equation}
     		 If \(\{(v, \phi, \sigma), \tau\}\) is a local weak solution in the sense of Definition \ref{def4.1}, then for any $ T>0$,  $ r \geq 1$ and $p \geq 2$, we have 
\begin{eqnarray}
		(i)&&\!\!\!\!\!\mathbb{E} \sup_{t \in [0, \tau \wedge T]} 	\left[\mathcal{E}_{\mathrm{tot}}(v, \phi, \sigma)(t)\right]^{\frac{p}{2}}
 		 + \mathbb{E} \int_0^{\tau \wedge T} \!\!\Bigg(\|v\|^{r+1}_{L^{r+1}_{\mathrm{div}}} +\|v\|^2 + \|\mu\|^2_{H^1} + \|\sigma\|^2_{H^1} \Bigg) 
 		\left[\mathcal{E}_{\mathrm{tot}}(v, \phi,\sigma)\right]^{\frac{p-2}{2}}	dt\,< \infty, \label{eqequality} \hspace{.32in} \\	 
        (ii) && \!\!\!\!\!\mathbb{E}\left(\int_0^{\tau \wedge T}  \Bigg(\|v\|^{r+1}_{L^{r+1}_{\mathrm{div}}}+\|v\|^2 + \|\mu\|^2_{H^1}+\|\sigma\|^2_{H^1} \Bigg)dt\right)^{\frac{p}{2}} 
		< \infty.
		\label{eqequality0} 	 
\end{eqnarray}
		\end{Lem}
	\noindent	{\bf Proof of (i)} Applying the infinite-dimensional It\^o formula (see \cite{Rozovsky2018}, \cite{DaPrato2014}) for \(\|v(\cdot)\|_{L^2_{\mathrm{div}}}^2\), \(|\sigma(\cdot)|_{L^2}^2\) together  with Remark \ref{pre}, $\eqref{b0_1}_1$ and $\eqref{b1_1}_1,$ and taking the  \(L^2\) inner product of \(\eqref{a2}\) with \( 2 \mu \) and \(2\epsilon\phi\), using Remark \ref{R0B1}, combining these give us:
	  \begin{equation}
	 	\begin{aligned}
	 			d\mathcal{E}_{\text{tot}}(v, \phi, \sigma)(t) &= 2\, \Bigg(-\eta\|v\|^{r+1}_{L^{r+1}_{\mathrm{div}}} -  \nu  \|v\|^2 - |A_1^{1/2}\mu|^2_{L^2} - \epsilon(\nabla \mu, \nabla \phi) -|A_1^{1/2}\sigma|^2_{L^2} + \langle  z, v\rangle -c (|\sigma|^2,h(\phi))\\ 
	 			& \qquad -b|\sigma|^2_{L^2} + b\langle w, \sigma\rangle+ ( (P\sigma - A- \alpha u)h(\phi), \mu) +\epsilon( (P\sigma - A- \alpha u)h(\phi), \phi)  \\
	 			&\qquad + \frac{1}{2}\|G_1(t, v)\|^2_{\mathcal{L}_2(U_1, L^2_{\mathrm{div}})} +\frac{1}{2}\|G_2(t, \sigma)\|^2_{\mathcal{L}_2(U_2, L^2)} \Bigg) \, dt+ \mathcal{S}(t),
	 	\end{aligned}
	 	\label{Etot_diff}
	 \end{equation}
     where  \( \mathcal{S}_1(t) := 2 \sum_{j=1}^{\infty} \Big( 		g_{1,j}(t, v),v\Big) \, d\beta^1_j(t) \), \( \mathcal{S}_2(t):= 2 \sum_{j=1}^{\infty} \Big(  g_{2,j}(t, \sigma),\sigma \Big ) \, d\beta^2_j(t), \)  \(\mathcal{S}(t) := \mathcal{S}_1 (t)+ \mathcal{S}_2(t) \).
				
Assuming  \( p \geq 2 \), we proceed by applying the It\^o formula once more to the real-valued quantity  \( \chi_p := [\mathcal{E}_{\text{tot}}(v, \phi, \sigma)(\cdot)]^{p/2} \) and utilizing the identity \eqref{Etot_diff} to get
 \begin{equation}
	 	\begin{aligned}
	 			d\chi_p(t) &= \frac{p}{2}d\mathcal{E}_{\text{tot}}(v, \phi, \sigma)(t) \chi_{p-2} 
	 			+ \frac{p(p-2)}{8} \chi_{p-4} \, 
	 			\langle\!\langle \mathcal{S} \rangle\!\rangle_t \,  
	 	\end{aligned}
	 	\label{chi_p_ito}
	 \end{equation}
	 		where \(\langle\!\langle \mathcal{S} \rangle\!\rangle_t\) represent the quadratic variation of \(\mathcal{S}\).
	 		Now, for $ k>0$ define a increasing sequence of stopping times
	 		\begin{equation}\label{sto2}
	 			\tau^k = \inf_{t \geq 0}\{~\|\psi(\phi(t))\|_{L^1} + |\sigma(t)|_{L^2}^2 > k~\} \wedge \tau.
	 		\end{equation}
	 		The next step is to show that, \(	\tau^k \to \tau \) a.s., as $k \to \infty$: suppose by contradiction that $\mathbb{P}\{\lim_{k \to \infty}\tau^k <\tau\}>0$. Then, by the definition of $\tau^k$, one has	therefore, for each $k$,
		\begin{equation*}
	 		\sup_{s\in[0,\tau]}(\|\psi(\phi(s))\|_{L^1} + |\sigma(s)|_{L^2}^2) \;\geq\; \|\psi(\phi(\tau^k ))\|_{L^1} + |\sigma(\tau^k )|_{L^2}^2\;>\; k,
	 	\end{equation*}
	 		$\text{and hence,} \;\mathbb{P}\{\sup_{t \in [0, \tau]}\|\psi(\phi(t)\|_{L^1}+\sup_{t \in [0, \tau]}|\sigma(t)|_{L^2}^2=+\infty\}>0.$	This is impossible because \(\{(v, \phi, \sigma), \tau\}\) is a local weak solution, which ensures finiteness a.s.. Hence necessarily $\lim_{k \to \infty}\tau^k =\tau$ a.s..
	 	Next, for a fixed $k >0$, let  \( \tau_a \) and \( \tau_b \) be stopping times such that \( 0 \leq \tau_a \leq \tau_b \leq \tau^k \wedge T \). In view of \eqref{chi_p_ito}, we can write
 \begin{equation}
	 	\begin{aligned}
	 		\mathcal I:=&\mathbb{E} \sup_{t \in [\tau_a, \tau_b]} \chi_p(t) 
	 		+ p \mathbb{E} \int_{\tau_a}^{\tau_b} 
	 		\left(\,\eta\|v\|^{r+1}_{L^{r+1}_{\mathrm{div}}} +\nu  \|v\|^2 + |A_1^{1/2}\mu|^2_{L^2} +|A_1^{1/2}\sigma|^2_{L^2} \right) \chi_{p-2}(t) \, dt \\
	 		&\leq \mathbb{E} \left[\mathcal{E}_{\text{tot}}(v, \phi, \sigma)(\tau_a)\right]^{p/2}
	 		+ C_{p} \left[\mathbb{E} \int_{\tau_a}^{\tau_b} \epsilon\chi_{p-2}
	 		|(\nabla \mu, \nabla \phi)| dt+  \mathbb{E} \int_{\tau_a}^{\tau_b} \chi_{p-2} |\langle z,v\rangle| dt\right. \\
            & \left.\quad+ \mathbb{E} \int_{\tau_a}^{\tau_b} \chi_{p-2} \Bigg(c |(|\sigma|^2,h(\phi))|+b|\sigma|^2_{L^2} + b|\langle w, \sigma\rangle|\Bigg)dt + \mathbb{E} \int_{\tau_a}^{\tau_b} \chi_{p-2}|((P\sigma - A- \alpha u) h(\phi), \mu)| dt \right.\\ 
	 		&\left. \quad +\mathbb{E} \int_{\tau_a}^{\tau_b} \chi_{p-2}\epsilon|\left((P\sigma - A- \alpha u)h(\phi),  \phi\right)|dt + \frac{1}{2}\mathbb{E} \int_{\tau_a}^{\tau_b} \chi_{p-2}\Bigg(\|G_1(t, v)\|^2_{\mathcal{L}_2(U_1, L^2_{\mathrm{div}})}\right.\\ &\left.   \quad+\|G_2(t, \sigma)\|^2_{\mathcal{L}_2(U_2, L^2)} \Bigg) dt \right] + C_{p} \mathbb{E} \int_{\tau_a}^{\tau_b} \chi_{p-4}\,\langle\!\langle  \mathcal {S} \rangle\!\rangle_t  + C_{p} \mathbb{E} \sup_{t \in [\tau_a, \tau_b]} 
	 		\left|\int_{\tau_a}^t \chi_{p-2}(s) \,\mathcal {S} \right| =  \sum_{i=1}^9 I_i.
	 	\end{aligned}
	 	\label{energy_est}
 \end{equation}
			 Next, we estimate each term on the right-hand side of \eqref{energy_est}. Since $\epsilon|A_1^{1/2}\phi|^2_{L^2} \leq \mathcal{E}_{\mathrm{tot}}(v, \phi, \sigma),$ we have
 \begin{equation}	\label{1_est}
	 	\begin{aligned}
	 		I_2 
	 		 &\leq  \frac{p}{4}\, \mathbb{E} \int_{\tau_a}^{\tau_b} \chi_{p-2} |A_1^{1/2}\mu|^2_{L^2}\,dt+ C_{p,\epsilon} \,\mathbb{E} \int_{\tau_a}^{\tau_b} \chi_{p}(t)\,dt.
	 	     \end{aligned}
		\end{equation}
 				Using H\"older’s and Young’s inequalities, we derive		
 	\begin{equation}
 		\begin{aligned}
 			I_3 
            &\leq C_{p, \nu}\, \mathbb{E} \int_{\tau_a}^{\tau_b} \chi_{p-2} \|z\|^2_{V^'}\,dt+ \frac{\nu p}{2} \,\mathbb{E} \int_{\tau_a}^{\tau_b} \chi_{p-2}\|v\|^2\,dt\\
 			&\leq\frac{1}{8} \mathbb{E}  \sup_{t \in [\tau_a, \tau_b]} \chi_p(t)+ \frac{\nu p}{2} \,\mathbb{E} \int_{\tau_a}^{\tau_b}\|v\|^2 \chi_{p-2}\,dt + C_{p,\nu} \mathbb{E} \left( \int_{\tau_a}^{\tau_b} \|z\|_{V^{\prime}}^{2}  \, dt \right)^{p/2}.
 	\end{aligned}
 	\label{2_est}
 \end{equation}
			By invoking $|\sigma|^2_{L^2} \leq \mathcal{E}_{\mathrm{tot}}(v, \phi, \sigma)$ and boundedness of \(h\), one can get
 	\begin{align}	\label{3_est}
 			I_4&\leq\frac{1}{8} \mathbb{E}  \sup_{t \in [\tau_a, \tau_b]} \chi_p(t)+ \frac{\gamma p}{2} \mathbb{E} \int_{\tau_a}^{\tau_b} \chi_{p-2}\|\sigma\|^2_{H^1}dt + C_{p, b,\gamma} \mathbb{E} \left( \int_{\tau_a}^{\tau_b} \|w\|_{(H^{1})^{\prime}}^{2} dt \right)^{p/2}\!\!\!+ C_{p,b} \mathbb{E} \int_{\tau_a}^{\tau_b} \chi_{p}dt.
 	\end{align}
 			Since $\int_{\mathcal{O}} \Delta \phi dx=\int_{\partial \mathcal{O}}\frac{\partial\phi}{\partial{\bf n}}dS=0,$ the condition $\eqref{eq:ass1}_3$ gives
\begin{equation*}
		\left| \int_{\mathcal{O}} \mu  \, dx \right|
		= \left| \int_{\mathcal{O}} \epsilon^{-1} \psi'(\phi) \, dx \right| \leq \epsilon^{-1}C_{\psi} \int_{\mathcal{O}} \left(1 + |\phi|^{\rho - 1}\right) dx = C_{c_{\psi},\epsilon} \left(1 + \| \phi \|_{L^{\rho - 1}}^{\rho - 1} \right). 
\end{equation*}
	After utilizing the embedding $ L^{\rho}\hookrightarrow L^{\rho-1}$ for any $\rho \geq 2$ and $\eqref{eq:ass1}_1$, one obtains 
\begin{equation}\label{vphi41}
		|\bar{\mu}| 
        \leq C_{c_{\psi},\epsilon, |\mathcal{O}|} \left(1 + \| \phi \|_{L^{\rho}}^{\rho-1} \right)\leq C_{c_{\psi}, \epsilon, |\mathcal{O}|} \left(1 + \| \phi \|_{L^{\rho}}^{\rho} \right)\leq C_{c_{\psi}, \epsilon, |\mathcal{O}|} \left(1 + \|\psi(\phi)\|_{L^1} \right).
\end{equation}
		Using the boundedness of $h$, \eqref{vphi41},  and Poincar\'e–Wirtinger's  inequality, we deduce
\begin{equation*}
	\begin{aligned}
			I_5 &\leq C_{p,C_{\mathcal{O}}}\, \mathbb{E} \int_{\tau_a}^{\tau_b} \chi_{p-2}\, |P\sigma - A- \alpha u|_{L^2}\left(|\mu-\bar \mu|_{L^2} + |\bar{\mu}|_{L^2} \right)\,dt\\
			&\leq C_{p,C_{\mathcal{O}},c_{\psi},\epsilon}\,	\mathbb{E} \int_{\tau_a}^{\tau_b}\chi_{p-2}\left(1+|P\sigma - A- \alpha u|_{L^2}^2\right)\left(1+\|\psi(\phi)\|_{L^1}\right)\,dt \\
            &\quad+\frac{p}{4}	\mathbb{E} \int_{\tau_a}^{\tau_b}\chi_{p-2} |A_1^{1/2}\mu|_{L^2}^2\,dt :=I_{51}+I_{52}.
\end{aligned}
\end{equation*}
		Further, observing  $|\sigma|^2_{L^2} \leq \mathcal{E}_{\mathrm{tot}}(v, \phi, \sigma),$ and $ \|\psi(\phi)\|_{L^1}\leq \epsilon\,\mathcal{E}_{\mathrm{tot}}(v, \phi, \sigma),$ we get
\begin{align}
		I_{51} 	&\leq C_1\,\mathbb{E} \int_{\tau_a}^{\tau_b}(1+ |P\sigma - A- \alpha u|_{L^2}^2)\chi_{p}\,dt + 	C_2\,\mathbb{E}\int_{\tau_a}^{\tau_b}(1+ |u|_{L^2}^2)\chi_{p-2}\,dt \nonumber\\
			& \leq \frac{1}{8} \,\mathbb{E}  \sup_{t \in [\tau_a, \tau_b]} \chi_p(t)  + C_1\,\mathbb{E}\int_{\tau_a}^{\tau_b}(1+ |P\sigma - A- \alpha u|_{L^2}^2 )\chi_{p}\,dt  + C_2\,\mathbb{E} \left( \int_{\tau_a}^{\tau_b} |u|_{L^2}^2 \, dt \right)^{p/2}+ C_3, \label{5_est}  
	\end{align}
	       where \( C_1 := C_1(p,\epsilon, 
		C_{\mathcal{O}},C_{\psi}, P),\, C_2 := C_2(p,\epsilon, \alpha,|\mathcal{O}|,C_{\psi}, A)\), and \(C_3 := C_2 \cup \{T\}.\)
    The estimate similar to $I_5$ works for $I_6$ as well. 
    Next, owing to the Assumption [\ref{[A2]}] on \(G_1\) and \(G_2\), one can infer 
\begin{equation}
	\begin{aligned}
		I_7&\leq C_{p,B_{L^2_{\mathrm{div}}}} \mathbb{E} \int_{\tau_a}^{\tau_b}(1+   \|v\|^2_{L^2_{\mathrm{div}}})\chi_{p-2} \,dt + C_{p, B_{L^2}} \mathbb{E} \int_{\tau_a}^{\tau_b} (1+ |\sigma|^2_{L^2})\chi_{p-2} \,dt\\
		&\leq \frac{1}{8} \mathbb{E} \sup_{t \in [\tau_a, \tau_b]} \chi_p(t) +
		C_{p,B_{L^2_{\mathrm{div}}},B_{L^2},T} \,\mathbb{E}\left( 1 +  \int_{\tau_a}^{\tau_b} \chi_p \, dt \right),
	\end{aligned}
	\label{7est}
\end{equation}
where we also used $\|v\|^2_{L^2_{\mathrm{div}}} \leq \mathcal{E}_{\mathrm{tot}}(v, \phi, \sigma).$		The definition of quadratic variation of \(\mathcal{S}\) (\cite{Rozovsky2018}) leads to
	\begin{equation*}
	\begin{aligned}
       \langle\!\langle \mathcal{S}\rangle\! \rangle_t &= \langle\!\langle \mathcal{S}_1\rangle\! \rangle_t + \langle\!\langle \mathcal{S}_2 \rangle\!\rangle_t + 2\langle\!\langle \mathcal{S}_1, \mathcal{S}_2 \rangle\! \rangle_t\\
		& =   4\sum_{j=1}^{\infty} \Big(g_{1,j}(t,v), v\Big)^2 \, dt +  4\sum_{j=1}^{\infty} \Big( g_{2,j}(t,\sigma), \sigma\Big)^2 \, dt +8 \sum_{i,j=1}^{\infty} \left( g_{1,i}(t,v),v \right)\,\Big(  g_{2,j}(t,\sigma), \sigma\Big) \langle\!\langle d\beta^1_i , d\beta^2_j  \rangle\!\rangle_t.
	\end{aligned}
\end{equation*}
 			Since $\beta^1_i$ and $\beta^2_j$ are independent Brownian motions for all $i,j$, it follows from the vanishing quadratic covariation of independent Brownian motions that:
\begin{equation*}
            I_8 =  C_{p} \mathbb{E} \int_{\tau_a}^{\tau_b} \chi_{p-4} \sum_{j=1}^{\infty} \Big(  g_{1,j}(t,v),v \Big)^2 \, dt+ C_p \mathbb{E} \int_{\tau_a}^{\tau_b} \chi_{p-4} \sum_{j=1}^{\infty} \Big(   g_{2,j}(t,\sigma), \sigma \Big) ^2 \, dt. 
\end{equation*}	 
			 Again by [\ref{[A2]}] and \eqref{7est}, we have
 \begin{equation}
	 	\begin{aligned}
	 		I_8 &\leq C_p \mathbb{E} \int_{\tau_a}^{\tau_b} \chi_{p-4} \|v\|^2_{L^2_{\mathrm{div}}} \|G_1(t, v)\|^2_{\mathcal{L}_2(U_1, L^2_{\mathrm{div}})} \, dt + C_p \mathbb{E} \int_{\tau_a}^{\tau_b} \chi_{p-4} |\sigma|^2_{L^2} \|G_2(t, \sigma)\|^2_{\mathcal{L}_2(U_2, L^2)} \, dt \\
	 	    &\leq \frac{1}{8} \mathbb{E} \sup_{t \in [\tau_a, \tau_b]} \chi_p(t) +
	 	    C_{p,B_{L^2_{\mathrm{div}}},B_{L^2},T}\, \mathbb{E}\, \Big( 1 +  \int_{\tau_a}^{\tau_b} \chi_p \, dt \Big).
	 	\end{aligned}
	 	\label{8est}
 \end{equation}
			 Finally, for the noise-driven terms, applying the Burkholder-Davis-Gundy inequality and the linear growth conditions on \(G_1, G_2\), followed by Young's inequality we find:
	 \begin{equation}
	 	\begin{aligned}
	 		I_9&\leq C_{p} \mathbb{E} \left[ \left( \int_{\tau_a}^{\tau_b} \sum_{j=1}^{\infty} \chi_{2(p-2)} ( g_{1,j}(t,v), v  )^2 \, dt \right)^{1/2} \right] +  C_{p} \mathbb{E} \left[ \left( \int_{\tau_a}^{\tau_b} \sum_{j=1}^{\infty} \chi_{2(p-2)} ( g_{2,j}(t,\sigma),\sigma )^2 \, dt \right)^{1/2} \right] \\
	   	   &\leq C_{p} \mathbb{E} \left[ \left( \int_{\tau_a}^{\tau_b} \chi_{2(p-1)} \|G_1(t, v)\|^2_{\mathcal{L}_2{(U_1, L^2_{\mathrm{div}})}} \, dt \right)^{1/2} \right] +C_{p} \mathbb{E} \left[ \left( \int_{\tau_a}^{\tau_b} \chi_{2(p-1)} \|G_2(t, \sigma)\|^2_{\mathcal{L}_2(U_2, L^2)} \, dt \right)^{1/2} \right] \\
	   	    &\leq C_{p,B_{L^2_{\mathrm{div}}}} \mathbb{E} \left[ \left( \int_{\tau_a}^{\tau_b} \chi_{2(p-1)} ( 1 + \|v\|^2_{L^2_{\mathrm{div}}}) \, dt \right)^{1/2} \right] +C_{p B_{L^2}} \mathbb{E} \left[ \left( \int_{\tau_a}^{\tau_b} \chi_{2(p-1)} (1+ |\sigma|^2_{L^2}) \, dt \right)^{1/2} \right] \\
	 		&\leq \frac{1}{8} \mathbb{E} \sup_{t \in [\tau_a, \tau_b]} \chi_p(t) + C_{p,B_{L^2_{\mathrm{div}}}, B_{L^2},T} \mathbb{E} \left(1+ \int_{\tau_a}^{\tau_b} \chi_p \, dt\right).
	 	\end{aligned}
	 	\label{9est}
	 \end{equation}
			Inserting the estimates \eqref{1_est}-\eqref{9est} into \eqref{energy_est} and rearranging all terms, we derive the inequality
	 \begin{equation}
	 	\begin{aligned}
	 		\mathcal I 
	 		&\leq \bar{C}_1 \mathbb{E} \left[ 1+\left( \mathcal{E}_{\text{tot}}(v, \phi, \sigma)(\tau_a) \right)^{p/2} \right] + \bar{C}_3 \mathbb{E} \int_{\tau_a}^{\tau_b}\big( 1+ |P\sigma - A- \alpha u|_{L^2}^2\big) \chi_p \, dt \\
            &\quad+ \bar{C}_2 \mathbb{E}~ \left[\left( \int_{\tau_a}^{\tau_b} \|z\|_{V^{\prime}}^{2}\, dt \right)^{p/2} + \left( \int_{\tau_a}^{\tau_b} |u|_{L^2}^2 \, dt \right)^{p/2}+ \left( \int_{\tau_a}^{\tau_b} \|w\|_{(H^{1})^{\prime}}^{2} \, dt \right)^{p/2} \right],
	 	\end{aligned}
	 	\label{10est}
	 \end{equation}
			 where $ \bar{C}_1 := \bar{C}_1(p,A,T), \bar{C}_2 := \bar{C}_2(p, 	\nu,\epsilon, 
			 \mathcal{O},C_{\psi}, P, c, b,A, \alpha), \bar{C}_3 := \bar{C}_3 (p, \nu,c_{\epsilon},\mathcal{O},C_{\psi}, P, b, c, A, \alpha,$ $ B_{H}, B_{L^2},T). $ 
			Next, applying the Poinca\'re–Wirtinger inequality and \eqref{vphi41}, we infer 
	 \begin{equation}\label{11est}
	 	\begin{aligned}
	 		&\mathbb{E} \int_{\tau_a}^{\tau_b} |\mu|^2_{L^2}\chi_{p-2} \, dt 
            \leq \bar{C}_4\, \mathbb{E}\int_{\tau_a}^{\tau_b} \left[ |A_1^{1/2}\mu|^2_{L^2} + \big(1+\|\psi(\phi)\|_{L^1}\big)^2\right]\chi_{p-2} \, dt\\
	 		&\qquad\leq \frac{1}{2}\mathbb{E} \sup_{t \in [\tau_a, \tau_b]} \chi_p(t)+ \bar{C}_4\, \mathbb{E}\int_{\tau_a}^{\tau_b}|A_1^{1/2}\mu|^2_{L^2}\chi_{p-2} \, dt+ \bar{C}_4~\mathbb{E}\int_{\tau_a}^{\tau_b}\|\psi(\phi)\|_{L^1}\chi_{p} \, dt + \bar{C}_5,
	 	\end{aligned}
	  \end{equation}
	  where \(\bar{C}_4 := \bar{C}_4(\mathcal{O},c_{\psi}, \epsilon)\) and \(\bar{C}_5 := \bar{C}_5(\mathcal{O},c_{\psi}, \epsilon,p,T)\).
	 		     From \eqref{sto2} and the boundedness of $u$ by 1 (see [\ref{[A4]}]) on $[0,\tau^k \wedge T]$, we infer that $\|\psi(\phi)\|_{L^1} +|P\sigma - A- \alpha u|_{L^2}^2  \leq C_k $, where $C_k$ depends on $  P, A, |\mathcal{O}|,\alpha, $ and k. Finally, combining \eqref{10est} and \eqref{11est} by the stochastic Gronwall Lemma  \ref{l3}, the monotone convergence theorem, and  $\tau^k \to \tau$  a.s. as $ k \to \infty$, we conclude the 
	  		proof of  \eqref{eqequality}. \\
  		
            \noindent{\bf Proof of (ii)}
           For fixed $k>0$ and $p \in [2, \infty)$, take the same sequence of stopping times $\tau^k$ defined in \eqref{sto2}. Integrating \eqref{Etot_diff} over $[0, \tau^k \wedge T]$, raising both sides to the $\frac{p}{2}$-power, taking the supremum over $s \in [0, \tau^k \wedge T]$, and then taking expectation, we obtain
    \begin{eqnarray}
				\lefteqn{\mathcal J(0,\tau^k \wedge T):=\mathbb{E} \sup_{t \in [0, \tau^k \wedge T]}\mathcal{E}_{\text{tot}}^{p/2}(v, \phi, \sigma)(s)+ \mathbb{E}\left(\int_0^{\tau^k \wedge T}  \Bigg(\|v\|^{r+1}_{L^{r+1}}+\|v\|^2 + |A_1^{1/2}\mu|^2_{L^2}+\|\sigma\|^2_{H^1} \Bigg)dt\right)^{\frac{p}{2}} }\nonumber \\& \leq \tilde{C}\left[ \mathbb{E}\,\mathcal{E}_{\text{tot}}^{p/2}(v, \phi, \sigma)(0)+ \,\mathbb{E}\,\Bigg|\int_{0}^{\tau^k \wedge T} \epsilon(\nabla \mu, \nabla \phi) +\langle z, v\rangle +c (|\sigma|^2,h(\phi))  +(1+b)|\sigma|^2_{L^2}\right. \nonumber \\
                &\ \ \ \ + b\langle w, \sigma\rangle+ ((P\sigma - A- \alpha u)h(\phi), \mu)+\epsilon( (P\sigma - A- \alpha u)h(\phi), \phi)\,dt\Bigg|^{p/2}\label{Etot_diff1}\\
				&\left. \ \  \ \quad +\mathbb{E} \,\Bigg|\int_{0}^{\tau^k \wedge T}\left(\,\|G_1(t, v)\|^2_{\mathcal{L}_2(U_1, L^2_{\mathrm{div}})}+\|G_2(t, \sigma)\|^2_{\mathcal{L}_2(U_2, L^2)}\right)\,dt \Bigg|^{p/2}+ \mathbb{E} \,\sup_{s \in [0, \tau^k \wedge T]}\left| \int_{0}^{s}\mathcal {S}(t)\,\right|^{p/2} \right] \nonumber,
	\end{eqnarray}
	    where $\tilde{C}$ depends on $p, \nu$ and $\eta$. 
        We estimate the stochastic integral $\mathcal K :=\mathbb{E} \sup_{s \in [0, \tau^k \wedge T]}\left| \int_{0}^{s}\mathcal {S}(t)\right|^{p/2}$ and the other terms on the right-hand side can be estimated in analogy to the proof of (i). 
		Applying the Burkholder-Davis-Gundy and Young inequalities, we obtain
\begin{align}
        \mathcal K &  \leq C\,\mathbb{E} \Bigg[\,\Bigg(\int_{0}^{\tau^k \wedge T}\|v\|^2_{L^2_{\mathrm{div}}}\|G_1(t, v)\|^2_{\mathcal{L}_2(U_1, L^2_{\mathrm{div}})} \,dt\Bigg)^{p/4}+ \Bigg(\int_{0}^{\tau^k \wedge T}|\sigma|^2_{L^2}\|G_2(t, \sigma)\|^2_{\mathcal{L}_2(U_2, L^2)}\,dt\Bigg)^{p/4}  \Bigg]\nonumber\\
        & \leq C\,\mathbb{E}\,\Bigg[\Bigg(\sup_{s \in [0, \tau^k \wedge T]}\|v(s)\|^2_{L^2_{\mathrm{div}}}\int_{0}^{\tau^k \wedge T} B_{L^2_{\mathrm{div}}}(1+   \|v\|^2_{L^2_{\mathrm{div}}})\,dt\Bigg)^{p/4} \nonumber \\
        &\qquad+ \Bigg(\sup_{s \in [0, \tau^k \wedge T]}|\sigma(s)|^2_{L^2}\int_{0}^{\tau^k \wedge T} B_{L^2} (1+ |\sigma|^2_{L^2})\,dt\Bigg)^{p/4}  \Bigg] \label{Etot_diff5}\\
        & \leq \frac{1}{2}\,\mathbb{E}\sup_{s \in [0, \tau^k \wedge T]}\mathcal{E}_{\text{tot}}^{p/2}(v, \phi, \sigma)(s)+ C_{p,B_{L^2}B_{L^2_{\mathrm{div}}}} \mathbb{E}\int_{0}^{\tau^k\wedge T}\Bigg(1+ \mathcal{E}_{\mathrm{tot}}^{p/2}(v, \phi, \sigma)\Bigg)\;dt+ C_{B_{L^2},B_{L^2_{\mathrm{div}}}, T}\nonumber.
\end{align}
         Thus, making use of \eqref{Etot_diff5}, we obtain the following from \eqref{Etot_diff1}:             
	\begin{align}
                \mathcal J(0,\tau^k \wedge T) &\leq \tilde{C}\left[ \mathbb{E}\,\mathcal{E}_{\text{tot}}^{p/2}(v, \phi, \sigma)(0) + \tilde{C}_1+	 \tilde{C}_2 \mathbb{E}\int_{0}^{\tau^k\wedge T} \mathcal{E}_{\mathrm{tot}}^{p/2}(v, \phi, \sigma)\;dt  \right. \nonumber\\
                &\left. \quad +\mathbb{E} \left(\int_{0}^{ T}\|z\|_{V^{\prime}}^{2} \,ds\right)^{p/2} + \tilde{C}_3\mathbb{E} \left(\int_{0}^{ T} \|w\|_{(H^{1})^{\prime}}^{2}\,ds\right)^{p/2} +\tilde{C}_4\mathbb{E} \left(\int_{0}^{ T}|u|^2_{L^2}\,ds\right)^{p/2}  \right], \label{Etot_diff6}
	\end{align}
                 where  $\tilde{C}_1 := \tilde{C}_1(p,B_{L^2}B_{L^2_{\mathrm{div}}},A, |\mathcal{O}|,T), \tilde{C}_2 := \tilde{C}_2(p,b,c,B_{L^2}B_{L^2_{\mathrm{div}}},P,c_{\psi},c_{\mathcal{O}}, \epsilon,|\mathcal{O}|,\alpha,k),\tilde{C}_3 := \tilde{C}_3(p,b)$, and $\tilde{C}_4 := \tilde{C}_4(p,\epsilon,\alpha)$.
		           Hence, from the (deterministic) Gronwall lemma, assumptions on, $ z, u,w$ and the fact that $\tau^k \to \tau$ as $k \to \infty$, we obtain $\mathcal J(0,\tau \wedge T) <\infty.$
	
                    Applying the Poincar\'e–Wirtinger inequality together with estimate \eqref{vphi41} and $\mathcal J(0,\tau \wedge T) <\infty,$ one can obtain
	                $\mathbb{E}\, \left(\int_{0}^{\tau \wedge T}|\mu|^2_{L^2}\,ds\right)^{p/2}< \infty,$ which completes the proof.
\hfill{$\Box$}
	\subsection	{Regularity Estimates for \(\phi\).}
		
			Assume that the energy inequality \eqref{eqequality} (for $p = 4$) holds.
	       Taking the inner product of \eqref{a4} with \(A_1\phi\), we derive: $\epsilon |A_1 \phi|^2_{ L^2} +  \epsilon^{-1} (\psi''(\phi), |\nabla \phi|^2) 
	  		= (\nabla\mu, \nabla \phi).$
	       Using assumptions on \( \psi_1 \), \( \psi_2 \) (see \eqref{psi-decomp}-\eqref{psi2-bounds} in [\ref{[A1]}]), we get
	  		$\epsilon |A_1 \phi|^2_{ L^2}\leq \epsilon^{-1} R_3 |A_1 ^{1/2} \phi|^2_{ L^2} + |A_1 ^{1/2}\mu |_{ L^2} |A_1 ^{1/2}\phi|_{ L^2}.$
            It leads to the following estimate:
	  \begin{align*}\label{hr4}
	  		\epsilon^2~ \mathbb{E}\int_0^{\tau\wedge T} |A_1 \phi|^4_{ L^2} \, dt &\leq ~\mathbb{E} \int_0^{\tau\wedge T} \left( C(\epsilon, R_3)\,|A_1 ^{1/2} \phi|^4_{ L^2} + |A_1 ^{1/2} \phi|^2_{ L^2} |A_1 ^{1/2}\mu|^2_{ L^2} \right) dt \\
	  		&\leq \,  C(\epsilon, R_3)\,\mathbb{E} \,\sup_{ t \in [0,\tau\wedge T]} |A_1 ^{1/2}\phi(t)|^4_{ L^2} +  \mathbb{E} \,\int_0^{\tau\wedge T}|A_1 ^{1/2}\phi|^2_{ L^2}\,|A_1 ^{1/2}\mu |^2_{ L^2} \,dt .
	  \end{align*}
	 		The energy inequality \eqref{eqequality} for the case $p = 4$ gives
\begin{equation}\label{hr1}
		\mathbb{E}\int_{0}^{\tau\wedge T}\|\phi\|_{H^2}^4\,dt  \, < \infty.
	\end{equation}
            Further, assume that the energy inequality \eqref{eqequality0} (for p = 2,8,12) holds.
            Again, the inner product of \eqref{a4}  with $\Delta^2 \phi$,  integration over $[0,  \tau\wedge T]$, and then expectation leads to the following: 
  \begin{equation}\label{phiDelta2}
		\epsilon \mathbb{E} \int_0^{\tau\wedge T}|A_1 ^{3/2}\phi|^2_{ L^2} = -\mathbb{E} \int_0^{\tau\wedge T} \int_{\mathcal{O}} \nabla\mu  \cdot \nabla \Delta \phi \, dx dt 
		+\mathbb{E} \int_0^{\tau\wedge T} \int_{\mathcal{O}} \epsilon^{-1} \psi''(\phi) \nabla \phi \cdot \nabla \Delta \phi \, dx dt:=I_1+I_2. 
	\end{equation}
	      Note that
		$I_1 \leq C_{\epsilon} \mathbb{E} \int_0^{\tau\wedge T}|A_1 ^{1/2}\mu|^2_{ L^2} \,dt + \frac{\epsilon}{8} \mathbb{E}\int_0^{\tau\wedge T}|A_1 ^{3/2}\phi|^2_{ L^2}\,dt, $
and assumption \eqref{eq:ass1}-\eqref{eq:ass2} and Agmon's inequality yield
	\begin{eqnarray}\label{phiDeltaEst2}
			I_2&\leq& C \mathbb{E}\int_0^{\tau\wedge T} (1 + \|\phi\|^4_{L^\infty}) \|A_1 ^{1/2} \phi|_{ L^2} |A_1 ^{3/2}\phi|_{ L^2} \, dt 
			\leq C\mathbb{E} \int_0^{\tau\wedge T} \left(1 + \|\phi\|^{2}_{H^2}\|\phi\|^{2}_{H^1} \right) |A_1 ^{1/2}\phi|_{ L^2}|A_1 ^{3/2} \phi|_{ L^2} \, dt \nonumber\\
			&\leq&\!\! C_{\epsilon} \mathbb{E}\sup_{ t \in [0,{\tau\wedge T}]}|A_1 ^{1/2}\phi(t)|^2_{ L^2}  + C_{\epsilon} \mathbb{E}\sup_{ t \in [0,{\tau\wedge T}]}\|\phi(t)\|^{12}_{H^1}+ \mathbb{E}\left(\int_0^{\tau\wedge T} \!\!\!\|\phi\|^{4}_{H^2}  \,dt\right)^2 + \frac{\epsilon}{8}\mathbb{E} \int_0^{\tau\wedge T} |A_1 ^{3/2} \phi|^2_{ L^2}dt. 		
\end{eqnarray}
	         The bound of $|A_1 \phi|_{ L^2}$ given above leads to the estimate 
\begin{equation*}
		C_{\epsilon} \mathbb{E}\left( \int_0^{\tau\wedge T}|A_1 \phi|^4_{ L^2}\,dt\right)^2\leq C_{\epsilon, R_3} \mathbb{E}\sup_{ t \in [0,{\tau\wedge T}]}|A_1 ^{1/2}\phi(t)|^8_{ L^2}+ \mathbb{E}\left(\int_0^{\tau\wedge T}|A_1 ^{1/2}\mu |^2_{L^2}dt\right)^4,
\end{equation*}
	       and hence the bounds \eqref{eqequality}-\eqref{eqequality0} show that
		 $\mathbb{E}\left(\int_0^{\tau\wedge T} \|\phi\|^{4}_{H^2}  \,dt\right)^2< \infty.$
            Further, collecting the estimates \eqref{phiDelta2}–\eqref{phiDeltaEst2}, invoking the preceding bound, and  \eqref{eqequality},  we also obtain that
$		\mathbb{E} \int_0^{\tau\wedge T} \| \phi\|^2_{H^3} \,dt< \infty. $
\subsection {Uniqueness of Solutions}
                In this section, we first demonstrate that the strong solution is unique in the larger class of weak solutions. We then establish uniqueness within the class of weak solutions itself.
\begin{Pro}(Weak-strong uniqueness)\label{weak-strong}  
		        Let \( \tau_1, \tau_2 > 0 \) be stopping times. Set $\tau := \tau_1 \wedge \tau_2$. Suppose that the assumptions of Lemma \ref{Ener_est} are satisfied and \(h\) satisfies Assumption [\ref{[A3]}]. We assume that \( \{(v_1, \phi_1, \sigma_1), \tau_1\} \) and \( \{(v_2, \phi_2,\sigma_2), \tau_2\} \) are, respectively, local strong and weak solutions to the stochastic CH-CBF reaction diffusion system for \( d = 2, 3 \) corresponding to the sources \((z_1, u_1, w_1)\) and  \((z_2, u_2, w_2)\) and initial data \( \Big(v_{1}, \phi_{1},\sigma_{1}\Big)(0)\) and \( (v_{2}, \phi_{2},\sigma_{2})(0) \) with $\bar{\Omega} = \{\big(v_{1}, \phi_{1},\sigma_{1} \big)(0) = \big(v_{2}, \phi_{2},\sigma_{2} \big)(0)\} \subseteq \Omega$, then the pathwise uniqueness (see Definition \ref{pathuq}) holds.   
	 \end{Pro}  
	  \begin{proof}
	   		Let \(\mu := \mu_1 - \mu_2, z:= z_1-z_2, u:= u_1-u_2 \)  and \(w:= w_1-w_2\).  Consider the difference \( ( v, \phi, \sigma) := (v_1, \phi_1, \sigma_1) - (v_2, \phi_2, \sigma_2) \). 
	  		 Then, \( ( v, \phi,\sigma, \mu) \) satisfies the system:
	 \begin{subequations}	\label{uqe1}
	   \begin{align}
	   	&d v + \nu A_0  v \, dt + \eta\left[\mathcal{A}_r(v_1) - \mathcal{A}_r(v_2)\right] dt + \left[ B_0(v_2,  v) + B_0(v, v_1) \right] dt \nonumber\\
	   	&\qquad= \  \left[ R_0(\mu_2, \phi) + R_0(\mu, \phi_1) \right] dt + z dt  + \sum_{k=1}^\infty \left[ g_{1,k}(t, v_1) - g_{1,k}(t, v_2) \right] d\beta^1_k(t), \label{uqe11}\\
	   	&\frac{d \phi}{dt} + A_1 \mu+ B_1( v_2, \phi) + B_1(v,  \phi_1)  = (P\sigma_1-A- \alpha u_1)h(\phi_1)- (P\sigma_2-A- \alpha u_2)h(\phi_2), \label{uqe12}\\
	   	&d \sigma +  A_1  \sigma \, dt + \left[ B_1(v_2,  \sigma) + B_1(v, \sigma_1) \right] dt +c\left[\sigma_1h(\phi_1)-\sigma_2h(\phi_2)\right] \,dt \nonumber\\
	   	&\qquad= b(w-\sigma) \,dt + \sum_{k=1}^\infty \left[ g_{2,k}(t, \sigma_1) - g_{2,k}(t, \sigma_2) \right] d\beta^2_k(t),\label{uqe13} \\
	   	& \mu =  \epsilon A_1  \phi + \epsilon^{-1} \left[ \psi'(\phi_1) - \psi'(\phi_2) \right] \label{uqe14}.
	  \end{align}
	  \end{subequations}
	  \textbf{Step 1.}
		 	 Applying the It\^o formula to $\|v(\cdot)\|_{L^2_{\mathrm{div}}}^2$, and using $\eqref{b0_1}_1$, we proceed as follows. For any stopping time $\tilde{\tau}\leq \tau$, we integrate the resulting equality over the interval $[0,\tilde{\tau}]$, 
		  	multiply by $\mathbf{1}_{\bar{\Omega}}$, and then take the expectation to get
	\begin{eqnarray}\label{uq01}
			\lefteqn{\mathbb{E} \mathbf{1}_{\bar{\Omega}} \left\| v(\tilde{\tau})\right\|_{L^2_{\mathrm{div}}}^2 
			+ 2 \nu \mathbb{E} \mathbf{1}_{\bar{\Omega}} \int_0^{\tilde{\tau}}	\|v\|^2\, ds +2 \eta \mathbb{E} \mathbf{1}_{\bar{\Omega}} \int_0^{\tilde{\tau}} 
			\left\langle\mathcal{A}_r(v_1) - \mathcal{A}_r(v_2), v\right\rangle \, ds}\nonumber\\
            &&= -2 \mathbb{E} 	\mathbf{1}_{\bar{\Omega}} \int_{0}^{\tilde{\tau}} b_0( v, v_1, v)\, ds+ 2  \mathbb{E} \mathbf{1}_{\bar{\Omega}} \int_0^{\tilde{\tau}} 
			\langle z,  v \rangle ds + 2 \,\mathbb{E} \mathbf{1}_{\bar{\Omega}} \int_0^{\tilde{\tau}} 
			\left[\left(R_0(\mu_2, \phi), v\right) + \left(R_0(\mu, \phi_1), v\right) 	\right] ds \nonumber\\
			&&\quad+  \mathbb{E} \mathbf{1}_{\bar{\Omega}} \int_0^{\tilde{\tau}} 
			\|G_1(s, v_1) - G_1(s, v_2)\|_{\mathcal{L}_2(U_1, L^2_{\mathrm{div}})}^2  ds:=\sum_{k=1}^4 \mathfrak U_k.
	\end{eqnarray}
			Observe that the third term on the left-hand side is non-negative, that is, for \(r \geq 1\), we have
		\begin{eqnarray}
		\left\langle\mathcal{A}_r(v_1) - \mathcal{A}_r(v_2), v\right\rangle &=& 
				\left(\,|v_1|^{r-1}v_1 -|v_2|^{r-1}v_2, v\right)\nonumber\\
				&=& (\,|v_1|^{r-1},\, |v_1|^2) + (|v_2|^{r-1},\, |v_2|^2) -(v_1\cdot v_2 ,\,|v_1|^{r-1}+ |v_2|^{r-1})\nonumber\\
				&=& \frac{1}{2}\Big\|(\,|v_1|^{\frac{r-1}{2}})v\,\Big\|^2_{ L^2_{\mathrm{div}}} + \frac{1}{2}\Big\|(\,|v_2|^{\frac{r-1}{2}})v\,\Big\|^2_{ L^2_{\mathrm{div}}} + \frac{1}{2} \Big((|v_1|^2-|v_2|^2),\,(|v_1|^{r-1}-|v_2|^{r-1})\Big) \nonumber\\
				&\geq& \frac{1}{2}\Big\|(\,|v_1|^{\frac{r-1}{2}})v\Big\|^2_{ L^2_{\mathrm{div}}} + \frac{1}{2}\Big\|(\,|v_2|^{\frac{r-1}{2}})v\Big\|^2_{ L^2_{\mathrm{div}}} \geq 0.
				\label{uq3}    
		\end{eqnarray}

				    For $d = 2, 3$, we use \eqref{b0_4} and \eqref{r_04}, to obtain the following:
				$|b_0( v, v_1, v)| 
				\leq  C_{\nu} \|v_1\|^4 \| v\|^2_{ L^2_{\mathrm{div}}}+\frac{\nu}{10} \| v\|^2, \label{uq15}$ and 
				$|\left(R_0(\mu_2, \phi), v\right) |\leq  C_{\nu}\|\mu_2\|^{2}_{H^1}\|\phi\|^2_{H^1} + \frac{\nu}{10} \| v\|^2 $.
		Moreover, we have
	\begin{equation} \label{uq10}
		|\langle z,  v \rangle | + \|G_1(s, v_1) - G_1(s, v_2)\|_{\mathcal{L}_2(U_1, L^2_{\mathrm{div}})}^2 \leq C_{\nu}\|z\|^2_{V'} + \frac{\nu}{10}\|v\|^2 +C_{L_{L^2_{\mathrm{div}}}} \|v\|^2_{ L^2_{\mathrm{div}}}.
	\end{equation}
\noindent\textbf{Step 2.}
			Taking the duality of \eqref{uqe12} with \( ( \phi + A_1\phi )\) and using $\eqref{b1_1}_1$, we obtain
	\begin{equation}\label{uq02}
			\begin{aligned}
				& \mathbb{E} \mathbf{1}_{\bar{\Omega}} \left\|\phi(\tilde{\tau}) \right\|_{H^1}^2 
				+ 2 \mathbb{E} \mathbf{1}_{\bar{\Omega}} \int_0^{\tilde{\tau}}
				\left( \epsilon |A_1  \phi|^2_{ L^2} + \epsilon |A_1^{3/2}  \phi|^2_{ L^2} \right) ds \\
				& = - 2 \mathbb{E} \mathbf{1}_{\bar{\Omega}} \int_0^{\tilde{\tau}} b_1( v, \phi_1, \phi)\, ds - 2 \mathbb{E} \mathbf{1}_{\bar{\Omega}} \int_0^{\tilde{\tau}} 
				\left[ b_1( v, \phi_1, A_1 \phi) 
				+ b_1(v_2, \phi, A_1 \phi) \right] ds \\
				& \quad 
				+ 2\mathbb{E} \mathbf{1}_{\bar{\Omega}} \int_0^{\tilde{\tau}} 
				\left[((P\sigma_1-A- \alpha u_1)h(\phi_1)- (P\sigma_2-A- \alpha u_2)h(\phi_2), \phi + A_1\phi) \right] ds\\
				 & \qquad 
				+ 2 \epsilon^{-1} \mathbb{E} \mathbf{1}_{\bar{\Omega}} \int_0^{\tilde{\tau}} 
				 \left( \nabla(\psi'(\phi_1) - \psi'(\phi_2)),  \nabla \Delta \phi-\nabla \phi\right)  ds.
		\end{aligned}
	\end{equation}
			With the help of \eqref{b1_5} and Young's inequality, we derive  
			$|b_1( v, \phi_1, \phi)| 
				 \leq C_{\nu}|A_1^{1/2} \phi_1|^2_{ L^2}\|\phi\|^2_{H^1} + \frac{\nu}{10}\|v\|^2.$			
Recalling the embedding of $H^1 \hookrightarrow L^6$, the Gagliardo-Nirenberg inequality (see Lemma \ref{l1}) together with Remark \ref{emer}, we further deduce
		\begin{equation} 
			\begin{aligned}
				|b_1(v_2, \phi, A_1 \phi)| & \leq \|v_2\|_{L^6_{\mathrm{div}}}\|\nabla \phi\|_{L^3}|A_1\phi|_{ L^2}\leq \|v_2\|\,|A_1^{1/2} \phi|^{1/2}_{ L^2}\|\nabla \phi\|^{1/2}_{H^1}|A_1^{1/2}\phi|^{1/2}_{ L^2} |A_1^{3/2} \phi|^{1/2}_{ L^2}\\
				& \qquad\leq \|v_2\|\,|A_1^{1/2} \phi|_{ L^2}\|\phi\|^{1/2}_{H^2}|A_1^{3/2} \phi|^{1/2}_{ L^2}\leq C_{\epsilon} \|v_2\|^2\|\phi\|^2_{H^1} + \frac{\epsilon}{6} \Big(\,\|\phi\|^{2}_{H^2}+|A_1^{3/2}\phi|^2_{ L^2}  \Big).
			\end{aligned}
			\label{uq12}
		\end{equation}
				The above two estimates hold for \(d=2,3.\) 
				The assumptions on \(u_1\) and \(h\) ([\ref{[A3]}] and [\ref{[A4]}])  lead to
		\begin{equation} 
			\begin{aligned}
				&|\left((P\sigma_1-A- \alpha u_1)h(\phi_1)- (P\sigma_2-A- \alpha u_2)h(\phi_2),\, \phi + A_1\phi\right)|\\
				&  \qquad\leq |\left((P\sigma_1-A- \alpha u_1)(h(\phi_1)-h(\phi_2)) + h(\phi_2)(P\sigma-\alpha u),\, \phi + A_1\phi\right)|\\
				&\qquad \leq C_{\epsilon}|(P\sigma_1-A- \alpha u_1)(h(\phi_1)-h(\phi_2))|^2_{ L^2} + C_{\epsilon, P} |\sigma|^2_{ L^2}+C_{\epsilon, \alpha}|u|^2_{ L^2} + \frac{\epsilon}{6}\|\phi\|^2_{ H^2}\\
				&\qquad\leq C_{\epsilon, L_h,P,A, \alpha}\Big(\,\|\sigma_1\|^2_{H^1} \|\phi\|^2_{H^1} +|\phi|^2_{ L^2}+ |\sigma|^2_{ L^2} + |u|^2_{ L^2}\Big)+ \frac{\epsilon}{6}\|\phi\|^2_{ H^2}.
				\label{uq5}
			\end{aligned}
		\end{equation}
				Next, we have
		\begin{equation} 
			|( \nabla(\psi'(\phi_1) - \psi'(\phi_2)), \nabla \Delta \phi- \nabla \phi)|  
			\leq C_{\epsilon}|\nabla(\psi'(\phi_1) - \psi'(\phi_2))|^2_{ L^2} + \frac{\epsilon}{6}|\nabla \Delta \phi-\nabla \phi |^2_{ L^2}.
			\label{uq6}
		\end{equation}
			Since
			$\nabla(\psi'(\phi_1) - \psi'(\phi_2))  
            = \psi''(\phi_1)\nabla\phi + \nabla\phi_2(\psi''(\phi_1)-\psi''(\phi_2)),$ Assumption [\ref{[A1]}] gives:	
		\begin{equation} 
				|\psi''(\phi_1)\nabla\phi|^2_{L^2} \leq C_{\psi}
				(1+\|\phi_1\|^4_{L^{\infty}})^2|\nabla\phi|^2_{L^2} \leq C_{\psi}(1+\|\phi_1\|^8_{H^2})\|\phi\|_{H^1}^2.
				\label{uq102}
		\end{equation}
			Applying Assumption [\ref{[A1]}] (see, \eqref{eq:ass2}) and Gagliardo-Nirenberg's inequalities, we arrive at 
		\begin{equation} 
			\begin{aligned}
				|\nabla\phi_2(\psi''(\phi_1)-\psi''(\phi_2))|^2_{L^2} &\leq
				C_{\psi}
				\int_{\mathcal{O}}(1+|\phi_1|^3+|\phi_2|^3)^2|\phi|^2|\nabla\phi_2|^2 \,dx\\
				&\leq C_{\psi}  (1+\|\phi_1\|^6_{L^{18}}+\|\phi_2\|^6_{L^{18}})\|\nabla\phi_2\|^2_{L^6}\|\phi\|^2_{L^6}\\
				&\leq C_{\psi}  (1+\|\phi_1\|^{4}_{L^6}\|\phi_1\|^{2}_{H^2}+\|\phi_2\|^{4}_{L^6}\|\phi_2\|^{2}_{H^2} )\|\phi_2\|^2_{H^2}\|\phi\|^2_{H^1}\\
				&\leq C_{\psi}  (1+\|\phi_1\|^{8}_{H^1}\|\phi_1\|^{4}_{H^2}+\|\phi_2\|^{8}_{H^1}\|\phi_2\|^{4}_{H^2} +\|\phi_2\|^{4}_{H^2}  )\|\phi\|^2_{H^1},
				\label{uq103}
			\end{aligned}
		\end{equation}
				which is true for \( d=3\). Similarly for \( d=2\), we derive
	 \begin{equation} 
			\begin{aligned}
				|\nabla\phi_2(\psi''(\phi_1)-\psi''(\phi_2))|^2_{L^2}
				&\leq C_{\psi}  (1+\|\phi_1\|^{12}_{H^1}+\|\phi_2\|^{12}_{H^1} +\|\phi_2\|^{4}_{H^2}  )\|\phi\|^2_{H^1}.
				\label{uq103'}
			\end{aligned}
	\end{equation}		 
			\textbf{Step 3.}
			Applying the It\^o formula to \(|\sigma(\cdot)|^2_{ L^2}\) and using $\eqref{b1_1}_1$, we obtain
\begin{eqnarray}\label{uq03}
	 	&&\mathbb{E} \mathbf{1}_{\bar{\Omega}} \left| \sigma(\tilde{\tau}) \right|_{L^2}^2 
	 		+ 2 \mathbb{E} \mathbf{1}_{\bar{\Omega}} \int_0^{\tilde{\tau}}
	 		 |A_1^{1/2}\sigma|^2_{ L^2} ds =  - 2 \mathbb{E} \mathbf{1}_{\bar{\Omega}} \int_0^{\tilde{\tau}} b_1( v, \sigma_1, \sigma)\, ds - 2b \mathbb{E} \mathbf{1}_{\bar{\Omega}} \int_0^{\tilde{\tau}}(|\sigma|^2_{L^2} - \langle w, \sigma\rangle) ds  \nonumber\\ 
	 		&&\quad - 2 c\mathbb{E} \mathbf{1}_{\bar{\Omega}} 	\int_0^{\tilde{\tau}}( \sigma_1h(\phi_1)-\sigma_2h(\phi_2) ,\sigma)  ds+ \mathbb{E} \mathbf{1}_{\bar{\Omega}} \int_0^{\tilde{\tau}} 
	 		\|G_2(s, \sigma_1) - G_2(s, \sigma_2)\|_{\mathcal{L}_2(U_2, L^2)}^2 \, ds.
\end{eqnarray}
					The trilinear term \(b_1(\cdot, \cdot, \cdot)\), using \eqref{b1_5}, is estimated  as
					$|b_1( v, \sigma_1, \sigma)|   \leq \frac{\nu}{10}	\|v\|^2 + \frac{1}{6}\|\sigma\|^{2}_{H^1} + C_{\nu} \|\sigma_1\|^4_{H^1}|\sigma|^2_{ L^2}$
			and Assumption [\ref{[A3]}] on \(h\) implies 
		\begin{equation} 
				\begin{aligned}
					|( \sigma_1h(\phi_1)-\sigma_2h(\phi_2) ,\sigma) | =
					|\left(\sigma_1(h(\phi_1)- h(\phi_2)) +\sigma h(\phi_2),\sigma\right)|
                    \leq C_{ L_h}\|\sigma_1\|^2_{H^1}\|\phi\|^2_{H^1}+ 2|\sigma|^2_{ L^2}.
				\end{aligned}
				\label{uq8}
			\end{equation}
				For the source term \(w\) and stochastic forcing \(G_2\), we have the following estimate:
\begin{equation} \label{uq11}
			|\langle w,  \sigma \rangle| + \|G_2(s, \sigma_1) - G_2(s, \sigma_2)\|_{\mathcal{L}_2(U_2, L^2)}^2   \leq C\|w\|^2_{(H^1)^'} + 	\frac{1}{6}\|\sigma\|^2_{H^1}+ C_{L_{L^2}} |\sigma|^2_{ L^2}.
\end{equation}
       The above three estimates are valid for $d=2,3$.
        
\medskip
\noindent\textbf{Step 4 (Compiling Steps 1-3).}				          For a fixed \( k\geq 0\), we define the following               stopping times:
\begin{equation*}
	\begin{aligned}
		\gamma_k^1 &:= \inf \{t \in [0,T] : \|(v_1, \phi_1,\sigma_1)(t\wedge \tau_1)\|^2_{\mathcal{V}} \, >k \}, \\
		\gamma_k^2 &:= \inf \{t\in [0,T]  : \| (v_2, \phi_2,\sigma_2)(t\wedge \tau_2)\|^2_{\mathcal{H}} + \|{\mathbf{1}}_{t \leq \tau_2}v_2\|^2_{L^2(0,t;V)} + \|{\mathbf{1}}_{t \leq \tau_2}\phi_2\|^4_{L^4(0,t;H^2)} + \|{\mathbf{1}}_{t \leq \tau_2}\mu_2\|^2_{L^2(0,t;H^1)} \, > k\},\\
		\gamma_k &:= \gamma_k^1  \wedge \gamma_k^2 \wedge \tau.
	\end{aligned}
\end{equation*}
		
        Since $\{(v_1, \phi_1,\sigma_1), \tau_1\}$ and  $\{(v_2, \phi_2,\sigma_2), \tau_2\}$ are, respectively, local strong and weak solutions, by means of convergence of sequences $\gamma_k^1 \to \tau_1$ and $\gamma_k^2 \to \tau_2$, we have $\gamma_k \to \tau$ as $ k \to \infty$ a.s..
        
        By adding  the identities \eqref{uq01},\eqref{uq02} and \eqref{uq03}, we notice that $\left(R_0(\mu, \phi_1), v\right) $ and $-b_1( v, \phi_1, A_1 \phi)$ cancel each other (see Remark \ref{R0B1}).  Consequently, invoking the estimates obtained in Steps 1 to 3 with the stopping time $\gamma_k \wedge t$, we infer at
 \begin{eqnarray}
	 		\mathfrak U &:=&\mathbb{E}\mathbf{1}_{\bar{\Omega}} \left\|(v(\gamma_k \wedge t), \phi(\gamma_k \wedge t), \sigma(\gamma_k \wedge t)) \right\|_{\mathcal{H}}^2 
	 		+ \mathbb{E}\mathbf{1}_{\bar{\Omega}} \int_0^{\gamma_k \wedge t} \left( \nu \|  v \|^2 +  \epsilon|A_1 \phi|^2_{ L^2}+  \epsilon|A_1^{3/2} \phi|^2_{ L^2} + \|\sigma\|^2_{H^1} \right)ds \nonumber\\
	 		& \leq& C_1\, \mathbb{E}\mathbf{1}_{\bar{\Omega}} \int_0^{\gamma_k \wedge t} \left(1 +  \| v_2 \|^2  +\|\phi_2\|^4_{H^2}+ \|\mu_2\|^2_{H^1} \right)\left\| ( v(s),  \phi(s) , \sigma(s)) \right\|_{\mathcal{H}}^2 ds \nonumber\\
	 		&& + C_2 \,\mathbb{E}\mathbf{1}_{\bar{\Omega}} \int_0^{\gamma_k \wedge t} \big(~\|z\|^2_{V'} + |u|^2_{ L^2}+ \|w\|^2_{(H^1)^'}\big) \,ds,  \label{uq16}
	 \end{eqnarray}
	 		where $ C_1 := C_1(\nu,\epsilon,P,A, L_h,L_{H},L_{L^2}, \alpha,  C_{\psi}, k)$ and $ C_2 := C_2(\nu,\epsilon, \alpha).$ 
			 We also note that for \(v_2\), \(\phi_2\) (cf. \eqref{hr1}) and  \(\mu_2\), we have
	 		$\int_{0}^{\gamma_k \wedge t}(1+ \| v_2 \|^2  +\|\phi_2\|^4_{H^2} + \|\mu_2\|^2_{H^1}) ds \leq \tilde{C} \ \text{a.s.},$
			 for some suitable constant \(\tilde{C} > 0\) depending on $T$ and $k$.
		   	Moreover, for uniqueness, we set \(z,u\) and \(w = 0\) a.s.. Hence, it follows from Gronwall's lemma (see Lemma \ref{l3}) that
			 $\mathbb{E}\mathbf{1}_{\bar{\Omega}} \left\| (v(\gamma_k \wedge t),  \phi(\gamma_k \wedge t),\sigma(\gamma_k \wedge t)) \right\|_{\mathcal{H}}^2 = 0,$ which concludes the proof as \(k \to \infty\).
	 \end{proof}
\begin{Rem}
         Though, we have formally established the energy estimate (Lemma \ref{Ener_est}) for the local weak solution of the stochastic CH-CBF reaction diffusion system (for $r \geq 1$) defined on a bounded domain $ \mathcal{O} \subset \mathbb{R}^d,d=2,3$ that can be rigorously justified, and the existence of a weak solution of CH-CBF system can be proven by following the stochastic Galerkin approximations and compactness arguments (see for example, \cite{Deugoue2022,Orrieri2020}). 
\end{Rem}
     
Next, we prove the \emph{uniqueness of weak solutions} of the CH-CBF system under a suitable condition only on $\sigma$ of \eqref{a3} in $d=3.$ 
\begin{Pro}(Uniqueness of weak solutions)\label{weakuniq}  
		Let \( \tau_1, \tau_2 > 0 \) be stopping times. Set $\tau := \tau_1 \wedge \tau_2$. Suppose that the assumptions of Proposition \ref{weak-strong} are satisfied. We assume that \( \{(v_1, \phi_1, \sigma_1), \tau_1\} \) and \( \{(v_2, \phi_2,\sigma_2), \tau_2\} \) are weak solutions to the system \eqref{a1}-\eqref{a5} with the associated sources \((z_1, u_1, w_1)\) and  \((z_2, u_2, w_2)\) and initial data \( \Big(v_{1}, \phi_{1},\sigma_{1}\Big)(0)\) and \( (v_{2}, \phi_{2},\sigma_{2})(0) \) such that $\bar{\Omega} = \{\big(v_{1}, \phi_{1},\sigma_{1} \big)(0) = \big(v_{2}, \phi_{2},\sigma_{2} \big)(0)\} \subseteq \Omega$, then the pathwise uniqueness (see Definition\ref{pathuq}) holds in $d = 2$ for all $\eta > 0$, $\nu > 0$, and $r \geq 1$, while in $d = 3,$ it holds for $r \geq 3$ under the assumption that $\sigma \in L^{4}(\Omega; L^{4}([0,T];H^1))$ with $\eta, \nu > 0$ when $r > 3$, and $\eta\nu \geq 1$ when $r = 3.$ 
			\end{Pro}
			\begin{proof} 
            We shall directly make use of the identities \eqref{uq01},\eqref{uq02} and \eqref{uq03} in Proposition \ref{weak-strong} to prove this result. We derive only the estimates that differ from the previous proposition. \medskip\\
		\noindent {\bf 2D Case.} For $r > 3,$	let us look at the terms in  \eqref{uq01}.
				We estimate the trilinear term as follows:
				\begin{equation}\label{weakuni3}
					\begin{aligned}
						|b_0( v, v_1, v)|= |b_0( v, v, v_1)|\leq \frac{\nu}{2}\|v\|^2+\frac{1}{2\nu}\|v_1v\|^2_{ L^2_{\mathrm{div}}},
					\end{aligned}
				\end{equation}
				and for $r>3,$ H\"older's and Young's 
                inequalities yield (see \cite{Hajduk2017}) 
\begin{align*}\label{weakuni4}
		      \|v_1v\|^2_{ L^2_{\mathrm{div}}} &= \int_{\mathcal{O}} |v_1|^2|v|^{\frac{4}{r-1}} |v|^{\frac{2(r-3)}{r-1}}\,dx \leq \Big\|(\,|v_1|^{\frac{r-1}{2}})v\Big\|^{\frac{4}{r-1}}_{ L^2_{\mathrm{div}}} \|v\|^{\frac{2(r-3)}{r-1}}_{ L^2_{\mathrm{div}}}\\
		       &\leq \eta \nu \Big\|(\,|v_1|^{\frac{r-1}{2}})v\Big\|^2_{ L^2_{\mathrm{div}}} + \left(\frac{2}{\eta\nu(r-1)}\right)^{\frac{2}{r-3}}\frac{r-3}{r-1}\, \|v\|^2_{ L^2_{\mathrm{div}}}.
	\end{align*}
             Consequently, we have		 
	\begin{equation}\label{weakuni5}
					|b_0( v, v_1, v)| \leq \frac{\nu}{2}\|v\|^2+ \frac{\eta}{2} \left\|(\,|v_1|^{\frac{r-1}{2}})v\right\|^2_{ L^2_{\mathrm{div}}}+ C_{r, \eta, \nu}\|v\|^2_{ L^2_{\mathrm{div}}}, \ \ C_{r, \eta, \nu} = \left(\frac{2}{\eta\nu(r-1)}\right)^{\frac{2}{r-3}}\frac{r-3}{2\nu(r-1)}.
    \end{equation}
				Next, for the Forchheimer term, from \eqref{uq3}, we have
		\begin{equation} \label{weakuni6}
					\left\langle\mathcal{A}_r(v_1) - \mathcal{A}_r(v_2), v\right\rangle  \geq \frac{1}{2}\left\|(\,|v_1|^{\frac{r-1}{2}})v\right\|^2_{ L^2_{\mathrm{div}}} + \frac{1}{2}\left\|(\,|v_2|^{\frac{r-1}{2}})v\right\|^2_{ L^2_{\mathrm{div}}} \geq \frac{1}{2}\left\|(\,|v_1|^{\frac{r-1}{2}})v\right\|^2_{ L^2_{\mathrm{div}}}.
				\end{equation}
				The remaining terms of the right-hand side of \eqref{uq01} follow the estimates of Step 1 of the previous proof.\\
				Next, the right-hand side terms of \eqref{uq02} (except \eqref{uq102}) can be controlled identically by the arguments of Step 2 in Proposition \ref{weak-strong}. Using Agmon's inequality (Lemma \ref{l2}), we obtain a modified estimate for \eqref{uq102}:
				\begin{equation}\label{weakuni7}
					|\psi''(\phi_1)\nabla\phi|^2_{L^2} 	\leq C_{\psi}
						(1+\|\phi_1\|^8_{L^{\infty}})|\nabla\phi|^2_{L^2}\leq 	C_{\psi}(1+|\phi_1|^4_{L^2}\|\phi_1\|^4_{H^2})\|\phi\|_{H^1}^2.
				\end{equation}
                Finally, for the evaluation of \eqref{uq03}, the trilinear estimate, in view of \eqref{b1_1}, is modified as follows:
\begin{equation}\label{weakuni8}
					\begin{aligned}
						|b_1( v, \sigma_1, \sigma)|   \leq
						\| v\|_{L^4_{\mathrm{div}}}|\nabla \sigma|_{L^2} \|\sigma_1\|_{L^4} 
						\leq\frac{\nu}{10}	\|v\|^2 + \frac{1}{6}\|\sigma\|^{2}_{H^1} + C_{\nu}|\sigma_1|^2_{ L^2} \|\sigma_1\|^2_{H^1}\| v\|_{L^2_{\mathrm{div}}}^2.
					\end{aligned}
				\end{equation}
            For a fixed \( k\geq 0\), we define the following stopping times:
				\begin{equation*}
					\begin{aligned}
						\gamma_k^{21} &:= \inf \{t \in [0,T] : \|(v_1, \phi_1,\sigma_1)(t\wedge \tau_1)\|^2_{\mathcal{H}} + \|{\mathbf{1}}_{t \leq \tau_1}\phi_1\|^4_{L^4(0,t;H^2)}+ \|{\mathbf{1}}_{t \leq \tau_1}\sigma_1\|^2_{L^2(0,t;H^1)} \, >k \}, \\
						\gamma_k^{22} &:= \inf \{t\in [0,T]  : \| (v_2, \phi_2,\sigma_2)(t\wedge \tau_2)\|^2_{\mathcal{H}} + \|{\mathbf{1}}_{t \leq \tau_2}v_2\|^2_{L^2(0,t;V)} + \|{\mathbf{1}}_{t \leq \tau_2}\phi_2\|^4_{L^4(0,t;H^2)} + \|{\mathbf{1}}_{t \leq \tau_2}\mu_2\|^2_{L^2(0,t;H^1)} \, > k\},\\
						\gamma_k^{2} &:= \gamma_k^{21}  \wedge \gamma_k^{22} \wedge \tau.
					\end{aligned}
				\end{equation*}
				Since $\{(v_1, \phi_1,\sigma_1), \tau_1\}$ and  $\{(v_2, \phi_2,\sigma_2), \tau_2\}$ are local weak solutions, via the convergence of sequences $\gamma_k^{21} \to \tau_1$ and $\gamma_k^{22} \to \tau_2$, we have $\gamma_k^{2} \to \tau$ as $ k \to \infty$ a.s..\\
                In the proof of Proposition \ref{weak-strong}, we change the following: The estimates of $b_0(v,v_1,v),$ \eqref{uq3}, \eqref{uq102} and $b_1(v,\sigma_1,\sigma)$ are replaced by \eqref{weakuni5},  \eqref{weakuni6}, \eqref{weakuni7} and \eqref{weakuni8}, respectively.   
                Thus, the estimate \eqref{uq16} with the new stopping time $\gamma_k^{2} \wedge t$ becomes as follows:
				\begin{equation}				
                \begin{aligned}
						\mathfrak U & \leq C_1\, \mathbb{E}\mathbf{1}_{\bar{\Omega}} \int_0^{\gamma_k^{2} \wedge t} \left(1 +\|\phi_1\|^4_{H^2}+  \| \sigma_1 \|^2_{H^1}  +  \| v_2 \|^2  +\|\phi_2\|^4_{H^2}+ \|\mu_2\|^2_{H^1} \right)\left\| ( v(s),  \phi(s) , \sigma(s)) \right\|_{\mathcal{H}}^2 ds\\
						&\quad\qquad + C_2 \,\mathbb{E}\mathbf{1}_{\bar{\Omega}} \int_0^{\gamma_k^{2} \wedge t} \big(~\|z\|^2_{V'} + |u|^2_{ L^2}+ \|w\|^2_{(H^1)^{\prime}}\big) \,ds, \label{uq16a} 
					\end{aligned}
				\end{equation}
				where $ C_1 := C_1(\nu,\eta, r,\epsilon,P,A, L_h,L_{H},L_{L^2}, \alpha,  C_{\psi}, k)$ and $ C_2 := C_2(\nu,\epsilon, \alpha).$ 
				Fom the definition of stopping times defined above, we infer that
				$	\int_{0}^{\gamma_k^{2} \wedge t}(1+\|\phi_1\|^4_{H^2}+  \| \sigma_1 \|^2_{H^1}+ \| v_2 \|^2  +\|\phi_2\|^4_{H^2} + \|\mu_2\|^2_{H^1}) ds \leq \tilde{C} \ \text{a.s.}.$
				Hence, by stochastic Gronwall's inequality, we can conclude the uniqueness for $r>3.$

			For the case of $r \in [1,3],$ we use the fact that $ v_1 \in L^{r+1}(\Omega; L^{r+1}([0,T];L^{r+1}_{\mathrm{div}}))$ as a weak solution (cf.\eqref{eqequality}) and non negativity of the term $\left\langle\mathcal{A}_r(v_1) - \mathcal{A}_r(v_2), v\right\rangle$ (cf. \eqref{uq3}). For the trilinear operator \(b_0( v, v_1, v)\), we use the same estimate obtained in Step 1 of Proposition \ref{weak-strong}.
				Thus, the estimate \eqref{uq16a} holds true with an additional term $\|v_1\|^4_{ L^4_{\mathrm{div}}}$ on the first right-hand side integral. By redefining the stopping time $	\gamma_k^{21},$ and then invoking $\int_{0}^{\gamma_k^{2} \wedge t}\|v_1\|^4_{ L^4_{\mathrm{div}}}\,ds < C$ a.s., we can conclude the uniqueness. \\
                
				\noindent {\bf 3D Case.} For $r > 3,$ we use the same bounds obtained in \eqref{weakuni5} and \eqref{weakuni6} along with other estimates in Step 1 of Proposition \ref{weak-strong} to bound the right-hand side of \eqref{uq01}. We treat the terms of \eqref{uq02}, except \eqref{uq102}, as in Step 2. By applying Agmon's inequality  $\|y\|_{L^{\infty}} \leq \|y\|^{3/4}_{H^1}\|y\|^{1/4}_{H^3}$, we derive the modified estimate for \eqref{uq102} as follows: 
				\begin{equation}\label{weakuni10}
					|\psi''(\phi_1)\nabla\phi|^2_{L^2} 	\leq C_{\psi}
						(1+\|\phi_1\|^8_{L^{\infty}})|\nabla\phi|^2_{L^2}
						\leq	C_{\psi}(1+\|\phi_1\|^6_{H^1}\|\phi_1\|^2_{H^3})\|\phi\|^2_{H^1}.
				\end{equation}
                Finally, we come to the identity \eqref{uq03}, where we use the same estimates derived in Step 3, and for the trilinear term $b_1(v,\sigma_1,\sigma)$ bound, we invoke the assumption that $\sigma \in L^{4}(\Omega; L^{4}([0,T];H^1)).$  
                
				Thus, in view of the above modification, the first integral on the right-hand side of \eqref{uq16a} is modified by including $\|\phi_1\|^2_{H^3}$ and $\|\sigma_1\|^4_{H^1}.$   Again, by redefining the stopping time, and then using $\int_{0}^{\gamma_k^{2} \wedge t}\Big(\|\phi_1\|^2_{H^3}+\|\sigma_1\|^2_{H^1}+\|\sigma_1\|^4_{H^1}+  \| v_2 \|^2  +\|\phi_2\|^4_{H^2}+ \|\mu_2\|^2_{H^1} \Big)\,ds < C$ a.s., we obtain the proof of this case.
				
				Finally, for the critical case $r = 3,$ we use the trilinear term estimate \eqref{weakuni3}, while for the Forchheimer term, we set 
				$r=3$ in \eqref{weakuni6} to obtain
				$		\left\langle\mathcal{A}_r(v_1) - \mathcal{A}_r(v_2), v\right\rangle   \geq \frac{1}{2}\left\|\,v_1v\right\|^2_{ L^2_{\mathrm{div}}}. $
				Therefore, from \eqref{uq01} for any $\tilde{\tau} \leq \tau$ provided $\nu\eta\geq 1$, we derive
				\begin{equation}\label{weakuni12}
					\begin{aligned}	 
						& \mathbb{E} \mathbf{1}_{\bar{\Omega}} \left\| v(\tilde{\tau})\right\|_{L^2_{\mathrm{div}}}^2 +  \mathbb{E} \mathbf{1}_{\bar{\Omega}} \int_0^{\tilde{\tau}}	\nu \|v\|^2\, ds + 	\left(\eta- \frac{1}{\nu}\right)\,\mathbb{E} \mathbf{1}_{\bar{\Omega}} \int_0^{\tilde{\tau}} 
						\left\|\,v_1v\right\|^2_{ L^2_{\mathrm{div}}} \, ds \leq \mathfrak U_2+ \mathfrak U_3+\mathfrak U_4. 
			    \end{aligned}
		\end{equation}
                Thus, one can finish the proof by combining the estimates for \eqref{uq02} and \eqref{uq03} from the same line of arguments described above for $r>3$ in the 3D case with the estimates of the right-hand side terms of \eqref{weakuni12}.
\end{proof}
            \begin{Rem}	One can readily discern that in the scenario where $\sigma = 0$, the stochastic CH-CBF system exhibits a unique global weak solution, in contrast to the stochastic CH-NS system  (\cite{Deugoue2021}), where a local weak-strong uniqueness is proven in the absence of the Forchheimer term.
\end{Rem}
            
\subsection{Existence of a local strong solution}
	\begin{Pro}(Existence of a local strong solution)\label{exismain_bdd}
	 		    	For $d=2,3$ and $ r \geq 1 $ in $d=2$, $ r \in [1,3] $ in $d=3,$ suppose initial data \((v_0, \phi_0, \sigma_0) \in L^2(\Omega; \mathcal{V})\) and that the sources $(z,u,w) \in L^2(\Omega; L^2(0,T; \mathcal{H}))$ with $u \in [0,1] $ a.e. in $\Omega \times (0,T)\times \mathcal{O}$ for some \(T >0\). Further, the noise coefficients, $G_i, i= 1,2$ and function h satisfy [\ref{[A2]}] and [\ref{[A3]}], respectively. Then there is a local strong solution $\{(v, \phi, \sigma); \tau\}$ of \eqref{a1}-\eqref{a5} as defined in Definition \ref{def4.2}.
	           \end{Pro}
	  \begin{proof} 
                We divide the proof into two steps:\\
                \textbf{Step 1.} First we assume that \(\|(v_{0}, \phi_{0}, \sigma_{0})\|_{\mathcal{V}} \leq \tilde{K}\) a.s. for some \(\tilde{K} \geq 0\), i.e., \( (v_0, \phi_0, \sigma_0) \in L^{\infty}(\Omega; \mathcal{V})\).
                We fix \(M>1\) and \(T > 0\) as in Proposition \ref{nes_com}. Let \(\{ (v_n, \phi_n, \sigma_n)\}_{n \in \mathbb{N}}\) be the associated sequence of Galerkin solutions of the problem \eqref{gal21}-\eqref{gal25}. As this sequence adheres to both  \eqref{first} and \eqref{secd}, we can invoke the pairwise comparison theorem in \cite[Lemma 5.1, part (i)]{Glatt-Holtz2009}, for the spaces \( B_1=\mathcal{V}\) and \( B_2=\mathcal{Z}\) and the sequence \(\{X_n\} = \{(v_n, \phi_n, \sigma_n)\}\). Hence, one can extract a subsequence \(\{(v_{n_k}, \phi_{n_k}, \sigma_{n_k} )\}\), a strictly positive stopping time, \(\tau \leq T\) and a process \((v(\cdot), \phi(\cdot), \sigma(\cdot))\), continuous in \(\mathcal{V}\), such that \((v(\cdot), \phi(\cdot), \sigma(\cdot)) = (v(\cdot \wedge \tau), \phi(\cdot \wedge \tau), \sigma(\cdot \wedge \tau))\). In addition, we have 
	\begin{equation}
	  	\begin{aligned}
	  		&\sup_{t \in [0,\tau]} \left\| (v_{n_k}(t), \phi_{n_k}(t),\sigma_{n_k}(t)) - (v(t), \phi(t), \sigma(t)) \right\|^2_{\mathcal{V}} \ \to 0  \quad \text{a.s.,}\\ 
	  		& \int_0^{\tau} \Big( \nu \| A_0(v_{n_k} - v) \|^2_{L^2_{\mathrm{div}}} 
	  		+ \epsilon \| \phi_{n_k} - \phi \|^2_{H^4} + \| \sigma_{n_k} - \sigma \|^2_{H^2} \Big) ds \ \to 0  \quad \text{a.s.}.
	  		\label{4.21}\end{aligned}
	  \end{equation}
	 		  Finally, we also observe that the, \((v_{n}^{(0)}, \phi_{n}^{(0)}, \sigma_{n}^{(0)})\) satisfies the condition (ii) of Lemma 5.1 in \cite{Glatt-Holtz2009} for any \(p \in [1, \infty)\). Hence, for any such \(p\), we have \((v(\cdot \wedge \tau), \phi(\cdot \wedge \tau), \sigma(\cdot \wedge \tau)) \in L^p(\Omega; C([0, T]; \mathcal{V}))\). Moreover, the triplet of truncated processes (\(\mathbf{1}_{\{t \leq \tau\}} v\), \(\mathbf{1}_{\{t \leq \tau\}} \phi\), \(\mathbf{1}_{\{t \leq \tau\}} \sigma\)) belongs to \(L^p(\Omega; L^2(0, T; \mathcal{Z}))\). Again, from the  comparison theorem (see \cite[Lemma 5.1, part (ii)]{Glatt-Holtz2009}), we also obtain a collection of measurable sets \(\Omega_{n_k} \in \mathcal{F}\) with \(\Omega_{n_k} \uparrow \Omega\) such that
 \begin{eqnarray}
	 		&&\sup_{n_k} \mathbb{E} \mathbf{1}_{\Omega_{n_k}} 
	  		\left( \sup_{t \in [0,T]} \| (v_{n_k}(t \wedge \tau), \phi_{n_k}(t \wedge \tau), \sigma_{n_k}(t \wedge \tau)) \|^2_{\mathcal{V}} \right. \nonumber\\
	  		&&\left.\ \ \ \qquad + \int_0^{\tau} \Big( \nu \|A_0 v_{n_k}\|^2_{L^2_{\mathrm{div}}}+ \epsilon \| \phi_{n_k}\|^2_{H^4} + \| \sigma_{n_k}\|^2_{H^2} \Big) ds \right)^{p/2}< \infty. \label{4.23}
	  \end{eqnarray}
	  		Next, by virtue of \eqref{4.21}, \eqref{4.23}, and Lemma 5.2 in \cite{Glatt-Holtz2009}, we deduce that
\begin{eqnarray}\label{conver}	 	
	 	   &&\hspace{-.2in}\mathbf{1}_{\Omega_{n_k}} \mathbf{1}_{\{t \leq \tau\}}\,\left( v_{n_k},\phi_{n_k}, \sigma_{n_k}) \to \mathbf{1}_{\{t \leq \tau\}} \Big(v,\phi,\sigma\right) \, \text{ weakly in}\; L^p\Big(\Omega; L^2(0, T; \mathcal{Z})\Big),\\
			&&\hspace{-.2in}\mathbf{1}_{\Omega_{n_k}} \, \Big(v_{n_k}(\cdot \wedge \tau),\phi_{n_k}(\cdot \wedge \tau),\sigma_{n_k}(\cdot \wedge \tau)\Big) \to (v,\phi,\sigma) \; \text{weakly-star in}\; L^p\left(\Omega; L^\infty(0, T; \mathcal{V})\right). \nonumber
\end{eqnarray}	 

	 		 After all these arguments, we are now only left with showing the identity \eqref{regu2} in \( L^2_{\mathrm{div}} \times L^2\times L^2 \). From \cite[Remark 4.1, part (ii)]{Glatt-Holtz2009} we observe that for any \((\theta_1,\theta_2,\theta_3 ) \in (L^2_{\mathrm{div}} \times L^2\times L^2) \) and any measurable set \( K \subset \Omega \times [0,T]\), it is equivalent to demonstrate that 
	   \begin{equation} \label{finaidene1}
	  	\begin{aligned}
	  		&\mathbb{E}\, \int_{0}^{T} \chi_{K} (	v(t \wedge \tau), \theta_1 ) \, dt = \;\mathbb{E}\, \int_{0}^{T} \chi_{K} \Bigg[\big (v_0, \theta_1\big)- \int_0^{t \wedge \tau} \Bigg\langle A_0 v +\eta \mathcal{A}_r(v) + B_0(v,v)  \\ 
	  		& \qquad - \epsilon R_0 (A_1 \phi, \phi) - z, \theta_1 \Bigg\rangle \, ds\Bigg] \,dt+ \mathbb{E}\, \int_{0}^{T} \chi_{K}\Bigg[\int_0^{t \wedge \tau} \sum_{k=1}^{\infty} \big\langle g_{1,k}(s,v) , \theta_1 \big\rangle d\beta_k^1(s) \Bigg] \,dt, 
	  	\end{aligned}
	  \end{equation}
	  \begin{equation}\label{finaidene2}
	  	\begin{aligned}	 
	  		&\mathbb{E}\, \int_{0}^{T} \chi_{K}	\big(\phi(t \wedge \tau),\theta_2 \big) \,dt =\; \mathbb{E}\, \int_{0}^{T} \chi_{K} \left[\big(\phi_0, \theta_2 \big)- \int_0^{t \wedge \tau} \Bigg\langle \epsilon A_1^2 \phi + \epsilon^{-1}A_1 \psi'(\phi) + B_1(v,\phi), \theta_2\Bigg\rangle  \,ds \right] \,dt\\ &\qquad +\mathbb{E}\, \int_{0}^{T} \chi_{K} \Bigg[\int_0^{t \wedge \tau} \big\langle \big(P\sigma - A - \alpha u \big) h(\phi), \theta_2\big\rangle\,ds\Bigg]\,dt, \;\ \ \ \mu = \epsilon A_1 \phi + \epsilon^{-1} \psi'(\phi),  
	  	\end{aligned}
	  \end{equation}
	  \begin{equation}\label{finaidene3}
	  	\begin{aligned}
	  		&\mathbb{E}\, \int_{0}^{T} \chi_{K}	\big(\sigma(t \wedge \tau), \theta_3 \big) \,dt=\; \mathbb{E}\, \int_{0}^{T} \chi_{K} \Bigg[\big(\sigma_0, \theta_3\big) -  \int_0^{t \wedge \tau} \Bigg\langle A_1 \sigma + B_1(v,\sigma)+ c \sigma h(\phi) \\
	  		 & \qquad+ b(\sigma - w), \theta_3 \Bigg\rangle\,  ds \Bigg]\,dt +\mathbb{E}\, \int_{0}^{T} \chi_{K} \Bigg[ \int_0^{t \wedge \tau} \sum_{k=1}^{\infty} \big\langle g_{2,k}(s,\sigma) , \theta_3 \big\rangle d\beta_k^2(s) \Bigg] \,dt.
	  	\end{aligned}	
	  \end{equation}
        \noindent\textbf{Substep 1.}	First, consider the approximate equation below, and then justify passing the limit. For any fixed \(\theta_1 \in V\),  from \eqref{gal21} and \eqref{conver}, we obtain 
 \begin{equation}\label{exvest} 
	  	\begin{aligned}
	  	&\mathbb{E}\, \int_{0}^{T} \chi_{K} (	v(t \wedge \tau), \theta_1 ) \, dt = \lim_{n_k \to \infty}
	 	 \mathbb{E}\, \int_{0}^{T} (\mathbf{1}_{\Omega_{n_k}}	v_{n_k}(t \wedge \tau),\; \chi_{K} \theta_1 ) \, dt \\
         & \ \ = \lim_{n_k \to \infty}\left( \mathbb{E}\, \int_{0}^{T} \chi_{K}\mathbf{1}_{\Omega_{n_k}} \left[\big (\mathcal{P}_1^{n_k}v_0, \theta_1\big)- \int_0^{t \wedge \tau} \langle \nu A_0 v_{n_k} + \eta \mathcal{P}_1^{n_k}\mathcal{A}_r(v_{n_k}) +  \mathcal{P}_1^{n_k}B_0(v_{n_k},v_{n_k}), \theta_1 \rangle \, ds \right.\right.\\&\qquad \left.\left.+ \epsilon\int_0^{t \wedge \tau}\langle \mathcal{P}_1^{n_k}R_0( A_1 \phi_{n_k}, \phi_{n_k}) + \mathcal{P}_1^{n_k}z, \theta_1\rangle \,ds+ \int_0^{t \wedge \tau}\sum_{k=1}^{\infty} \big\langle \mathcal{P}_1^{n_k}g_{1,k}(s,v_{n_k}) e_k^1, \theta_1 \big\rangle d\beta_k^1(s) \right]dt \right). 
	  \end{aligned}
	\end{equation}
		Using \eqref{for1}, generalized H\"older's inequality, and Lemma \ref{norm_esti}, we derive		
		\begin{equation}
			\begin{aligned}
				&|\langle \mathcal{P}_1^{n_k} \mathcal{A}_r(v_{n_k}) - \mathcal{A}_r(v), \theta_1\rangle|\leq |\langle\mathcal{A}_r(v_{n_k})-\mathcal{A}_r(v), \mathcal{P}_1^{n_k} \theta_1\rangle| + |\langle\mathcal{A}_r(v), \mathcal{Q}_1^{n_k} \theta_1 \rangle|\\
				&\qquad\leq C_r \|v_{n_k}-v\|_{L^e_{\mathrm{div}}}\,\|~|v_{n_k}| +|v|~\|^{(r-1)}_{L^{l(r-1)}_{\mathrm{div}}}~\|\mathcal{P}_1^{n_k} \theta_1\|_{L^2_{\mathrm{div}}} +\|v\|^{r}_{L^{2r}_{\mathrm{div}}} \|\mathcal{Q}_1^{n_k} \theta_1\|_{L^2_{\mathrm{div}}}\\
				&\qquad\leq C_r \|v_{n_k}-v\|_{L^e_{\mathrm{div}}}~\left(\|v_{n_k}\|^{(r-1)}_{L^{l(r-1)}_{\mathrm{div}}} 	+\|v\|^{(r-1)}_{L^{l(r-1)}_{\mathrm{div}}}\right)~\|\theta_1\|_{L^2_{\mathrm{div}}} + \frac{1}{\lambda_{n_k}^{1/2}}\|v\|^{r}_{L^{2r}_{\mathrm{div}}}\|\theta_1\|\\
                &\qquad\leq C_r \|v_{n_k}-v\|\left(~\|v_{n_k}\|^{(r-1)} +\|v\|^{(r-1)}~\right)\|\theta_1\|_{L^2_{\mathrm{div}}} + \frac{1}{\lambda_{n_k}^{1/2}}\|v\|^{r}\|\theta_1\|. \label{eslim2}
			\end{aligned}
		\end{equation}
	where $e, l >2$ such that $\frac{1}{e} + \frac{1}{l} + \frac{1}{2} = 1.$	Further, for the  last inequality, as in Proposition~\ref{nes_com} (cf. \eqref{eststa1}), we used the embedding $\mathbb{H}^1 \hookrightarrow \mathbb{L}^p$ to get $r \in (1, \infty)$ (in $d = 2$) and $r \in (1,3]$ (in $d = 3$) for any H\"older exponents $e, l$, satisfying $\frac{1}{e} + \frac{1}{l} = \frac{1}{2}$. 
    
	By invoking \eqref{4.21} and \eqref{eslim2}, we obtain for all \(\theta_1 \in V\) that
		\begin{equation}\label{eslim21}
				\mathbf{1}_{\{t \leq \tau\}}\,\langle \mathcal{P}_1^{n_k} 	\mathcal{A}_r(v_{n_k}), \theta_1\rangle \to 	\mathbf{1}_{\{t \leq \tau\}}\,\langle  \mathcal{A}_r(v), \theta_1\rangle  \; \text{as}\: n_k \to \infty,
		\end{equation} 
				for almost every \((\omega, t) \in \Omega \times [0,T]\) and for \(d =2,3\). Additionally for $r \in (1, \infty)$ (in $d = 2$) and $r \in (1,3]$ (in $d = 3$), using the embedding of \(\mathbb{H}^1 \hookrightarrow \mathbb{L}^{2r}\) and  \eqref{4.23}, with \(p= 2r\), we deduce
		\begin{eqnarray}\label{eslim22}
				\sup_{n_k}\, \mathbb{E}\Big[ \mathbf{1}_{\Omega_{n_k}}\int^{\tau}_{0}\|\mathcal{P}_1^{n_k} \mathcal{A}_r(v_{n_k})\|^2_{L^2_{\mathrm{div}}} \,ds\Big] &\leq& C\sup_{n_k}\, \mathbb{E}\Big[\mathbf{1}_{\Omega_{n_k}}\int^{\tau}_{0} \|v_{n_k}\|^{2r}_{L^{2r}_{\mathrm{div}}} \,ds\Big]
				\leq C\sup_{n_k}\, \mathbb{E}\,\mathbf{1}_{\Omega_{n_k}} \sup_{s\in [0, \tau]}\|v_{n_k}(s)\|^{2r}_{L^{2r}_{\mathrm{div}}}\nonumber\\ 
                &\leq& C\sup_{n_k}\, \mathbb{E}\,\mathbf{1}_{\Omega_{n_k}} \sup_{t\in [0, \tau]}\|(v_{n_k}, \phi_{n_k},\sigma_{n_k})(t)\|^{2r}_{\mathcal{V}}\: <\infty.      
		\end{eqnarray}
					From \eqref{eslim21}, \eqref{eslim22}, and \cite[Lemma 5.2]{Glatt-Holtz2009}, we get
		\begin{equation}\label{key1}
					\mathbf{1}_{\{\Omega_{n_k},t \leq \tau\}}\, \mathcal{P}_1^{n_k} 	\mathcal{A}_r(v_{n_k}) \rightharpoonup	\mathbf{1}_{\{t \leq \tau\}}\,  \mathcal{A}_r(v) \; \text{in}\; L^2(\Omega; L^2(0,T; L^2_{\mathrm{div}})).
		\end{equation}
               The case $r=1$ yields the convergence \eqref{key1} straightforwardly.
				Now as in \cite[Proposition 4.2]{Deugoue2021}, we can get 
		\begin{equation}\label{key3}
					\mathbf{1}_{\{\Omega_{n_k},t \leq \tau\}}\,\mathcal{P}_1^{n_k} B_0(v_{n_k}, v_{n_k})\rightharpoonup \mathbf{1}_{\{t \leq \tau\}}\,B_0(v, v)\; \text{in}\; L^2(\Omega; L^2(0,T; L^2_{\mathrm{div}})). 
			\end{equation}
   			\begin{equation}\label{key5}
					\mathbf{1}_{\{\Omega_{n_k},t \leq \tau\}}\,\mathcal{P}_1^{n_k} 	R_0(\epsilon A_1 \phi_{n_k}, \phi_{n_k}) \rightharpoonup \mathbf{1}_{\{t \leq \tau\}}\,R_0(\epsilon A_1 \phi, \phi)\; \text{in}\; L^2(\Omega; L^2(0,T; L^2_{\mathrm{div}})). 
		  \end{equation}
		 		 		For the stochastic term, we use \eqref{equinoise}, Assumption [\ref{[A2]}] and Lemma \ref{norm_esti} to obtain
		\begin{equation}\label{eslim15}
		 		\begin{aligned}
		 		 		\hspace{-.12in}\|\mathcal{P}_2^{n_k} G_1(t, v_{n_k})- G_1(t, v) \|_{\mathcal{L}_2(U_2, L^2_{\mathrm{div}})}^2 & \leq C \Big( \|G_1(t, v_{n_k})-G_1(t, v)\|^2_{\mathcal{L}_2(U_1, L^2_{\mathrm{div}})} + \|\mathcal{Q}_1^{n_k}G_1(t, v)\|^2_{\mathcal{L}_2(U_1, L^2_{\mathrm{div}})}\Big)	\\
		 		 		& \leq C_{L_{L^2_{\mathrm{div}}}}\|v_{n_k}-v\|^2_{L^2_{\mathrm{div}}} + \frac{C_{B_{V}}}{\beta_{n_k}}\left(1+ \|v\|^2 \right).
		 		 	\end{aligned}
		\end{equation}
		 		 		With this, together with \eqref{4.21}, for almost every \((\omega, t) \in \Omega \times [0,T]\), we get 
		 		 			$\mathbf{1}_{\{t \leq \tau\}}\mathcal{P}_1^{n_k} G_1(t, v_{n_k}) \to 	\mathbf{1}_{\{t \leq \tau\}}G_1(t, v) \ \text{in } \  \mathcal{L}_2(U_1,L^2_{\mathrm{div}}).$
		 		 		  On the other hand,
		 		 \begin{equation}\label{eslim152}
		 		 		\begin{aligned}
		 		 				\sup_{n_k}\, \mathbb{E}\Big[ \mathbf{1}_{\Omega_{n_k}}\int^{\tau}_{0}\|\mathcal{P}_1^{n_k}G_1(t, v_{n_k}) \|_{\mathcal{L}_2(U_1,L^2_{\mathrm{div}})}^2 \,ds\Big] 
		 		 				&\leq C_{B_{L^2_{\mathrm{div}}}}\sup_{n_k}\, \mathbb{E}\left[\mathbf{1}_{\Omega_{n_k}}\int^{\tau}_{0}  (1+\|v_{n_k}\|^2_{L^2_{\mathrm{div}}})\,ds\right]<\infty.
		 		 		\end{aligned}
		 		\end{equation}
		 		 		Using the above bound and convergence, we deduce the following:
		 		\begin{equation}\label{key6}
		 		 			\mathbf{1}_{\{\Omega_{n_k},t \leq \tau\}}\,  \mathcal{P}_1^{n_k} G_1(t, v_{n_k})\, \rightharpoonup \mathbf{1}_{\{t \leq \tau\}}\, G_1(t, v)\; \text{in}\; L^2(\Omega; L^2(0,T; \mathcal{L}_2(U_1,L^2_{\mathrm{div}}))).
		 		\end{equation}
		Next, we use the fact that the weak convergence is preserved under bounded linear operators, which ensures the identification of weak limits. More precisely, from \eqref{conver} and \eqref{key1}-\eqref{key6}, we derive for any fixed \(\theta_1 \in L^2_{\mathrm{div}}\)  that
			\begin{align}
				\mathbf{1}_{\Omega_{n_k}} \int_0^{t \wedge \tau} (\mathcal{P}_1^{n_k} \mathcal{A}_r(v_{n_k}), \theta_1)\,ds &\rightharpoonup \int_0^{t \wedge \tau} ( \mathcal{A}_r(v), \theta_1)\,ds,\nonumber\\
				\mathbf{1}_{\Omega_{n_k}}\int_0^{t \wedge \tau} (\mathcal{P}_1^{n_k} B_0(v_{n_k}, v_{n_k}), \theta_1)\,ds &\rightharpoonup \int_0^{t \wedge \tau} (B_0(v, v), \theta_1)\,ds, \label{remark4.1}\\
				\mathbf{1}_{\Omega_{n_k}} \int_0^{t \wedge \tau} (\mathcal{P}_1^{n_k} 	R_0(\epsilon A_1 \phi_{n_k}, \phi_{n_k}), \theta_1)\,ds &\rightharpoonup \int_0^{t \wedge \tau} (R_0(\epsilon A_1 \phi, \phi), \theta_1)\,ds, \nonumber\\
				\mathbf{1}_{\Omega_{n_k}} \sum_{j=1}^\infty \int_0^{t \wedge \tau} (\mathcal{P}_1^{n_k} g_{1,j}(s, v_{n_k}), \theta_1)\,d\beta^1_j(s) 	&\rightharpoonup \sum_{j=1}^\infty \int_0^{t \wedge \tau} (g_{1,j}(s, v) , \theta_1)\,d\beta^1_j(s),\nonumber
			\end{align}
				weakly in \(L^2(\Omega \times [0,T])\). By weak convergence in \eqref{remark4.1}, we mean the convergence when tested against any function in $L^2(\Omega \times [0,T])$. In particular, this allows us to select the specific test function \(\chi_{K}\) for any measurable set \( K \subset \Omega \times [0,T]\). Using this argument, we can now pass to the limit in \eqref{exvest} and thus obtain the desired identity \eqref{finaidene1}. \medskip 
                
\noindent\textbf{Substep 2.} For any fixed \(\theta_2 \in H^1\), using \eqref{conver} and \eqref{gal22}, we obtain 
		\begin{eqnarray}
		  \label{exvest1} 
					&& \hspace{-.2in}\mathbb{E}\, \int_{0}^{T} \chi_{K} (	\phi(t \wedge \tau), \theta_2 ) \, dt = \lim_{n_k \to \infty}
					\mathbb{E}\, \int_{0}^{T} (\mathbf{1}_{\Omega_{n_k}}	\phi_{n_k}(t \wedge \tau),\; \chi_{K} \theta_2 ) \, dt \\
                    && = \lim_{n_k \to \infty}\left( \mathbb{E}\, \int_{0}^{T} \chi_{K}\mathbf{1}_{\Omega_{n_k}} \left[\big (\mathcal{P}_1^{n_k}\phi_0, \theta_2\big)- \int_0^{t \wedge \tau} \big\langle \epsilon A_1^2 \phi_{n_k} + \epsilon^{-1}\mathcal{P}_2^{n_k}A_1 \psi'(\phi_{n_k}) + \mathcal{P}_2^{n_k}B_1(v_{n_k},\phi_{n_k}), \theta_2\big \rangle \, ds\right.\right. \nonumber\\
                    &&\qquad\left.\left.+ \int_0^{t \wedge \tau}\langle \mathcal{P}_2^{n_k}\big(P\sigma_{n_k} - A - \alpha u \big) h(\phi_{n_k}), \theta_2\rangle \,ds \right]dt \right).	\nonumber		\end{eqnarray}
				Utilizing \eqref{eq:ass2}, \eqref{estst2} and H\"older's inequality, we have
		\begin{equation}\label{eslim9}
			\begin{aligned}
				&|\langle \mathcal{P}_2^{n_k}  A_1 \psi'(\phi_{n_k})-  A_1 \psi'(\phi) ,\theta_2 \rangle| \leq |\langle  A_1 \psi'(\phi_{n_k}) - A_1 \psi'(\phi), \mathcal{P}_1^{n_k} \theta_2\rangle|+|\langle A_1 \psi'(\phi), \mathcal{Q}_1^{n_k} \theta_2 \rangle|\\
				& \leq \left(\, \Big|\psi'''(\phi_{n_k})\big( ~|\nabla \phi_{n_k} |^2- |\nabla \phi |^2\big)~\Big|_{L^2}+\Big|\,|\nabla \phi |^2 \big(\psi'''(\phi_{n_k}) - \psi'''(\phi)\big)\Big|_{L^2}+\Big|\psi''(\phi_{n_k})A_1(\phi_{n_k}- \phi) \Big|_{L^2}\right.\\
				& \qquad \ \ \left.+ \Big|A_1 \phi\big(\psi''(\phi_{n_k})-\psi''(\phi)\big)\Big|_{L^2} \right)|\mathcal{P}_1^{n_k}\theta_2|_{L^2}+ \left(\, |\psi'''(\phi) |\nabla \phi|^2|_{L^2}+|\psi''(\phi) A_1\phi|_{L^2} \right)|\mathcal{Q}_1^{n_k}\theta_2|_{L^2}\\
				&\leq C_{\psi} \mathbb{Q}_1(\|\phi_{n_k}\|_{H^2},\|\phi\|_{H^2})\|\phi_{n_k}-\phi\|_{H^2}|\theta_2|_{L^2} +  \frac{ C_{\psi}}{\beta_{n_k}^{1/2}}\mathbb{Q}_2(\|\phi\|_{H^2})\|\theta_2\|_{H^1},
			\end{aligned}
		\end{equation}
			for some monotonically increasing functions \(\mathbb{Q}_1\) and \(\mathbb{Q}_2\). Now using \eqref{4.21} and \eqref{eslim9}, we have for almost  every \((\omega, t) \in \Omega \times [0,T]\) and for any \(\theta_2 \in H^1\),
		\begin{equation}\label{eslim91}
			\mathbf{1}_{\{t \leq \tau\}}\,\langle \mathcal{P}_2^{n_k}A_1 \psi'(\phi_{n_k}) , \theta_2\rangle \to 	\mathbf{1}_{\{t \leq \tau\}}\, \langle  A_1 \psi'(\phi), \theta_2\rangle  \quad \text{as} \: n_k \to \infty.
		\end{equation}
			In addition, after applying \eqref{eq:ass2} and then using the uniform bound \eqref{4.23} for \(p=2\), \(p=4\), and \(p=10\), we obtain
		\begin{equation}\label{eslim92}
			\begin{aligned}
				&\sup_{n_k}\, \mathbb{E}\Big[ \mathbf{1}_{\Omega_{n_k}}\int^{\tau}_{0}|\mathcal{P}_2^{n_k} A_1 \psi'(\phi_{n_k})|^2_{L^2} \,ds\Big]\leq C\sup_{n_k}\, \mathbb{E}\Big[\mathbf{1}_{\Omega_{n_k}}\int^{\tau}_{0} |\psi'''(\phi_{n_k}) |\nabla \phi_{n_k}|^2 +\psi''(\phi_{n_k}) A_1\phi_{n_k}|^2_{L^2} \,ds\Big]\\
				&\leq C_{\psi}\sup_{n_k}\, \mathbb{E}\Big[\mathbf{1}_{\Omega_{n_k}}\int^{\tau}_{0} \Big(\big(1+ \|\phi_{n_k}\|^6_{L^{\infty}}\big) \|\nabla \phi_{n_k}\|^4_{L^4} + \big(1+ \|\phi_{n_k}\|^8_{L^{\infty}}\big)|A_1\phi_{n_k}|^2_{L^2}\Big) \,ds\Big]\\
				&\leq C_{\psi}\sup_{n_k}\, \mathbb{E}\left[\mathbf{1}_{\Omega_{n_k}}\int^{\tau}_{0} \Big(\|\phi_{n_k}\|^2_{H^2}+ \|\phi_{n_k}\|^4_{H^2}+\|\phi_{n_k}\|^{10}_{H^2}\Big) \,ds\right]< \infty.
			\end{aligned}
		\end{equation}
					From \eqref{eslim91} and \eqref{eslim92}, we get
					$\mathbf{1}_{\{\Omega_{n_k},t \leq \tau\}}\, \mathcal{P}_2^{n_k} 	A_1 \psi'(\phi_{n_k}) \rightharpoonup	\mathbf{1}_{\{t \leq \tau\}}\,  A_1 \psi'(\phi) \; \text{in}\; L^2(\Omega; L^2(0,T; L^2)).$
					Now for nonlinear term $B_1(\cdot, \cdot),$ we follow \cite[Proposition 4.2]{Deugoue2021} and obtain: 
		\begin{equation}\label{key9}
					\mathbf{1}_{\{\Omega_{n_k},t \leq \tau\}}\,\mathcal{P}_2^{n_k} B_1(v_{n_k}, \phi_{n_k})\rightharpoonup \mathbf{1}_{\{t \leq \tau\}}\,B_1(v, \phi)\; \text{in}\; L^2(\Omega; L^2(0,T; L^2)). 
		\end{equation}
					Next, by making use of Assumption~[\ref{[A3]}], the boundedness of $u$ (see [\ref{[A4]}]), H\"older's inequality, and Lemma \ref{norm_esti}, we obtain, for any $\theta_2 \in H^1$, that
	\begin{eqnarray}\label{eslim10}
					\lefteqn{|\langle(\mathcal{P}_2^{n_k}(P \sigma_{n_k} - A -\alpha u)h(\phi_{n_k})- (P\sigma -A- \alpha u)h(\phi), \theta_2\rangle|}\nonumber\\
					&&\leq |\langle P(\sigma_{n_k}- \sigma) ,\mathcal{P}_2^{n_k}\theta_2\rangle| + |\langle(P\sigma -A- \alpha u)(h(\phi_{n_k})-h(\phi)) ,\mathcal{P}_2^{n_k}\theta_2\rangle|+ |\langle (P\sigma -A- \alpha u)h(\phi),\mathcal{Q}_2^{n_k}\theta_2\rangle|\nonumber\\
					&&\leq P|\sigma_{n_k}- \sigma|_{L^2}|\theta_2|_{L^2}+ L_{h}|P\sigma -A- \alpha u|_{L^2}\|\phi_{n_k}- \phi\|_{H^1}\|\theta_2\|_{H^1}+ \frac{1}{\beta_{n_k}^{1/2}}|P\sigma -A- \alpha u|_{L^2} \|\theta_2\|_{H^1}.
		\end{eqnarray}
					Using \eqref{4.21} and \eqref{eslim10}, we obtain for almost every \((\omega, t) \in \Omega \times [0,T]\) and  \(\theta_2 \in H^1\) that
		\begin{equation}\label{eslim101}
					\mathbf{1}_{\{t \leq \tau\}}\,\langle \mathcal{P}_2^{n_k}(P \sigma_{n_k} - 	A -\alpha u)h(\phi_{n_k}) , \theta_2\rangle \to 	\mathbf{1}_{\{t \leq \tau\}}\, \langle  (P\sigma -A- \alpha u)h(\phi), \theta_2\rangle  ~\text{as}~ n_k \to \infty.
		\end{equation}
					Moreover, from  \eqref{4.23} for \(p=2\), together with the assumptions on \(u\) and \(h\), we further get
		\begin{equation}\label{eslim102}
			\begin{aligned}
					\sup_{n_k}\, \mathbb{E}\Big[ \mathbf{1}_{\Omega_{n_k}}\int^{\tau}_{0}|\mathcal{P}_2^{n_k}(P \sigma_{n_k} - A -\alpha u)h(\phi_{n_k}) |^2_{L^2} \,ds\Big] 
                    < \infty.
			\end{aligned}
		\end{equation}	
					From \eqref{eslim101} and \eqref{eslim102}, we deduce again that
					$\mathbf{1}_{\{\Omega_{n_k},t \leq \tau\}}\, \mathcal{P}_2^{n_k}(P \sigma_{n_k} - A -\alpha u)h(\phi_{n_k}) \rightharpoonup	\mathbf{1}_{\{t \leq \tau\}}\,  (P\sigma -A- \alpha u)h(\phi) \; \text{in}\; L^2(\Omega; L^2(0,T; L^2)).$
					Thus, from the preceding convergences,  for any fixed \(\theta_2 \in L^2\),  one can obtain the weak convergences, as in \eqref{remark4.1}, for the terms on the right-hand side of \eqref{exvest1} in \(L^2(\Omega \times [0,T])\).
				 Hence, passing to the limit in \eqref{exvest1} yields \eqref{finaidene2}. \medskip
                 
\noindent\textbf{Substep 3.} For any fixed \(\theta_3 \in H^1\), in view of  \eqref{conver} and \eqref{gal23}, we infer that 
	   \begin{eqnarray}\label{exvest2} 
	   			\lefteqn{\mathbb{E}\, \int_{0}^{T} \chi_{K} (	\sigma(t \wedge \tau), \theta_3 ) \, dt = \lim_{n_k \to \infty}
	   			\mathbb{E}\, \int_{0}^{T}  (\mathbf{1}_{\Omega_{n_k}}	\sigma_{n_k}(t 	\wedge \tau),\; \chi_{K} \theta_3 ) \, dt \nonumber}
                \\&&= \lim_{n_k \to \infty}\left( \mathbb{E}\, \int_{0}^{T} \chi_{K}\mathbf{1}_{\Omega_{n_k}} \left[\big (\mathcal{P}_1^{n_k}\sigma_0, \theta_3\big)- \int_0^{t \wedge \tau} \big\langle  A_1 \sigma_{n_k} +\mathcal{P}_2^{n_k} B_1(v_{n_k},\sigma_{n_k})+ c\mathcal{P}_2^{n_k} \sigma_{n_k} h(\phi_{n_k}), \theta_3\big \rangle \, ds \right.\right. \nonumber\\
                &&\qquad \left.\left.+ \int_0^{t \wedge \tau}\langle b\mathcal{P}_2^{n_k}(\sigma_{n_k} - w),\theta_3 \big\rangle ds+ \sum_{k=1}^{\infty} \big\langle \mathcal{P}_2^{n_k}G_2(s,\sigma_{n_k}) e_k^2, \theta_3 \big\rangle d\beta_k^2(s) \, \right]dt \right).
	   \end{eqnarray}
			 	  For the nonlinear term \(B_1(\cdot, \cdot)\), using H\"older's and Gagliardo-Nirenberg's inequalities (see \eqref{b1_5}) along with Lemma \ref{norm_esti}, one can get 
	\begin{equation}\label{eslim12}
			\begin{aligned}
					&|\langle\mathcal{P}_2^{n_k}B_1(v_{n_k},\sigma_{n_k})-B_1(v, \sigma),\theta_3\rangle|\\&\leq
					 |\langle B_1(v_{n_k} -v,\sigma_{n_k}),\mathcal{P}_2^{n_k}\theta_3\rangle|+ |\langle B_1(v,\sigma_{n_k}- \sigma),\mathcal{P}_2^{n_k}\theta_3\rangle|+|\langle B_1(v,\sigma), \mathcal{Q}_2^{n_k}\theta_3\rangle|\\
					&\leq \left(~\|v_{n_k} -v\|\,|A_1^{1/2}\sigma_{n_k}|_{L^2}+\|v\|\,|A_1^{1/2}(\sigma_{n_k}- \sigma)|_{L^2} \right)\|\theta_3\|_{H^1}+\|v\|\,|A_1^{1/2}\sigma|_{L^2}|\mathcal{Q}_2^{n_k}\theta_3|^{1/2}_{L^2}\|\theta_3\|_{H^1}^{1/2}\\
					&\leq \left(~\|v_{n_k} -v\|\,\|\sigma_{n_k}\|_{H^1}+\|v\|\,\|\sigma_{n_k}- \sigma\|_{H^1} \right)\|\theta_3\|_{H^1} + \frac{1}{\beta_{n_k}^{1/4}}\|v\|\,\|\sigma\|_{H^1}\|\theta_3\|_{H^1}.
			\end{aligned}
		\end{equation}
					From \eqref{4.21} and \eqref{eslim12}, we get for all \(\theta_3 \in H^1\) and for almost every \((\omega, t) \in \Omega \times [0,T]\) that
		\begin{equation}\label{eslim121}
					\mathbf{1}_{\{t \leq \tau\}}\,\langle 	\mathcal{P}_2^{n_k}B_1(v_{n_k},\sigma_{n_k}) , \theta_3\rangle \to 	\mathbf{1}_{\{t \leq \tau\}}\, \langle B_1(v, \sigma), \theta_3\rangle  \quad \text{as}\ \  n_k  \to \infty.
		\end{equation}
					Moreover, after utilizing \eqref{4.23} (for \(p=4\)), together with H\"older's inequality, we estimate
			\begin{equation}\label{eslim122}
			\begin{aligned}
			   		&\sup_{n_k}\, \mathbb{E}\Big[ 	\mathbf{1}_{\Omega_{n_k}}\int^{\tau}_{0}|\mathcal{P}_1^{n_k}B_1(v_{n_k},\sigma_{n_k}) |^2_{L^2} \,ds\Big] 
					\leq C\sup_{n_k}\, \mathbb{E}\Big[\mathbf{1}_{\Omega_{n_k}}\int^{\tau}_{0}\|v_{n_k}\|^2_{L^{\infty}_{\mathrm{div}}} |A_1^{1/2} \sigma_{n_k}|^2_{L^2}\,ds\Big]\\
					&\qquad\leq C\sup_{n_k}\, \mathbb{E}\Big[\mathbf{1}_{\Omega_{n_k}} \sup_{s \in [0, \tau]}|A_1^{1/2} \sigma_{n_k}(s)|^2_{L^2}\int^{\tau}_{0}\|A_0v_{n_k}\|^2_{L^2_{\mathrm{div}}} \,ds\Big]\\
					&\qquad\leq C\sup_{n_k}\, \mathbb{E}\,\mathbf{1}_{\Omega_{n_k}}\left[ \sup_{s \in [0, \tau]}\| \sigma_{n_k}(s)\|^4_{H^1} + \left(\int^{\tau}_{0}\|A_0v_{n_k}\|^2_{L^2_{\mathrm{div}}} \,ds\right)^2 \,\right] < \infty.
				\end{aligned}
		\end{equation}
				Using \eqref{eslim121},\eqref{eslim122} and  \cite[Lemma 5.2]{Glatt-Holtz2009}, we derive
			\begin{equation}\label{key11}
						\mathbf{1}_{\{\Omega_{n_k},t \leq \tau\}}\,\mathcal{P}_2^{n_k} B_1(v_{n_k}, \sigma_{n_k})\rightharpoonup \mathbf{1}_{\{t \leq \tau\}}\,B_1(v, \sigma)\; \text{in}\; L^2(\Omega; L^2(0,T; L^2)). 
			\end{equation}
                    By the Assumption [\ref{[A3]}] on $h$, we get $\mathbf{1}_{\{\Omega_{n_k},t \leq \tau\}}\,\mathcal{P}_2^{n_k}\sigma_{n_k}h(\phi_{n_k}) \rightharpoonup \mathbf{1}_{\{t \leq \tau\}}\, \sigma h(\phi)\; \text{in}\; L^2(\Omega; L^2(0,T; L^2))$. 
			             For the stochastic term, we follow the same argument as in \eqref{key6} to obtain:
				\begin{equation}\label{key13}
					\mathbf{1}_{\{\Omega_{n_k},t \leq \tau\}}\,  \mathcal{P}_2^{n_k} G_2(t, \sigma_{n_k})\, \rightharpoonup \mathbf{1}_{\{t \leq \tau\}}\, G_2(t, \sigma)\; \text{in}\; L^2(\Omega; L^2(0,T; \mathcal{L}_2(U_2,L^2))).
				\end{equation}
              Consequently, by invoking the above convergences and following the same line of arguments as in Substep 1,  we can pass the limit in \eqref{exvest2} to obtain: \eqref{finaidene3}.
        
             Hence, the verification of equations \eqref{finaidene1},\eqref{finaidene2} and \eqref{finaidene3} in each of the three substeps confirms that \(\{(v, \phi, \sigma), \tau\}\) collectively satisfy the full coupled system \eqref{regu2} in \( L^2_{\mathrm{div}} \times L^2\times L^2 \).\medskip\\
               \textbf{Step 2.} Next, we consider a more general case of \( \mathbb{E}\,\|(v_{0}, \phi_{0}, \sigma_{0})\|_{\mathcal{V}}^2 < \infty\), i.e., \((v_0, \phi_0, \sigma_0) \in L^2(\Omega; \mathcal{V})\). 
              For some fixed \( i \in \mathbb{N} \cup \{0\} \), let \(\{(v_i, \phi_i, \sigma_i); \tau_i\}\) be a local strong solution corresponding to the initial data \((v_0, \phi_0, \sigma_0)\mathbf{1}_{i \leq \|(v_0, \phi_0, \sigma_0)\|_{\mathcal{V}} < i+1}\). We further define \(\{(v, \phi, \sigma); \tau\}\) for each \(t \in [0, T] \) and \(\omega \in \Omega\) as:
    \begin{equation*}
    	\begin{aligned}
   				  &v(t,\omega) = \sum_{i=0}^{\infty}v_i(t,\omega)\mathbf{1}_{i \leq  \|v_0(\omega)\| < i+1}, \,\phi(t,\omega) = \sum_{i=0}^{\infty}\phi_i(t,\omega)\mathbf{1}_{i \leq  \|\phi_0(\omega)\|_{H^2} < i+1}, \\ 
   				   &\sigma(t,\omega) = \sum_{i=0}^{\infty}\sigma_i(t,\omega)\mathbf{1}_{i \leq  \|\sigma_0(\omega)\|_{H^1} < i+1},	\,	\tau(t,\omega) = \sum_{i=0}^{\infty}\tau_i(\omega)\mathbf{1}_{i \leq \|(v_0, \phi_0, \sigma_0)(\omega)\|_{\mathcal{V}} < i+1}.
       \end{aligned}
       	 \end{equation*}
  				 We proceed to show that \(\{(v, \phi, \sigma); \tau\}\) is a local strong solution corresponding to the initial data \((v_0, \phi_0, \sigma_0)\). Since for each fixed \(\omega\), the infinite
  				 sum is simply a single element, we can write: \( \{(v(\omega), \phi(\omega), \sigma(\omega)); \tau(\omega)\}  = \{(v_i(\omega), \phi_i(\omega), \sigma_i(\omega)); \tau_i(\omega)\}  \) for some \(i\). As a consequence of \((v_i(\cdot \wedge \tau_i), \phi_i(\cdot \wedge \tau_i),\sigma_i(\cdot \wedge \tau_i)) \in C([0, T]; \mathcal{V})\) and \((v_i \mathbf{1}_{\cdot\leq \tau_i}, \phi_i\mathbf{1}_{\cdot\leq \tau_i}, \sigma_i\mathbf{1}_{\cdot\leq \tau_i}) \in L^2(0,T; \mathcal{Z})\) a.s., we have \((v(\cdot \wedge \tau), \phi(\cdot \wedge \tau),\sigma(\cdot \wedge \tau)) \in C([0, T]; \mathcal{V})\) and \((v \mathbf{1}_{\cdot\leq \tau}, \phi\mathbf{1}_{\cdot\leq \tau}, \sigma\mathbf{1}_{\cdot\leq \tau}) \in L^2(0,T; \mathcal{Z})\) a.s. for all \(T>0\). Arguing analogously to \cite[Proposition 4.2]{Glatt-Holtz2009}, we conclude that, for \(\{(v, \phi, \sigma); \tau\}\) defined as above, $\mathbb{E}\,\sup_{r \in [0, \tau]} \| (v, \phi, \sigma)(r) \|_{\mathcal{V}}^2 
  				 + \mathbb{E}\int_{0}^{\tau} (\, \| A_0 v \|^2_{L^2_{\mathrm{div}}} +  \|\phi\|^2_{H^4}   
  				 + \|\sigma\|^2_{H^2})ds < \infty $ and identities in \eqref{regu2} hold true. Hence, from Definition \ref{def4.2}, the conclusion follows.
   	 \end{proof}
  				
  \subsection{Existence of maximal solution}  In the last section, we have shown the existence of a local strong solution \(\{(v, \phi, \sigma); \tau\}\) of the system \eqref{a1}-\eqref{a5} in both \(d = 2,3\) cases, and this local strong solution \(\{(v, \phi, \sigma); \tau\}\) is unique in a larger class of weak solutions. Next, we extend the pair \(\{(v, \phi, \sigma); \tau\}\) corresponding to the initial data $(v_0, \phi_0, \sigma_0) \in L^2(\Omega; \mathcal{V})$ to a solution defined on a time interval \([0, \xi)\), where \(\xi\) is a maximal time of existence. The argument follows in a manner analogous to that of  \cite{Jacod1979,Mikulevicius2004,Glatt-Holtz2009, Breit2018}.
  \begin{Thm}\label{maxsol}
        Under the hypotheses of Proposition \ref{exismain_bdd}, there exists a unique maximal solution \(\{(v, \phi, \sigma),\{t_R\}_{R \in \mathbb{N}}, \xi\}\). Moreover, the pair  \(\{(v, \phi, \sigma); \xi\}\) is a weak solution. 
  \end{Thm}
  \begin{proof}
		 Let \((v_0, \phi_0, \sigma_0) \in L^2(\Omega; \mathcal{V})\) be fixed. Take \(\mathcal{K}\) as the collection of stopping times \(\tau\) such that there exists a local strong solution \(\{(v, \phi, \sigma); \tau\}\) of the system \eqref{a1}-\eqref{a5} with respect to the initial data \((v_0, \phi_0, \sigma_0)\). It is closed with respect to the finite minimum and finite maximum operations. Let \( \xi \) be the essential upper bound of the set \(\mathcal{K}\) (see \cite[Chapter 5, Section 18]{Doob1994}). So, there
		 is an increasing sequence \(\tau_k\), which converges to \(\xi\). Suppose \(\{(v_k, \phi_k, \sigma_k); \tau_k\} \) be a local strong solution corresponding to \(\tau_k\). Next by taking a sequence of truncated processes \((v_k(t \wedge \tau_k) \mathbf{1}_{t< \xi}\,, \phi_k(t \wedge \tau_k)\mathbf{1}_{t< \xi}\,, \sigma_k(t \wedge \tau_k)\mathbf{1}_{t<\xi})\), one can notice from the local uniqueness of solutions that, for every \(t > 0\), this sequence converges in \(\mathcal{V}\) a.s., i.e.,
 \begin{equation}\label{pro1}
  			(v(t), \phi(t), \sigma(t)) = \lim_{k \to \infty} \big((v_k(t \wedge \tau_k) 	\mathbf{1}_{t< \xi}\,, \phi_k(t \wedge \tau_k)\mathbf{1}_{t< \xi}\,, \sigma_k(t \wedge \tau_k)\mathbf{1}_{t<\xi})\big) \quad \text{a.s. in}\; V.
  \end{equation}
					Moreover, by Lemma \ref{Ener_est} and \eqref{pro1}, it follows that \(\{(v, \phi, \sigma); \xi\}\) is a local weak solution (for more details on a similar argument, see \cite{Glatt-Holtz2009}). For $R \in \mathbb{N}$, we now define 
	\begin{align}\label{rho_R}
				\rho_R := \inf_{t \geq 0}\{ \sup_{s \in [0, t]} \| (v, \phi, \sigma)(s) \|_{\mathcal{V}}^2 
				+ \int_{0}^{t} ( \| A_0 v \|^2_{L^2_{\mathrm{div}}} +  \|\phi\|^2_{H^4}   
				+ \|\sigma\|^2_{H^2})ds  >  R\} \wedge \xi, 
		\end{align}
					as a stopping time. For each \(R \in \mathbb{N}\), \(\{(v, \phi, \sigma), \rho_R\}\) is a local strong solution. Note that the stopping time $\rho_R$ is not a.s. strictly positive unless $\|(v_0, \phi_0, \sigma_0)\|_{\mathcal{V}}^2 \leq R$. 
					Since $(v_0, \phi_0, \sigma_0)\in L^2(\Omega; \mathcal{V})$, it follows that for almost every $\omega$, there exists $R = R(\omega)$ sufficiently large such that $\|(v_0, \phi_0, \sigma_0)(\omega)\|_{\mathcal{V}}^2 \leq R(\omega)$ and hence $\rho_{R(\omega)}(\omega) > 0,$ while this $R$ depends on $\omega$. To ensure uniform strict positivity, define $t_R = \rho_R \vee \tau_R.$ We see that \(\{t_R \}_{R\in \mathbb{N}}\) is monotonically increasing to \(\xi\) and for each $R \in \mathbb{N}$,  $\{(v, \phi, \sigma), t_R\}$, forms a local strong solution with an a.s. strictly positive stopping time $t_R$. Furthermore, through an iterative extension of the local solution, any solution on $[0, t_R]$ can be extended to a solution on $[0, t_R + \bar{\tau}]$ for some a.s. strictly positive stopping time $\bar{\tau}$. Indeed by the method of \cite[Lemma 4.1]{Glatt-Holtz2009}, we can restart the solution from  $(v(t_R), \phi(t_R), \sigma(t_R))$ and combine the two parts, where uniqueness guarantees their consistency. To prove that $t_R < \xi$ for every $R\in \mathbb{N}$ on $\{\xi < \infty\}$, assume, towards a contradiction, that 
					$\mathbb{P} ( t_R = \xi < \infty ) >0$. 
					In this case, we have $t_R + \bar{\tau} \in \mathcal{K}$, for some a.s. strictly positive stopping time $\bar{\tau}$, implying that 
					$\mathbb{P}(\xi < t_R + \bar{\tau}) > 0$. 
					This however, contradicts the maximality of $\xi$.
					Hence on the set \(\{ \xi < \infty\}\), it follows that $\rho_R < \xi$; using \eqref{rho_R} we then deduce
		\begin{equation*}
				\sup_{r \in [0, \xi)} \| (v, \phi, \sigma)(r) \|_{\mathcal{V}}^2 
				+ \int_{0}^{\xi} ( \| A_0 v \|^2_{L^2_{\mathrm{div}}} +  \|\phi\|^2_{H^4}   
				+ \|\sigma\|^2_{H^2})ds  >  R,
		\end{equation*}
				 for arbitrary \(R \in \mathbb{N}\). Since all criteria (see Definition \ref{def4.3})
                 for a maximal solution are satisfied by $\left\{\, (v, \phi,\sigma), \{t_R\}_{R \in \mathbb{N}}, \xi \,\right\}$ therefore it qualifies as the desired solution. 
		 \end{proof}
	\section{Global existence in dimension two}\label{Section 5}
	\begin{Thm}(Global existence of strong solution in 2D)\label{global2d} 
			Suppose \(d=2\). In addition to the hypotheses of Proposition  \ref{exismain_bdd}, we further assume that
		   $ r \in [1,3] $,	$(v_0, \phi_0, \sigma_0) \in L^p(\Omega; \mathcal{H})$ and 
			$(z,u,w) \in$ $L^p(\Omega; L^2_{\text{loc}}(\,[0, \infty); V' \times L^2 \times (H^1)^{'}\,))$	for some \(p \geq 4\).
			Then the maximal solution $\left\{\, (v, \phi,\sigma), \{t_n\}_{n \in \mathbb{N}}, \varrho \,\right\}$
			 is global in the sense that \( \varrho = \infty\) a.s..
\end{Thm}
\begin{proof}
			Let \(\{(v, \phi, \sigma),\{t_n\}_{n \in \mathbb{N}}, \, \varrho \}\) be a maximal solution and fix $T< \infty$. Since \(t_n\) is increasing to $\varrho $, the events $A_N := \bigcap_{n=1}^{N} \{ t_n \le T \} = \{t_N \leq T\}$ monotonically decrease as $N$ increases i.e., $A_1 \supseteq A_2 \supseteq \cdots, \;\bigcap_{n=1}^{\infty} \{t_n \le T\} = \lim_{N \to \infty} A_N$. Therefore we have 
				$\mathbb{P} \left( \bigcap_{n=1}^{\infty} \{t_n \leq T\} \right) 
				= \lim_{N \to \infty} \mathbb{P} \left( A_N \right) 
				= \lim_{N \to \infty} \mathbb{P}(\{t_N \leq T\})$,
			and
			$\{\varrho < \infty\} = \bigcup_{T=1}^{\infty} \{\varrho\leq T\} = \bigcup_{T=1}^{\infty} \bigcap_{n=1}^{\infty} \{t_n \leq T\}$.
            Hence, to demonstrate the required result, it is sufficient to prove that
	$	\lim_{N \to \infty} \mathbb{P}(t_N \leq T) = 0.$\\
			For some \(M>0\), we define the following stopping time:
\begin{equation} \label{glob4}
	\begin{aligned}
		\zeta_M := \inf_{t \geq 0} \{\int_{0}^{t \wedge \varrho}\big(\,\|v\|^2_{L^2_{\mathrm{div}}}\|v\|^2 + \|v\|^2\|\phi\|_{H^1}^6 +\|\phi\|_{H^1}^6 + \|\phi\|_{H^1}^{10}+ \|\phi\|_{H^1}^{18}+ \|\phi\|_{H^1}^{34} \big)\,ds >M \} \wedge 2T.
		\end{aligned}
\end{equation}
		For any $N \in \mathbb{N}$, from the definition of $\rho_N$ (cf. \eqref{rho_R}), one can infer the following: 
\begin{eqnarray}\label{glob5}
		\lefteqn{\mathbb{P}(t_N \leq T)
		\leq \mathbb{P}(\rho_N \leq T) \leq \mathbb{P} \left( \left\{ \sup_{s \in [0,\ \rho_N \wedge T]} \|(v, \phi, \sigma)(s)\|_{\mathcal{V}}^2\right.\right.}\\
		&& \left.\left.+ \int_0^{\rho_N \wedge T} \!\!\!\left( \|A_0 v\|^2_{L^2_{\mathrm{div}}} + \| 	\phi\|^2_{H^4}+ \|\sigma\|_{H^2}^2\right) \, ds \geq N \right\} 
		\cap \{\zeta_M > T\} \right)+\mathbb{P}(\zeta_M \leq T)  \nonumber\\
		&\leq& \mathbb{P} \left( \left\{ \sup_{s \in [0,\  \rho_N \wedge \zeta_M]} \|(v, \phi, \sigma)(s)\|_{\mathcal{V}}^2 
		+ \int_0^{\rho_N \wedge \zeta_M} \left( \|A_0 v\|^2_{L^2_{\mathrm{div}}} + \| 	\phi\|^2_{H^4}+ \|\sigma\|_{H^2}^2 \right) \, ds \geq N \right\} \right) + \mathbb{P}(\zeta_M \leq T). \nonumber
\end{eqnarray}
        \textbf{Step 1.} Using It\^o formula for the processes \(\|A_0^{1/2}v(\cdot)\|^2_{L^2_{\mathrm{div}}}\) via $\eqref{a1}$ for fixed \( N, T, M \) and a pair of stopping times \( \tau_a \leq \tau_b \leq \rho_N 	\wedge \zeta_M \) and integrating over the interval [\( \tau_a, \tau_b \)], we arrive at
\begin{align}\label{limiv0}
				&\mathbb{E} \sup_{t \in [\tau_a, \tau_b]} \| v(t) \|^2 
				+ 2 \mathbb{E} \int_{\tau_a}^{\tau_b}  \nu \| A_0 v \|^2_{L^2_{\mathrm{div}}}ds \nonumber\\ &	\leq \mathbb{E}\, \|v(\tau_a)\|^2 + 2 \mathbb{E} \int_{\tau_a}^{\tau_b} \left(~ |(\eta\mathcal{A}_r(v), A_0 v)|+ |(B_0(v, v), A_0 v)| + | (R_0(\epsilon A_1 \phi, \phi), A_0 v) |~
				\right.\nonumber\\
				& \quad\left.+ |(z, A_0 v)| + \frac{1}{2} \| G_1(s, v) \|_{\mathcal{L}_2(U_1, V)}^2 \,\right) \, ds + 2 \mathbb{E} \sup_{t \in [\tau_a, \tau_b]} 
				\left| \sum_{j=1}^\infty \int_{\tau_a}^t (g_{1,j}(s,v), A_0 v) \, d\beta^1_j(s) \right|.
		\end{align}
				We shall estimate each term on the right-hand side of \eqref{limiv0} for \(d=2\) only. Thanks to H\"older's and Gagliardo-Nirenberg's (see Lemma \ref{l1}) inequalities, we obtain the following estimate:  
		\begin{equation}\label{glob8}
			\begin{aligned}
				|\eta(\mathcal{A}_r(v), A_0 v)| 
				\leq C_{\eta,\nu}\|v\|_{L^{2r}_{\mathrm{div}}}^{2r} + \frac{\nu}{10}\|A_0 v\|^2_{L^2_{\mathrm{div}}}
				\leq C_{\eta,\nu}\|v\|_{L^2_{\mathrm{div}}}^2\|v\|^{2(r-1)}+\frac{\nu}{10}\|A_0 v\|^2_{L^2_{\mathrm{div}}},
			\end{aligned}
		\end{equation}
			which can be further simplified for the various values of \(r\) as follows:
		\begin{equation*}\label{glob81}
			\begin{aligned}
				\|v\|_{L^2_{\mathrm{div}}}^2\|v\|^{2(r-1)} \, \leq	\begin{cases}
					\bar{C}\, (\|v\|_{L^2_{\mathrm{div}}}^2 + \|v\|_{L^2_{\mathrm{div}}}^2\|v\|^ 4) &: \ \ r \in [1,3)\\
					\bar{C} \, \|v\|_{L^2_{\mathrm{div}}}^2\|v\|^4 &:\ \ r= 3.
				\end{cases}
			\end{aligned}
		\end{equation*}
			Consequently,  \(|\eta(\mathcal{A}_r(v), A_0 v)| \leq C_{\eta,\nu ,\bar{C}}\,(\|v\|_{L^2_{\mathrm{div}}}^2 + \|v\|_{L^2_{\mathrm{div}}}^2\|v\|^ 4) + \frac{\nu}{10}\|A_0 v\|^2_{L^2_{\mathrm{div}}}\).
   			Invoking \eqref{b_03} and Young's inequality, one can get
	\begin{equation}
   		    |( B_0(v, v), A_0 v)| \leq  C_\nu  \|v\|^2_{L^2_{\mathrm{div}}}\| v \|^4 + \frac{\nu}{10} \|A_0 v\|^2_{L^2_{\mathrm{div}}}.\label{glob17}
   	\end{equation}
  	 With the help of Lemma \ref{turiq}, Remark \ref{emer} and Young's inequality, we obtain
   \begin{equation}
   		\begin{aligned}
   		|( R_0(\epsilon A_1 \phi, \phi), A_0 v)| &\leq \epsilon\|A_1 \phi\|_{L^4}\| \nabla\phi\|_{L^4}\|A_0 v\|_{L^2_{\mathrm{div}}}\\
   		&\leq C_{\epsilon} | A_1^{1/2}\phi|_{L^2}^{1/2}|A_1 \phi|_{L^2} |A_1^{3/2} \phi|_{L^2}^{1/2}\|A_0 v\|_{L^2_{\mathrm{div}}}\\
   		&\leq C_{\epsilon} | A_1^{1/2}\phi|_{L^2} |A_1^{3/2} \phi|_{L^2}\|A_0 v\|_{L^2_{\mathrm{div}}}
   		\leq c_{\nu, \epsilon} \|\phi \|^{10}_{H^1} +\frac{\nu}{10}  \|A_0 v\|^2_{L^2_{\mathrm{div}}}+ \frac{\epsilon}{8} \|\phi \|^2_{H^4}. 
   		\end{aligned}
   	\end{equation}\label{glob18} 
   			The linear growth condition on $G_1$ (see [\ref{[A2]}]) and the  Burkholder-Davis-Gundy inequality give
	\begin{align}
   			\| G_1(t,v)\|_{\mathcal{L}_2(U_1,V)}^2 &\leq  C_{B_{V}}(1+\|v\|^2 ),\label{glob20}\\
   			\mathbb{E} \left( \sup_{t \in [\tau_a, \tau_b]} \left| 2 \sum_{j=1}^{\infty} \int_{\tau_a}^{t} \left( g_{1,j}(t,v), A_0v \right) \, d\beta^1_j \right| \right) &\leq \mathbb{E} \left( \frac{1}{2} \sup_{t \in [\tau_a, \tau_b]} \|v(t)\|^2 + C_{B_{V},\nu} \int_{\tau_a}^{\tau_b} (1 + \|v\|^2) \, dt \right)\label{glob21}.
   	\end{align}		
  			\textbf{Step 2.} Taking \(L^2\) inner product of  $\eqref{a2}$ with \(2( \phi+A_1^2 \phi )\), we obtain
\begin{equation}\label{limiphi}
	\begin{aligned}
			&\mathbb{E} \sup_{t \in [\tau_a, \tau_b]} \|\phi(t) \|_{H^1}^2 
			+ 2\epsilon\, \mathbb{E} \int_{\tau_a}^{\tau_b} \left( |A_1\phi|^2_{ L^2} + |A^2_1\phi|^2_{ L^2}\right)ds  \\
			&\leq \mathbb{E} \| \phi(\tau_a) \|_{H^1}^2  + 2 \mathbb{E} \int_{\tau_a}^{\tau_b} \left( 
			\epsilon^{-1} | (A_1 \psi'(\phi),\, \phi + A_1^2 \phi) | 
			+ | (B_1(v, \phi),\, \phi + A_1^2 \phi) | 
			\right) \, ds \\
			&\quad + 2 \mathbb{E} \int_{\tau_a}^{\tau_b} \left|\left((P\sigma- A- \alpha u)h(\phi),\, \phi + A_1^2 \phi\right)\right| \, ds.
	\end{aligned}
\end{equation}
			From Assumption [\ref{[A1]}] (see \eqref{eq:ass2}), we deduce that
		\begin{align}
				&|\epsilon^{-1}( A_1 \psi'(\phi),\, \phi+A_1^2 \phi)| \,=
				|\epsilon^{-1}( -\psi'''(\phi) |\nabla \phi|^{2} +  \psi''(\phi)A_1(\phi) ,\, \phi + A_1^2 \phi)|\nonumber\\
				&\qquad\leq C_{\epsilon,c_{\psi}}\int_{\mathcal{O}}(1+|\phi|^3)|\nabla\phi|^2\,| \phi + A_1^2 \phi| \,dx +C_{\epsilon,C_{\psi}}\int_{\mathcal{O}}(1+|\phi|^4) |A_1\phi|\,| \phi + A_1^2 \phi|\,dx := I_1 + I_2. \label{glob10}
			\end{align}
		With the Gagliardo–Nirenberg inequality (see Lemma \ref{l1}), Lemma \ref{turiq}, Remark \ref{emer}, and $H^1\hookrightarrow L^p, p\geq 2$, we obtain
      \begin{equation}
			\begin{aligned}
				I_1 & \leq C_{\psi, \epsilon}\int_{\mathcal{O}}\left(\: |\nabla\phi|^2 |\phi| + |\nabla\phi|^2|A_1^2\phi|+ |\phi|^4|\nabla\phi|^2+ |\phi|^3|\nabla\phi|^2|A_1^2\phi|\:\right) \,dx\\
				& \leq C_{\psi, \epsilon}\left(\: \|\nabla \phi\|_{L^4}|\nabla \phi|_{L^2}\|\phi\|_{L^4} + \|\nabla \phi\|_{L^4}^2|A_1^2\phi|_{L^2}+\|\phi\|^4_{L^{16}}\|\nabla \phi\|_{L^4}|\nabla \phi|_{L^2} + \|\phi\|^3_{L^{12}}\|\nabla \phi\|_{L^8}^2|A_1^2 \phi|_{L^2} \:\right)\\
				& \leq C_{\psi, \epsilon}\left(\:|\nabla \phi|_{L^2}^{1/2}\|\nabla \phi\|_{H^1}^{1/2}|A_1\phi|_{L^2}\|\phi\|_{H^1}+ |\nabla \phi|_{L^2}\|\nabla \phi\|_{H^1}|A_1^2\phi|_{L^2}+ \|\phi\|^4_{H^1}|\nabla \phi|_{L^2}^{1/2}\|\nabla \phi\|_{H^1}^{1/2}|A_1 \phi|_{L^2}\right.\\
				& \qquad  \left.+ \|\phi\|^3_{H^1}|\nabla \phi|_{L^2}^{1/2} |A_1 \phi|_{L^2}^{3/2}|A_1^2 \phi|_{L^2} \:\right)\\
                & \leq C_{\psi, \epsilon}\left(\:\|\phi\|_{H^1}^{3/2}|A_1\phi|_{L^2}^{3/2}+ |\nabla \phi|_{L^2}^{3/2}\|\phi\|_{H^4}^{3/2}+ \|\phi\|_{H^1}^{9/2}|A_1\phi|_{L^2}^{3/2} + \|\phi\|_{H^1}^{17/4}|A_1\phi|_{L^2}^{3/8}|A_1^2 \phi|_{L^2}^{11/8}\:\right)\\
				&\leq C_{\psi, \epsilon}\Big(\, \|\phi\|^6_{H^1}+ \|\phi\|^{18}_{H^1} +\|\phi\|^{34}_{H^1} +  |A_1 \phi|^2_{L^2} \Big) + \frac{\epsilon}{16}\|\phi\|^2_{H^4}. 
			\end{aligned}
		\end{equation}
			Similarly, we have
				\begin{equation}
				\begin{aligned}
						I_2 & \leq C_{\psi, \epsilon}\int_{\mathcal{O}}\left(\:|A_1\phi||\phi|+ |A_1\phi||A_1^2\phi| + |\phi|^5|A_1\phi|+  |\phi|^4|A_1\phi|\,|A_1^2\phi| \:\right) \,dx\\
						& \leq C_{\psi, \epsilon}\left(\: |A_1\phi|_{L^2}|\phi|_{L^2} + |A_1\phi|_{L^2}|A_1^2\phi|_{L^2}+ \|\phi\|^5_{L^{10}}|A_1\phi|_{L^2}+  \|\phi\|^4_{L^{16}}\|A_1\phi\|_{L^4}|A_1^2 \phi|_{L^2} \:\right)\\
						& \leq C_{\psi, \epsilon}\left(\: |A_1\phi|_{L^2}^2 +  \|\phi\|^{10}_{H^1} +  \|\phi\|^{17/4}_{H^1}|A_1^{3/2}\phi|_{L^2}^{3/4}|A_1^2 \phi|_{L^2}\:\right)+\frac{\epsilon}{32}\|\phi\|^2_{H^4}\\
						&\leq C_{\psi, \epsilon}\Big(\,  \|\phi\|^{10}_{H^1} +\|\phi\|^{34}_{H^1} +  |A_1\phi|_{L^2}^2 \Big) + \frac{\epsilon}{16}\|\phi\|^2_{H^4}.
				\end{aligned}
			\end{equation}
					For the trilinear term, by invoking $\eqref{b1_1}_1$, \eqref{b1_4}, Lemma \ref{turiq} and Remark \ref{emer} together with Young's inequality, one can derive
					\begin{equation}\label{glob19}
						\begin{aligned}
							|(B_1(v, \phi), \phi + A_1^2 \phi)|&\leq C\|v\|^{1/2}\|A_0v\|^{1/2}_{L^2_{\mathrm{div}}}|A_1^{1/2} \phi|_{L^2}^{3/4}|A_1^{3/2}\phi|_{L^2}^{1/4}|A_1^{2}\phi|_{L^2}\\
							&\leq  C_{\nu, \epsilon} \|v \|^4\|\phi \|^6_{H^1}+\frac{\nu}{10}  \| A_0 v \|^2_{L^2_{\mathrm{div}}} + \frac{\epsilon}{8} \|\phi \|^2_{H^4}.
						\end{aligned}
					\end{equation}
				The boundedness assumption for $h$ leads to
				\begin{align}
				|((P\sigma- A- \alpha u)h(\phi), \phi+A_1^2 \phi)| 
				&\leq C_{\epsilon,A, \alpha, |\mathcal{O}|} \big(1+ |\sigma|^2_{ L^2} + |u|^2_{ L^2}\big)+\frac{\epsilon}{8}\|\phi\|^2_{H^4}.\label{glob11}
			\end{align}
				\textbf{Step 3.} Using the It\^o formula for the processes \(\|\sigma(\cdot)\|^2_{H^1}\), given \eqref{a3}, we derive 
	\begin{align}
		&\mathbb{E} \sup_{t \in [\tau_a, \tau_b]} \|\sigma(t) \|_{H^1}^2 
		+ 2 \mathbb{E} \int_{\tau_a}^{\tau_b} (|A^{1/2}_1\sigma|^2_{ L^2} + |A_1\sigma|^2_{ L^2})ds \leq \mathbb{E} \| \sigma(\tau_a) \|_{H^1}^2  \nonumber \\
		&
		+ 2\mathbb{E} \int_{\tau_a}^{\tau_b}\Bigg(\,|(B_1(v, \sigma),\, \sigma + A_1 \sigma)| + c|(\sigma h(\phi),\,\sigma +A_1 \sigma)|+   b|(\sigma-w,\, \sigma+ A_1 \sigma )|\,  \label{glob7} \\
		& \quad + \frac{1}{2} \| G_2(s, \sigma) \|_{\mathcal{L}_2(U_2, H^1)}^2\Bigg) \, ds  + 2 \mathbb{E} \sup_{t \in [\tau_a, \tau_b]} 
		\left| \sum_{j=1}^\infty \int_{\tau_a}^t (g_{2,j}(s,\sigma), \sigma + A_1\sigma(s)) \, d\beta^2_j(s) \right|\nonumber.
	\end{align}
			Using $\eqref{b1_1}_1$ and \eqref{b1_4} along with Young's inequality, we derive
	\begin{equation}\label{glob12}
		\begin{aligned}
			|(B_1(v, \sigma), \sigma+ A_1 \sigma) |   \leq C\|v\|^{1/2}_{L^2_{\mathrm{div}}}\|v\|^{1/2}|A_1^{1/2} \sigma|^{1/2}_{L^2}\|\sigma\|_{H^2}^{3/2}
			\leq
			C\|v\|^2_{L^2_{\mathrm{div}}}\|v\|^2  |A_1^{1/2}\sigma|^2_{ L^2} + \frac{1}{6}\|\sigma\|^2_{H^2}.
		\end{aligned}
	\end{equation}
			Owing to the boundedness of \(h\) (see [\ref{[A3]}]) and Young's inequality, we find 
	\begin{equation}
			|(\sigma h(\phi),\sigma+A_1 \sigma)| +b|(\sigma-w, \sigma+ A_1 \sigma )| \leq
			C_b( |w|^2_{ L^2} + |\sigma|^2_{ L^2})+ \frac{1}{3}|A_1\sigma|^2_{ L^2}.\label{glob14}
		\end{equation}
			For the noise-driven term, applying the Burkholder–Davis–Gundy inequality and the estimate $\| G_2(t,\sigma)\|_{\mathcal{L}_2(U_2,H^1)}^2 \leq  C_{B_{H^1}}(1+\|\sigma\|^2_{H^1} )$ together with the analysis as in \eqref{sinoise} to get
	\begin{equation}\label{glob15}
		\begin{aligned}
			&	\mathbb{E} \left( \sup_{t \in [\tau_a, \tau_b]} \left| 2 \sum_{j=1}^{\infty} \int_{\tau_a}^{t} \left( g_{2,j}(t,\sigma),  \sigma+ A_1\sigma \right) \, d\beta^2_j \right| \,\right)\\
			& \leq	\mathbb{E} \left( \sup_{t \in [\tau_a, \tau_b]} \left| 2 \sum_{j=1}^{\infty} \int_{\tau_a}^{t} \left( g_{2,j}(t,\sigma), \sigma \right) \, d\beta^2_j \right| \right) +	\mathbb{E} \left( \sup_{t \in [\tau_a, \tau_b]} \left| 2 \sum_{j=1}^{\infty} \int_{\tau_a}^{t} \left(\nabla g_{2,j}(t,\sigma), \nabla \sigma \right) \, d\beta^2_j \right| \right)\\
			&\leq \frac{1}{2} \mathbb{E} \sup_{s \in [\tau_a, \tau_b]} \| \sigma(s) \|^2_{H^1} 
			+ C_{B_{H^1}}\, \mathbb{E} \int_{\tau_a}^{\tau_b} 
			\left( 1+ \| \sigma\|^2_{H^1} \right) ds.
		\end{aligned}
	\end{equation}
		\textbf{Step 4.} By combining \eqref{limiv0}, \eqref{limiphi}, and \eqref{glob7} with the estimates \eqref{glob8}-\eqref{glob21}, \eqref{glob10}-\eqref{glob11} and \eqref{glob12}-\eqref{glob15} for their right-hand side terms, we deduce that
	\begin{eqnarray}\label{glob22}
			\lefteqn{\frac{1}{2}	\mathbb{E} \sup_{t \in [\tau_a, \tau_b]} \| (v(t), \phi(t), \sigma(t)) \|_{\mathcal{V}}^2 
				+  \mathbb{E} \int_{\tau_a}^{\tau_b} \left( \nu \| A_0 v \|^2_{L^2_{\mathrm{div}}} +  \epsilon \left(\,|A_1\phi|^2_{ L^2} + |A^2_1\phi|^2_{ L^2}\right)   
				+  \left(\, |A^{1/2}_1\sigma|^2_{ L^2} + |A_1\sigma|^2_{ L^2}\right) \right)ds}  \nonumber\\
				&\leq& \mathbb{E} \| (v(\tau_a), \phi(\tau_a), \sigma(\tau_a)) \|_{\mathcal{V}}^2 +  C_1\,\mathbb{E} \int_{\tau_a}^{\tau_b} \|v\|^2\big( 1+ \|v\|^2_{L^2_{\mathrm{div}}}\| v \|^2 + \|v \|^2\|\phi \|^6_{H^1}  \big) \,dt\nonumber\\ &&+ C_1\,\mathbb{E} \int_{\tau_a}^{\tau_b} \|\phi\|^2_{H^2}\,dt + C_1\,\mathbb{E} \int_{\tau_a}^{\tau_b} \|\sigma\|^2_{H^1}\Big( 1+  \|v\|^2_{L^2_{\mathrm{div}}}\| v \|^2 \Big) \,dt \nonumber\\
				&&+ C_1\,\mathbb{E} \int_{\tau_a}^{\tau_b} \left(1+ \|z\|^2_{L^2_{\mathrm{div}}} + |u|^2_{L^2} + |w|^2_{L^2} \right)\,dt+ C_1\,\mathbb{E} \int_{\tau_a}^{\tau_b}\Big(\,\|\phi\|^6_{H^1} + \|\phi\|^{10}_{H^1}+ \|\phi\|^{18}_{H^1} +\|\phi\|^{34}_{H^1} \Big) dt\nonumber\\
				&\leq& \mathbb{E} \| (v(\tau_a), \phi(\tau_a), \sigma(\tau_a)) \|_{\mathcal{V}}^2 +  C_1\,\mathbb{E} \int_{\tau_a}^{\tau_b} \|(v, \phi, \sigma)\|^2_{\mathcal{V}}\left( 1+ \|v\|^2_{L^2_{\mathrm{div}}}\| v \|^2 + \|v \|^2\|\phi \|^6_{H^1}\right) dt\\
				&&+ C_1\,\mathbb{E} \int_{\tau_a}^{\tau_b}\left(\,\|\phi\|^6_{H^1} + \|\phi\|^{10}_{H^1}+ \|\phi\|^{18}_{H^1} +\|\phi\|^{34}_{H^1} \right) dt +C_1\,\mathbb{E} \int_{\tau_a}^{\tau_b} \left(1+ \|z\|^2_{L^2_{\mathrm{div}}} + |u|^2_{L^2} + |w|^2_{L^2}\right )dt,\nonumber
	\end{eqnarray} 
			where  \(C_1:= C_1( \nu,\epsilon, |\mathcal{O}|, \,C_{\psi},\eta, A, b, c, \alpha, P, L_{h} ,B_{V}, B_{H^1})\). 
			Define,  
	\begin{equation}\label{glob23}
		\begin{aligned}
			X(t) := C_1\left( \|v\|^2_{L^2_{\mathrm{div}}}\| v \|^2 + \|v \|^2\|\phi \|^6_{H^1}+ \|\phi\|^6_{H^1} + \|\phi\|^{10}_{H^1}+ \|\phi\|^{18}_{H^1} +\|\phi\|^{34}_{H^1}\right).
			\end{aligned}
			\end{equation}
			Since the definition of \(\zeta_M\) implies that
				$\int_{0}^{\zeta_M} X(s)\,ds \, \leq M, \ a.s.,$
                the stochastic Gronwall lemma together with the assumptions on \(z, u, w\) leads to the following:
			\begin{equation}\label{glob25}
				\begin{aligned}
					\mathbb{E} &\sup_{t \in [0, \rho_N \wedge \zeta_M]} \| (v(t), \phi(t), \sigma(t)) \|_{\mathcal{V}}^2 
					+  \mathbb{E} \int_{0}^{\rho_N \wedge \zeta_M} \left(  \| A_0 v\|^2_{L^2_{\mathrm{div}}} +   \|\phi\|^2_{H^4}   
					+ \|\sigma\|^2_{H^2} \right)ds  \\
					&\leq C_2\,\mathbb{E} \| (v_0, \phi_0, \sigma_0) \|_{\mathcal{V}}^2 +  C_2\,\mathbb{E} \int_{0}^{2T}\Big(1+ \|z\|^2_{L^2_{\mathrm{div}}} + |u|^2_{L^2} + |w|^2_{L^2} \Big) \,dt,
				\end{aligned}
			\end{equation}
		where  \( C_2:= C_2(C_1, T, M)\) is independent of \(N\).
		From \eqref{glob5} and \eqref{glob25}, and Markov's inequality, we obtain 
	\begin{equation}\label{glob26}
		\begin{aligned}
				\mathbb{P}(t_N \leq T) 
				&\leq \frac{C_2}{N}\, \mathbb{E}\left( \| (v_0, \phi_0, \sigma_0) \|_{\mathcal{V}}^2 +  \int_{0}^{2T}\big(1+ \|z\|^2_{L^2_{\mathrm{div}}} + |u|^2_{L^2} + |w|^2_{L^2} \big) \,dt \right)+ \mathbb{P}(\zeta_M \leq T).
		\end{aligned}
		\end{equation}
		Therefore, for any fixed \(M\), we have	
	\begin{equation}\label{glob27}
		\lim_{N \to \infty}\mathbb{P}(t_N \leq T) \leq \mathbb{P}(\zeta_M \leq T).
		\end{equation}
		Next, taking into account the energy inequality \eqref{eqequality} for $ p=4,6,10,18,$ and 34, we deduce that
		\begin{equation}\label{glob28} 
			\mathbb{E}\,\int_{0}^{T \wedge \varrho} X(s) ds < \infty.
			\end{equation}
			Finally, using Markov's inequality together with \eqref{glob4} and \eqref{glob28}, we obtain
			\begin{equation*}\label{glob29}
			\mathbb{P}(\zeta_M \leq T) 
			\leq \mathbb{P}\left\{ \int_0^{T \wedge \varrho} X(s)\, ds \geq M \right\}
			\leq \frac{1}{M}\, \mathbb{E} \int_0^{T \wedge \varrho} X(s)\, ds 
			\to 0, \:\text{as}\: M \to \infty.
					\end{equation*}
					Consequently, using \eqref{glob27}, we deduce that	$	\lim_{N \to \infty} \mathbb{P}(t_N \leq T) = 0.$ Hence the proof.
		\end{proof}			
	\section{Appendix}\label{Section 6}
		      We have invoked the following results multiple times in the main results:
        \begin{Lem}(Gagliardo-Nirenberg Inequality )\label{l1}  Let \( \mathcal{O} \subset 		\mathbb{R}^d \), \( d \in \mathbb{N} \), be a bounded domain with Lipschitz boundary, and let \( y \in W^{s,r} \cap L^q \), with \( 1 \leq q, r \leq \infty \). For any integer \( j \), \( 0 \leq j < s \), suppose there exists \( \alpha \in \mathbb{R} \) such that:
		\(
		j - \frac{d}{p} = \alpha \left( s - \frac{d}{r} \right) + (1 - \alpha) \left( 	\frac{-d}{q} 	\right), \ \text{with } j/s \leq \alpha \leq 1.
		\)
		Then, there exists a positive constant \( C \), depending only on \( \mathcal{O}, d, s, j, q, r \), and \( 	\alpha \), such that:
		\[
		\| D^j y \|_{L^p} \leq C \| y \|_{L^q}^{1 - \alpha}\| y \|_{W^{s,r}}^{\alpha}.
		\]
		In particular, we use the following inequalities for \( d = 2, 3 \):
\begin{equation}
 \begin{aligned}
		\|y\|_{L^4} &\leq C \|y\|_{L^2}^{(4-d)/4} 	\|y\|_{H^1}^{d/4},   & \text{if } \ \ d = 2,3, \\
		\|y\|_{L^3} &\leq C \|y\|_{L^2}^{1/2} 	 \|y\|_{H^1}^{1/2},  & \text{if } \ \ d = 3. \label{GNI1}
	\end{aligned}   
\end{equation}
\end{Lem}
		\begin{Lem}(Agmon’s Inequality) \label{l2} Let \( \mathcal{O} \subset \mathbb{R}^d \)  be a bounded Lipschitz domain, and let \( 0 \leq s_1 < d/2 < s_2 \), with \( 0 < \alpha < 1 \) such that
		$
		\frac{d}{2} = \alpha s_1 + (1 - \alpha) s_2.
		$
		Then, there exists a positive constant \(C\) depending only on the measures of the sets and the parameters such that the following  holds:
\begin{equation}
		\|y\|_{L^\infty} \leq C \|y\|_{H^{s_1}}^\alpha 	\|y\|_{H^{s_2}}^{1-\alpha}.
\end{equation}
		In particular:
\begin{align*}
	\|y\|_{L^\infty} &\leq C \|y\|_{L^2}^{1/2} \|y\|_{H^2}^{1/2}, \quad \text{if } \ \ d = 2, \\
	\|y\|_{L^\infty} &\leq C \|y\|_{H^1}^{1/2} \|y\|_{H^2}^{1/2}, \quad\text{if } \  \ d = 3.
\end{align*}
\end{Lem}
		\begin{Lem}(Gronwall Lemma for Stochastic Processes, \cite{Glatt-Holtz2009})\label{l3} Fix \( T > 0 \). Assume that \( X, Y, Z, R: [0, T) \times \Omega \to \mathbb{R} \) are real-valued, non-negative stochastic processes. Let \( \tau < T \) be a stopping time such that
$		\mathbb{E} \int_0^{\tau} (R(s)X(s) + Z(s))\, ds < \infty.$ 
Further, assume that   $\int_0^\tau R(s) 	\, ds <  \, k \ \text{a.s.} $ for some fixed constant $k>0.$ 
		Suppose that for all stopping times \( 0 \leq \tau_a < \tau_b \leq \tau \),
\begin{equation*}
		\mathbb{E} \left[ \sup_{t \in [\tau_a,\tau_b]} X(t) + \int_{\tau_a}^{\tau_b} Y(s)\, ds 	\right] \leq C_0 \mathbb{E} \left[ X(\tau_a) + \int_{\tau_a}^{\tau_b} (R(s)X(s) + Z(s))\, ds \right],
\end{equation*}
		where \(C_0 \) is a constant independent of the choice of \(\tau_a,\tau_b\). Then the following inequality holds:  
\begin{equation*}
		\mathbb{E} \left[ \sup_{t \in [0, \tau]} X(t) + \int_0^{\tau} Y(s)\, ds \right] 
		\leq C \mathbb{E} \left[ X(0) + \int_0^{\tau} Z(s)\, ds \right] 
\end{equation*}
		where  \( C = C(C_0, T,k)\).
\end{Lem}
 \begin{Lem}(see \cite{wehrheim2004uhlenbeck})\label{turiq0} 
		Let $\mathcal{O}$ be a bounded smooth domain in $\mathbb{R}^d$ and $s \in \mathbb{N}$. Then there exists a constant $C_{s,d} > 0$ such that for all $y \in H^{s+2}$ with $\partial_{\textbf{n}} y = 0$ on $\partial \mathcal{O}$, it holds that
		\[
		\|y\|_{H^{s+2}} \leq C_{s,d} \left( |y|_{ L^2} + \|A_1 	y\|_{H^s} \right).
		\]
	\end{Lem}
	\begin{Lem}(see \cite{patnaik2024optimal,PatnaikSakthivel2025}) \label{turiq} Let \(\mathcal{O}\) be a regular bounded subset of \(\mathbb{R}^2\) or \(\mathbb{R}^3\). There exists a constant \(C > 0\), depending on \(\mathcal{O}\), such that for all \(y \in H^2\) with \(\partial_{\textbf{n}}y = 0 \) on \( \partial \mathcal{O}\) , we have
\begin{equation*}
			\|y\|_{L^\infty} \leq C \left( |y|^2_{ L^2} + |A_1 y|^2 \right)^{1/2}, \ \ \ \
			\|A^{1/2}_1 y\|_{L^s} \leq C |A_1 y|_{ L^2} \ \ \forall s \in [2,6], \ \ \ \
			|D^2 y|_{ L^2} \leq C |A_1 y|_{ L^2}.
\end{equation*}
			For any \(y \in H^3\) with  $\partial_{\textbf{n}} y = 0$ on $\partial \mathcal{O}$, the following hold:  
\begin{equation*}
			|A_1 y|_{ L^2} \leq C |A^{3/2}_1 y|_{ L^2}, \hspace{.51in} |D^3 y|_{ L^2} \leq C | A^{3/2}_1y|_{ L^2}, \hspace{.51in}
			\|D^2 y\|_{L^3} \leq C |A^{1/2}_1 y|^{1/2}_{ L^2} |A^{3/2}_1 y|^{1/2}_{ L^2}.
\end{equation*}
\end{Lem}
    \bibliographystyle{abbrv}
 \bibliography{references} 
	
     \end{document}